\documentclass[11pt]{article}
\usepackage{tikz}
\usepackage{ucs}
\usepackage{amssymb}
\usepackage{amsthm}
\usepackage{amsmath}
\usepackage{latexsym}
\usepackage[cp1251]{inputenc}
\usepackage[english]{babel}
\usepackage{graphicx}
\usepackage{wrapfig}
\usepackage{caption}
\usepackage{subcaption}
\usepackage{txfonts}
\usepackage{lmodern}
\usepackage{mathrsfs}
\usepackage{algorithm}
\usepackage{multicol}
\usepackage{enumerate}
\usepackage{titlesec}
\usepackage{paralist}
\usepackage{mathtools}
\usepackage[hidelinks]{hyperref}

\newtheorem{thm}{Theorem}
\newtheorem{lemma}[thm]{Lemma}
\newtheorem{claim}[thm]{Claim}
\newtheorem{conj}[thm]{Conjecture}

\newtheorem{defn}[thm]{Definition}
{}
\newtheorem{theorem}[thm]{Theorem}

\newtheorem*{example*}{Example}

\newtheorem{remark}[thm]{Remark}
\newtheorem*{definition*}{Definition}
\newtheorem*{remark*}{Remark}

\titleformat{\paragraph}
{\normalfont\normalsize\bfseries}{\theparagraph}{1em}{}
\titlespacing*{\paragraph}
{0pt}{3.25ex plus 1ex minus .2ex}{1.5ex plus .2ex}

\newcommand*{\myproofname}{Proof}

\usepackage{wasysym}

\usepackage{fullpage}

\def\qed{\hfill\ifhmode\unskip\nobreak\fi\qquad\ifmmode\Box\else\hfill$\Box$\fi}

\title{Long monochromatic paths and cycles in $2$-edge-colored multipartite graphs}
\date{\today}
\author{
J\' ozsef Balogh~\thanks{Department of Mathematics, University of Illinois at Urbana--Champaign, IL, 
USA, and Moscow Institute of Physics and Technology, Dolgoprodny,  Russian Federation, 
jobal@illinois.edu. 
Research of this author is partially supported by  NSF Grants DMS-1500121, DMS-1764123, Arnold O. Beckman Research Award (UIUC) Campus Research Board 18132 and the Langan Scholar Fund (UIUC).}
\and Alexandr Kostochka \thanks{Department of Mathematics, University of Illinois at Urbana--Champaign, IL, USA and
Sobolev Institute of Mathematics, Novosibirsk 630090, Russia, kostochk@math.uiuc.edu. Research of this author is supported in part by NSF grant
 DMS-1600592, by Award RB17164 of the UIUC Campus Research Board, and by grants 18-01-00353  and 19-01-00682 of the Russian Foundation for Basic Research.}
  \and Mikhail Lavrov\thanks{Department of Mathematics, University of Illinois at Urbana--Champaign, IL, USA,  mlavrov@illinois.edu.}
 \and Xujun Liu\thanks{Department of Mathematics, University of Illinois at Urbana--Champaign, IL, USA,  xliu150@illinois.edu.
 Research of this author
is supported in part by Award RB17164 of the Research Board of the University of Illinois at Urbana-Champaign.
}
 }

\begin{document}
	\maketitle

\begin{abstract}
We solve four similar problems: For  every fixed $s$ and large $n$, we describe all values of $n_1,\ldots,n_s$ such that  for every 
$2$-edge-coloring of the
complete $s$-partite graph $K_{n_1,\ldots,n_s}$ there exists a monochromatic (i) cycle $C_{2n}$ with $2n$ vertices,
(ii) cycle $C_{\geq 2n}$ with at least $2n$ vertices,
(iii) path $P_{2n}$ with $2n$ vertices, and (iv) path $P_{2n+1}$ with $2n+1$ vertices.

This implies a generalization for large $n$ of the conjecture by Gy\' arf\' as,  Ruszink\' o,   S\' ark\H ozy and  Szemer\' edi that for every 
$2$-edge-coloring of the
complete $3$-partite graph $K_{n,n,n}$ there is a monochromatic path $P_{2n+1}$. An important tool is our recent stability theorem on monochromatic connected matchings.
\\
\\
 {\small{\em Mathematics Subject Classification}: 05C15, 05C35, 05C38.}\\
 {\small{\em Key words and phrases}:  Ramsey number,  Regularity Lemma, paths and cycles.}
\end{abstract}

\section{Introduction}
A {\em connected matching} in a graph $G$ is a matching all whose edges are in the same component of $G$.
By $M_n$ we will always denote a connected matching with $n$ edges and by $P_n$ -- the path with $n$ vertices.
Also by $C_n$ we denote the  cycle with $n$ vertices, and by $C_{\geq n}$ -- a cycle of length at least $n$.

For graphs $G_0,\ldots,G_k$ we write $G_0 \mapsto (G_1,\ldots,G_k)$ if for every $k$-coloring of the edges of $G_0$, for some $i\in [k]$ there is a copy of $G_i$ with all edges of color $i$.	The {\em Ramsey number} $R(G_1,\ldots,G_k)$ is the minimum $N$ such that $K_N\mapsto (G_1,\ldots, G_k)$,  and $R_k(G)=R(G_1,\ldots,G_k)$ where $G_1=\ldots=G_k=G$.

Gerencs\' er and Gy\' arf\' as~\cite{GG1} proved in 1967 that the $n$-vertex path $P_n$ satisfies $R_2(P_n)=\left\lfloor \frac{3n-2}{2}\right\rfloor$.
They actually proved a stronger result:

\begin{theorem}[\cite{GG1}]\label{ntt1} For any two positive integers $k\geq \ell$,  $R(P_{k},P_{\ell})=k-1+\left\lfloor \frac{\ell}{2}\right\rfloor$.
\end{theorem}

Many significant results bounding $R_k(P_n)$ for  $k\geq 3$ and $R_k(C_n)$ for even $n$  were proved in~\cite{BLSSW,BS1, DK1,DGKRS1, FS1,FL1,FL2,GRSS1,
GRSS2,KS1,L1,LSS1,S1}.  Many proofs used the Szemer\' edi Regularity Lemma~\cite{Sz} and a number of  them used the idea of connected matchings in regular partitions due to Figaj and  \L uczak~\cite{FL1}. 

Ramsey-type problems when the host graphs are not complete but  complete bipartite were studied by 
Gy\' arf\' as and  Lehel~\cite{GL1}, Faudree and Schelp~\cite{FS2},
DeBiasio,  Gy\' arf\' as,  Krueger,  Ruszink\' o and   S\' ark\H ozy~\cite{GRSS1},
DeBiasio
and  Krueger~\cite{DK1} and Bucic,  Letzter and  Sudakov~\cite{BLS1,BLS2}, and when the host graphs are complete $3$-partite  ---
by Gy\' arf\' as,  Ruszink\' o,   S\' ark\H ozy and  Szemer\' edi~\cite{GRSS0}. The main result in~\cite{FS2} and~\cite{GL1} was

\begin{theorem}[\cite{FS2, GL1}]\label{t21} For every positive integer $n$,  $K_{n,n}\mapsto (P_{2\lceil n/2\rceil},P_{2\lceil n/2\rceil})$.
Furthermore,\\ $K_{n,n}\not\mapsto (P_{2\lceil n/2\rceil+1},P_{2\lceil n/2\rceil+1})$.
\end{theorem}

DeBiasio
and  Krueger~\cite{DK1} extended the result 
 from paths $P_{2\lceil n/2\rceil}$ to  cycles of length at least $2\lfloor n/2\rfloor$ for large $n$.

 The main result in~\cite{GRSS0} was

\begin{theorem}[\cite{GRSS0}]\label{tt1} For every positive integer $n$,  $K_{n,n,n}\mapsto (P_{2n-o(n)},P_{2n-o(n)})$.
\end{theorem}

The following exact bound was also conjectured:

\begin{conj}[\cite{GRSS0}]\label{cc1} For every positive integer $n$, $K_{n,n,n}\mapsto (P_{2n+1},P_{2n+1})$.
\end{conj}	

The goal of this paper is to prove for large $n$ Conjecture~\ref{cc1} and  similar exact bounds for paths $P_{2n}$ (parity matters here) and cycles $C_{2n}$. We do it in a more general 
setting: for multipartite graphs with possibly different part sizes. In the next section, we discuss extremal examples, define some notions and state our main results. In Section~3, we describe our tools. In Sections~\ref{setupM}--\ref{k2n2n-1}, we prove the main part, namely, the result for even cycles $C_{2n}$.
In Sections~\ref{cc2n},~\ref{p2n} and~\ref{p2n+1} we use the main result to derive similar results for cycles $C_{\geq 2n}$, and paths $P_{2n}$ and $P_{2n+1}$.

\section{Examples and results}

For a graph $G$ and disjoint sets $A,B\subset V(G)$, by $G[A]$ we denote the subgraph of $G$ induced by $A$, and by $G[A,B]$ -- the bipartite subgraph of $G$ with parts $A$ and $B$ formed by all edges of $G$ connecting $A$ with $B$.

Our edge-colorings always will be with red (color $1$) and blue (color $2$).

 We consider  necessary restrictions on $n_1\geq n_2\geq \ldots\geq n_s$ providing that each $2$-edge-coloring of
$K_{n_1, n_2, \ldots , n_s}$ contains  (a) a monochromatic   path $P_{2n}$, (b) a monochromatic path 
$P_{2n+1}$, (c) a monochromatic  cycle  $C_{2n}$ and (d) a monochromatic  cycle  $C_{\geq 2n}$. Each condition we add is motivated by an example showing that the condition is necessary.

First, recall that each of $P_{2n}, P_{2n+1}$, $C_{2n}$ and $C_{\geq 2n}$ contains a connected matching $M_n$. Thus a graph with no $M_n$ also contains neither 
$P_{2n}$ nor $ P_{2n+1}$ nor $C_{\geq2n}$.

\subsection{Example with no monochromatic  $M_n$: too few vertices}\label{example-1}
Let  $G=K_{3n-2}$. Clearly, $G\supseteq K_{n_1, n_2, \ldots , n_s}$ for each $n_1,\ldots,n_s$ with $n_1+\ldots+n_s= 3n-2$. Partition $V(G)$
into sets $U_1$ and $U_2$ with $|U_1|=2n-1$ and $|U_2|=n-1$. Color the edges of $G[U_1,U_2]$ with red and the rest of the edges with blue. Since neither $K_{2n-1}$ nor $K_{n-1,2n-1}$ contains $M_{n}$, we conclude $G\not\mapsto (M_{n},M_{n})$; see Figure~\ref{ex1}.

To rule out this example, we add the condition
\begin{equation}\label{jj1}
 N:=n_1+\ldots+n_s\geq 3n-1.
 \end{equation}

\subsection{Example with no monochromatic $M_n$: too few vertices outside $V_1$}\label{example-2}
 Choose any $n_1$ and let $N=n_1+2n-2$. Let  $G$ be obtained from $K_N$ by deleting the edges inside a vertex subset $U_1$ with $|U_1|=n_1$. Graph $G$ contains every $K_{n_1, n_2, \ldots , n_s}$ with $n_2+\ldots+n_s=2n-2$.
 Partition $V(G)-U_1$ into sets $U_2$ and $U_3$ with $|U_2|=|U_3|=n-1$. Color all edges incident with $U_2$ red, and the remaining edges of $G$ blue. Since  the red and blue subgraphs of $G$ have vertex covers of size $n-1$ (namely, $U_2$ and $U_3$), neither of them  contains $M_{n}$.  Thus
$G\not\mapsto (M_{n},M_{n})$; see Figure~\ref{ex2}.

To rule out this example, we add the condition
\begin{equation}\label{jj2}
 N-n_1=n_2+\ldots+n_s\geq 2n-1.
\end{equation}
 
\begin{figure}[ht]\label{f1}
\hspace{5mm}
\begin{minipage}[b]{0.4\textwidth}
\begin{tikzpicture}[scale=0.5, transform shape]

\draw  (-2.5,3.5) [black, fill=blue!10] ellipse (4 and 2);
\draw  (-2.5,-1) [black, fill=blue!10] ellipse (2 and 1);
\node [ shape=circle, minimum size=0.1cm,  fill = black!1000, align=center] (v1) at (-1.95,3.5) {};
\node [ shape=circle, minimum size=0.1cm,  fill = black!1000, align=center] (v2) at (-0.8,3.5) {};
\node [ shape=circle, minimum size=0.1cm,  fill = black!1000, align=center] (v3) at (0.35,3.5) {};
\node [ shape=circle, minimum size=0.1cm,  fill = black!1000, align=center] (v4) at (-3.2,3.5) {};
\node [ shape=circle, minimum size=0.1cm,  fill = black!1000, align=center] (v5) at (-4.35,3.5) {};
\node [ shape=circle, minimum size=0.1cm,  fill = black!1000, align=center] (v6) at (-5.5,3.5) {};
\node [ shape=circle, minimum size=0.1cm,  fill = black!1000, align=center] (v7)  at (-3.5,-1) {};
\node [ shape=circle, minimum size=0.1cm,  fill = black!1000, align=center] (v8)  at (-1.5,-1) {};
\node[ shape=circle, minimum size=0.1cm,  fill = black!1000, align=center] (v10)  at (-2.5,-1) {};
\draw [red, line width = 2]  (v6) edge (v7);
\draw [red, line width = 2] (v6) edge (v10);
\draw [red, line width = 2] (v6) edge (v8);
\draw[red, line width = 2]  (v5) edge (v7);
\draw [red, line width = 2] (v5) edge (v10);
\draw  [red, line width = 2](v5) edge (v8);
\draw [red, line width = 2] (v4) edge (v7);
\draw [red, line width = 2] (v4) edge (v10);
\draw [red, line width = 2] (v4) edge (v8);

\draw  [red, line width = 2](v1) edge (v7);
\draw  [red, line width = 2](v1) edge (v10);
\draw [red, line width = 2] (v1) edge (v8);
\draw [red, line width = 2] (v2) edge (v7);
\draw [red, line width = 2] (v2) edge (v10);
\draw [red, line width = 2] (v2) edge (v8);
\draw [red, line width = 2] (v3) edge (v7);
\draw [red, line width = 2] (v3) edge (v10);
\draw [red, line width = 2] (v3) edge (v8);

\node at (-8.5,3.5) {\LARGE{$|U_1|=2n-1$}};
\node at (-8.5,-1) {\LARGE{$|U_2|=n-1$}};

\end{tikzpicture}

\caption{Example~\ref{example-1}.}\label{ex1}
\end{minipage}
\begin{minipage}[b]{0.5\textwidth}
\begin{tikzpicture}[scale=0.5, transform shape]

\draw  (-2.5,3.5) ellipse (4 and 2);
\draw [black, fill=red!80, rotate = 45] (-6,2.5) ellipse (1 and 2);
\draw [black, fill=blue!10, rotate = -45] (2.5,-1) ellipse (1 and 2);

\node [ shape=circle, minimum size=0.1cm,  fill = black!1000, align=center] (v1) at (-1.9,3.5) {};
\node [ shape=circle, minimum size=0.1cm,  fill = black!1000, align=center] (v2) at (-0.7,3.5) {};
\node [ shape=circle, minimum size=0.1cm,  fill = black!1000, align=center] (v3) at (0.5,3.5) {};
\node [ shape=circle, minimum size=0.1cm,  fill = black!1000, align=center] (v4) at (-3.25,3.5) {};
\node [ shape=circle, minimum size=0.1cm,  fill = black!1000, align=center] (v5) at (-4.35,3.5) {};
\node [ shape=circle, minimum size=0.1cm,  fill = black!1000, align=center] (v6) at (-5.5,3.5) {};
\node [ shape=circle, minimum size=0.1cm,  fill = black!1000, align=center] (v7) at (-5,-3.5) {};
\node [ shape=circle, minimum size=0.1cm,  fill = black!1000, align=center] (v8) at (-7,-1.5) {};
\node[ shape=circle, minimum size=0.1cm,  fill = black!1000, align=center] (v10) at (-6,-2.5) {};

\node [ shape=circle, minimum size=0.1cm,  fill = black!1000, align=center] (v11) at (0,-3.5) {};
\node [ shape=circle, minimum size=0.1cm,  fill = black!1000, align=center] (v12) at (2,-1.5) {};
\node[ shape=circle, minimum size=0.1cm,  fill = black!1000, align=center] (v13) at (1,-2.5) {};

\draw [blue!30, line width = 2]  (v6) edge (v11);
\draw [blue!30, line width = 2] (v6) edge (v13);
\draw [blue!30, line width = 2] (v6) edge (v12);
\draw[blue!30, line width = 2]  (v5) edge (v11);
\draw [blue!30, line width = 2] (v5) edge (v13);
\draw  [blue!30, line width = 2](v5) edge (v12);
\draw [blue!30, line width = 2] (v4) edge (v11);
\draw [blue!30, line width = 2] (v4) edge (v13);
\draw [blue!30, line width = 2] (v4) edge (v12);

\draw  [blue!30, line width = 2](v1) edge (v11);
\draw  [blue!30, line width = 2](v1) edge (v13);
\draw [blue!30, line width = 2] (v1) edge (v12);
\draw [blue!30, line width = 2] (v2) edge (v11);
\draw [blue!30, line width = 2] (v2) edge (v13);
\draw [blue!30, line width = 2] (v2) edge (v12);
\draw [blue!30, line width = 2] (v3) edge (v11);
\draw [blue!30, line width = 2] (v3) edge (v13);
\draw [blue!30, line width = 2] (v3) edge (v12);

\draw [red!100, line width = 2]  (v6) edge (v7);
\draw [red!100, line width = 2] (v6) edge (v10);
\draw [red!100, line width = 2] (v6) edge (v8);
\draw[red!100, line width = 2]  (v5) edge (v7);
\draw [red!100, line width = 2] (v5) edge (v10);
\draw  [red!100, line width = 2](v5) edge (v8);
\draw [red!100, line width = 2] (v4) edge (v7);
\draw [red!100, line width = 2] (v4) edge (v10);
\draw [red!100, line width = 2] (v4) edge (v8);

\draw  [red!100, line width = 2](v1) edge (v7);
\draw  [red!100, line width = 2](v1) edge (v10);
\draw [red!100, line width = 2] (v1) edge (v8);
\draw [red!100, line width = 2] (v2) edge (v7);
\draw [red!100, line width = 2] (v2) edge (v10);
\draw [red!100, line width = 2] (v2) edge (v8);
\draw [red!100, line width = 2] (v3) edge (v7);
\draw [red!100, line width = 2] (v3) edge (v10);
\draw [red!100, line width = 2] (v3) edge (v8);

\node at (4,3.5) {\LARGE{$|U_1|=n_1$}};
\node at (-9.5,-3) {\LARGE{$|U_2|=n-1$}};
\node at (4,-3) {\LARGE{$|U_3|=n-1$}};

\draw [red!100, line width = 2]  (v8) edge (v12);
\draw   [red!100, line width = 2](v8) edge (v13);
\draw [red!100, line width = 2]  (v8) edge (v11);
\draw  [red!100, line width = 2] (v10) edge (v12);
\draw  [red!100, line width = 2] (v10) edge (v13);
\draw  [red!100, line width = 2] (v10) edge (v11);
\draw  [red!100, line width = 2] (v7) edge (v12);
\draw  [red!100, line width = 2] (v7) edge (v13);
\draw  [red!100, line width = 2] (v7) edge (v11);

\end{tikzpicture}

\caption{Example~\ref{example-2}.}\label{ex2}

\end{minipage}
\end{figure}

\subsection{Example with no red $M_{n}$ and no blue $P_{2n+1}$: too few vertices}\label{example-3}
Let $G=K_{3n-1}$. Partition $V(G)$
into sets $U_1$ and $U_2$ with $|U_1|=2n$ and $|U_2|=n-1$. Color the edges of $G[U_1,U_2]$ with red and the rest of the edges with blue. Since the red subgraph of $G$ has vertex cover $U_2$ with $|U_2|=n-1$, it does not contain $M_{n}$. Since each component of the blue subgraph of $G$ has fewer than $2n+1$ vertices, it does not contain $P_{2n+1}$.

Therefore 
\[
 R(P_{2n},P_{2n+1})\geq R(M_{n},P_{2n+1})\geq 3n,
\]
which yields for $P_{2n+1}$ the following strengthening of~(\ref{jj1}):
\begin{equation}\label{jj31}
	 \mbox{for $P_{2n+1}$,}\quad N\geq 3n.
\end{equation}

\subsection{Example with no monochromatic  $C_{\geq 2n}$ when $N-n_1 -n_2 \le 2$} \label{example-4}
This example, and all the ones that follow, show that additional restrictions are necessary when $G$ is bipartite or close to bipartite.

Let $G=K_{n_1,\ldots,n_s}$ satisfy~(\ref{jj1}) and~(\ref{jj2}) with
$N-n_1-n_2\leq 2$ such that $n_1\leq 2n-2$. Then also $n_2\leq 2n-2$, so  
$G\subseteq K_{2n-2,2n-2,1,1}$. Thus we assume $G= K_{2n-2,2n-2,1,1}$ with
$V_1=\{v_1,\ldots,v_{2n-2}\}$, $V_2=\{u_1,\ldots,u_{2n-2}\}$, $V_3=\{x\}$ and $V_4=\{y\}$.
 Let $V'_1=\{v_1,\ldots,v_{n-1}\}$, $V''_1=V_1-V'_1$, $V'_2=\{u_1,\ldots,u_{n-1}\}$, $V''_2=V_2-V'_2$.
Color the edges in $G[V_1',V'_2]$, $G[V_1'',V''_2]$ and in $G[V_3,V_1\cup V_2\cup V_4]$ with red, and all other edges
with blue. Then the red graph $G_1$ has cut vertex $x$, and the components of $G_1-x$ have sizes $2n-2$, $2n-2$ and $1$,
so $G_1$ has no $C_{\ge 2n}$. Similarly, $G_2$ contains no $C_{\ge 2n}$; see Figure~\ref{ex4}.

To rule out this example, we add the condition
\begin{equation}\label{jf1}
 \mbox{\em For $C_{\geq 2n}$, if $N-n_1-n_2\leq 2$, then }\; n_1\geq 2n-1.
\end{equation}

\subsection{Example with no monochromatic  $C_{\geq 2n}$ when $N-n_1 -n_2 \le 1$} \label{example-5}

Let $G=K_{n_1,\ldots,n_s}$ satisfying~(\ref{jj1}),~(\ref{jj2}) and~(\ref{jf1})  with
$N-n_1-n_2\leq 1$ such that $N+n_1\leq 6n-3$. Since by~(\ref{jf1}),  $n_1\geq 2n-1$, we get 
$N-n_1\leq (6n-3)-2(2n-1)=2n-1$, but (\ref{jj2}) implies $N-n_1\ge 2n-1$; therefore both inequalities are tight and $N-n_1 = n_1 = 2n-1$. Hence 
$G\subseteq K_{2n-1,2n-2,1}$, which is a subgraph of the graph $K_{2n-2,2n-2,1,1}$ considered in Example~\ref{example-4}.

This example is not ruled out by \eqref{jf1}, so we add the condition
\begin{equation}\label{jm1}
 \mbox{\em For $C_{\geq 2n}$, if $N-n_1-n_2\leq 1$, then }\; n_1+N\geq 6n-2.
\end{equation}

\subsection{Example with no monochromatic $P_{2n+1}$ when $G$ is bipartite}\label{example-6}

Suppose $n_3=0$ and $n_1\leq 2n$. Then $n_2\leq 2n$ as well, so $G\subseteq K_{2n,2n}$. Thus we assume $G= K_{2n,2n}$ with
$V_1=\{v_1,\ldots,v_{2n}\}$ and $V_2=\{u_1,\ldots,u_{2n}\}$.
 Let $V'_1=\{v_1,\ldots,v_{n}\}$, $V''_1=V_1-V'_1$, $V'_2=\{u_1,\ldots,u_{n}\}$, $V''_2=V_2-V'_2$.
Color the edges in $G(V_1',V'_2)$ and  $G(V_1'',V''_2)$  with red, and all other edges
with blue. Then each component in the red graph and each component in the blue graph has  $2n$
vertices and thus does not contain $P_{2n+1}$; see Figure~\ref{ex5}.

To rule out this example, we add the condition
\begin{equation}\label{jf2}
 \mbox{\em For $P_{2n+1}$, if $n_3=0$, then }\; n_1\geq 2n+1.
 \end{equation}

\subsection{Example with no monochromatic  $C_{2n}$ when $N-n_1 -n_2 \le 2$} \label{example-7}

Let $G=K_{n_1,\ldots,n_s}$ satisfying~(\ref{jj1}),~(\ref{jj2}) and~(\ref{jf1})  with
$N-n_1-n_2=2$ such that $N\leq 4n-2$. By~(\ref{jf1}), $N-n_1\leq 2n-1$.
Now~(\ref{jj2}) implies $N-n_1=2n-1=n_1$. Hence $G\subseteq K_{2n-1,2n-3,1,1}$. 
Thus we assume $G= K_{2n-1,2n-3,1,1}$ with
$V_1=\{v_1,\ldots,v_{2n-1}\}$, $V_2=\{u_1,\ldots,u_{2n-3}\}$, $V_3=\{x\}$ and $V_4=\{y\}$.
 Define $A=\{v_2,v_3,\ldots,v_n\} $,  $B=\{v_{n+1},v_{n+2},\ldots,v_{2n-1}\} $, $C=\{u_1,u_2,\ldots,u_{n-1}\} $ and
$D=\{u_n,u_{n+1},\ldots,u_{2n-3}\} $.  We assign the colors to the edges of $G$ as follows. 

\begin{enumerate}
\item $G[A,C]$ and $G[B,D]$ are complete bipartite red graphs.

\item $G[A,D]$ and $G[B,C]$ are complete bipartite blue graphs.

\item $v_1$ has all blue edges to $V_2$.

\item $x$ has all red edges to $V_1 \cup V_2\cup \{y\}$.

\item $y$ has all red edges to $B \cup D\cup \{x\}$ and all blue edges to $A \cup C \cup \{v_1\}$.
\end{enumerate}

We claim that $G$ has no monochromatic cycle of length $2n$. Indeed, consider first the red graph $G_1$.
 The graph $G_1-x$  has three components: a) $A \cup C$ of size $2n-2$, b) $\{v_1\}$ of size $1$,  and c) $B \cup D \cup \{y\}$ of size $2n-2$. 
 Thus $G$ has  no red cycle of length $2n$ since the largest block of $G_1$  has order $2n-1$. 

  Consider now the blue graph $G_2$. We ignore $x$ since it is isolated.
 Suppose $G_2$ contains a $2n$-cycle $F$.  Since $v_1$ is a cut vertex of $G_2 - \{y\}$ with the components of $G_2 - \{y, v_1\}$  of  order  $2n-3$ and $2n-2$, $F$ contains $y$. 

 If we delete from $G_2$ all edges in $G_2[\{y\},C]$, then  the blocks in the remaining blue graph will be of order $2n-1$ and $2n-1$; thus
 $F$ contains an edge from $y$ to $C$, say $yz$.  Furthermore, if $yz$ is the only edge in $F$ connecting $y$ to $C$, then all other edges in
 $F$ belong to the bipartite graph
 $H=G_2[A \cup B \cup \{v_1\}, D \cup \{y\} \cup C]$. But this bipartite graph $H$ cannot have a path of odd length $2n-1$ between the vertices $y$ and $z$  in the same part.

Thus,  $F$ has to use two edges from $y$ to $C$, say $yz_1$ and $yz_2$. Then the problem is reduced to finding a blue path from $z_1$ to $z_2$ of length $2n-2$ in $G_2[C, B \cup \{v_1\}]$. However, it is impossible because $|C| = n-1$ and the longest path from $z_1$ to $z_2$ in $G_2[C, B\cup \{v_1\}]$  has $2n-3$ vertices.

Note that this example has cycles of length greater than $2n-1$, but all such cycles are odd.

\vspace{5mm}

\begin{figure}[ht]\label{f2}
\hspace{5mm}
\begin{minipage}[b]{0.5\textwidth}
\begin{tikzpicture}[scale=0.4, transform shape]
\draw  (-3.5,3.5) ellipse (6.5 and 2.5);

\draw  (-3.5,-3.5) ellipse (6.5 and 2.5);
\draw  (-6.5,3.5) ellipse (2.8 and 1.5);
\draw  (-0.5,3.5) ellipse (2.8 and 1.5);

\draw  (-6.5,-3.5) ellipse (2.8 and 1.5);
\draw  (-0.5,-3.5) ellipse (2.8 and 1.5);

\node [ shape=circle, minimum size=0.1cm,  fill = black!1000, align=center] (v1) at (-12,1) {};
\node [ shape=circle, minimum size=0.1cm,  fill = black!1000, align=center] (v2) at (-12,-1.5) {};

\node [ shape=circle, minimum size=0.1cm,  fill = black!1000, align=center] (v4) at (-5,3.5) {};
\node [ shape=circle, minimum size=0.1cm,  fill = black!1000, align=center] (v5) at (-6.5,3.5) {};
\node [ shape=circle, minimum size=0.1cm,  fill = black!1000, align=center] (v6) at (-8,3.5) {};

\node [ shape=circle, minimum size=0.1cm,  fill = black!1000, align=center] (v10) at (1,3.5) {};
\node [ shape=circle, minimum size=0.1cm,  fill = black!1000, align=center] (v11) at (-0.5,3.5) {};
\node [ shape=circle, minimum size=0.1cm,  fill = black!1000, align=center] (v12) at (-2,3.5) {};

\node [ shape=circle, minimum size=0.1cm,  fill = black!1000, align=center] (v23) at (-5,-3.5) {};
\node [ shape=circle, minimum size=0.1cm,  fill = black!1000, align=center] (v24) at (-6.5,-3.5) {};
\node [ shape=circle, minimum size=0.1cm,  fill = black!1000, align=center] (v25) at (-8,-3.5) {};

\node [ shape=circle, minimum size=0.1cm,  fill = black!1000, align=center] (v28) at (1,-3.5) {};
\node [ shape=circle, minimum size=0.1cm,  fill = black!1000, align=center] (v29) at (-0.5,-3.5) {};
\node [ shape=circle, minimum size=0.1cm,  fill = black!1000, align=center] (v30) at (-2,-3.5) {};

\node at (-12.5,3.5) {\huge{$|V_1|=2n-2$}};
\node at (-12.5,-4) {\huge{$|V_2|=2n-2$}};

\node at (-14,-1.5) {\huge{$V_4 = \{y\}$}};
\node at (-14,1) {\huge{$V_3 = \{x\}$}};

\node at (-0.5,-4.35) {\large{$|V''_2|=n-1$}};
\node at (-6.5,-4.35) {\large{$|V'_2|=n-1$}};
\node at (-0.5,4.35) {\large{$|V''_1|=n-1$}};
\node at (-6.5,4.35) {\large{$|V'_1|=n-1$}};

\draw [red, line width = 2] (v6) edge (v25);
\draw  [red, line width = 2](v6) edge (v24);
\draw[red, line width = 2]  (v6) edge (v23);

\draw [red, line width = 2] (v5) edge (v25);
\draw[red, line width = 2]  (v5) edge (v24);
\draw [red, line width = 2] (v23) edge (v5);

\draw [red, line width = 2] (v4) edge (v25);
\draw [red, line width = 2] (v4) edge (v24);
\draw [red, line width = 2] (v4) edge (v23);

\draw  [red, line width = 2](v12) edge (v30);
\draw  [red, line width = 2](v12) edge (v29);
\draw [red, line width = 2] (v12) edge (v28);

\draw [red, line width = 2] (v11) edge (v30);
\draw [red, line width = 2] (v11) edge (v29);
\draw [red, line width = 2] (v11) edge (v28);

\draw  [red, line width = 2](v10) edge (v30);
\draw [red, line width = 2] (v10) edge (v29);
\draw [red, line width = 2] (v10) edge (v28);

\draw [red, line width = 2] (v1) edge (v10);
\draw [red, line width = 2] (v1) edge (v11);
\draw [red, line width = 2] (v1) edge (v12);

\draw [red, line width = 2] (v1) edge (v4);
\draw [red, line width = 2] (v1) edge (v5);
\draw[red, line width = 2]  (v1) edge (v6);

\draw [red, line width = 2] (v1) edge (v28);
\draw[red, line width = 2]  (v1) edge (v29);
\draw [red, line width = 2] (v1) edge (v30);

\draw [red, line width = 2] (v1) edge (v23);
\draw [red, line width = 2] (v1) edge (v24);
\draw [red, line width = 2] (v1) edge (v25);
\draw  [red, line width = 2]  (v1) edge (v2);

\draw  [blue!20, line width = 2] (v6) edge (v30);
\draw  [blue!20, line width = 2](v6) edge (v29);
\draw  [blue!20, line width = 2](v6) edge (v28);

\draw  [blue!20, line width = 2](v5) edge (v30);
\draw  [blue!20, line width = 2](v29) edge (v5);
\draw [blue!20, line width = 2] (v5) edge (v28);

\draw [blue!20, line width = 2] (v4) edge (v30);
\draw [blue!20, line width = 2] (v4) edge (v29);
\draw [blue!20, line width = 2] (v4) edge (v28);

\draw [blue!20, line width = 2] (v25) edge (v12);
\draw[blue!20, line width = 2]  (v12) edge (v24);
\draw  [blue!20, line width = 2](v12) edge (v23);

\draw[blue!20, line width = 2]  (v11) edge (v25);
\draw  [blue!20, line width = 2](v11) edge (v24);
\draw[blue!20, line width = 2]  (v11) edge (v23);

\draw [blue!20, line width = 2] (v10) edge (v25);
\draw[blue!20, line width = 2]  (v10) edge (v24);
\draw[blue!20, line width = 2]  (v10) edge (v23);

\draw  [blue!20, line width = 2] (v2) edge (v10);
\draw  [blue!20, line width = 2] (v2) edge (v11);
\draw  [blue!20, line width = 2] (v2) edge (v12);

\draw   [blue!20, line width = 2](v2) edge (v4);
\draw  [blue!20, line width = 2] (v2) edge (v5);
\draw  [blue!20, line width = 2] (v2) edge (v6);

\draw   [blue!20, line width = 2](v2) edge (v28);
\draw  [blue!20, line width = 2] (v2) edge (v29);
\draw   [blue!20, line width = 2](v2) edge (v30);

\draw  [blue!20, line width = 2] (v2) edge (v23);
\draw  [blue!20, line width = 2] (v2) edge (v24);
\draw   [blue!20, line width = 2](v2) edge (v25);
\end{tikzpicture}
 
\caption{Example~\ref{example-4}.}\label{ex4}
\end{minipage}
\begin{minipage}[b]{0.3\textwidth}
\begin{tikzpicture}[scale=0.4, transform shape]
\draw  (-3.5,3.5) ellipse (6.5 and 2.5);

\draw  (-3.5,-3.5) ellipse (6.5 and 2.5);
\draw  (-6.5,3.5) ellipse (2.8 and 1.5);
\draw  (-0.5,3.5) ellipse (2.8 and 1.5);

\draw  (-6.5,-3.5) ellipse (2.8 and 1.5);
\draw  (-0.5,-3.5) ellipse (2.8 and 1.5);

\node [ shape=circle, minimum size=0.1cm,  fill = black!1000, align=center] (v3) at (-4.5,3.5) {};
\node [ shape=circle, minimum size=0.1cm,  fill = black!1000, align=center] (v4) at (-5.83,3.5) {};
\node [ shape=circle, minimum size=0.1cm,  fill = black!1000, align=center] (v5) at (-7.167,3.5) {};
\node [ shape=circle, minimum size=0.1cm,  fill = black!1000, align=center] (v6) at (-8.5,3.5) {};

\node [ shape=circle, minimum size=0.1cm,  fill = black!1000, align=center] (v9) at (1.5,3.5) {};
\node [ shape=circle, minimum size=0.1cm,  fill = black!1000, align=center] (v10) at (0.167,3.5) {};
\node [ shape=circle, minimum size=0.1cm,  fill = black!1000, align=center] (v11) at (-1.167,3.5) {};
\node [ shape=circle, minimum size=0.1cm,  fill = black!1000, align=center] (v12) at (-2.5,3.5) {};

\node [ shape=circle, minimum size=0.1cm,  fill = black!1000, align=center] (v22) at (-4.5,-3.5) {};
\node [ shape=circle, minimum size=0.1cm,  fill = black!1000, align=center] (v23) at (-5.83,-3.5) {};
\node [ shape=circle, minimum size=0.1cm,  fill = black!1000, align=center] (v24) at (-7.167,-3.5) {};
\node [ shape=circle, minimum size=0.1cm,  fill = black!1000, align=center] (v25) at (-8.5,-3.5) {};

\node [ shape=circle, minimum size=0.1cm,  fill = black!1000, align=center] (v27) at (1.5,-3.5) {};
\node [ shape=circle, minimum size=0.1cm,  fill = black!1000, align=center] (v28) at (0.167,-3.5) {};
\node [ shape=circle, minimum size=0.1cm,  fill = black!1000, align=center] (v29) at (-1.167,-3.5) {};
\node [ shape=circle, minimum size=0.1cm,  fill = black!1000, align=center] (v30) at (-2.5,-3.5) {};

\node at (5,3.5) {\huge{$|V_1|=2n$}};
\node at (5,-4) {\huge{$|V_2|=2n$}};

\node at (-0.5,-4.35) {\large{$|V''_2|=n$}};
\node at (-6.5,-4.35) {\large{$|V'_2|=n$}};
\node at (-0.5,4.35) {\large{$|V''_1|=n$}};
\node at (-6.5,4.35) {\large{$|V'_1|=n$}};

\draw [red, line width = 2] (v6) edge (v25);
\draw  [red, line width = 2](v6) edge (v24);
\draw[red, line width = 2]  (v6) edge (v23);
\draw [red, line width = 2] (v6) edge (v22);

\draw [red, line width = 2] (v5) edge (v25);
\draw[red, line width = 2]  (v5) edge (v24);
\draw [red, line width = 2] (v23) edge (v5);
\draw [red, line width = 2] (v5) edge (v22);

\draw [red, line width = 2] (v4) edge (v25);
\draw [red, line width = 2] (v4) edge (v24);
\draw [red, line width = 2] (v4) edge (v23);
\draw [red, line width = 2] (v4) edge (v22);

\draw [red, line width = 2] (v3) edge (v25);
\draw [red, line width = 2] (v3) edge (v24);
\draw [red, line width = 2] (v3) edge (v23);
\draw [red, line width = 2] (v3) edge (v22);

\draw  [red, line width = 2](v12) edge (v30);
\draw  [red, line width = 2](v12) edge (v29);
\draw [red, line width = 2] (v12) edge (v28);
\draw[red, line width = 2]  (v12) edge (v27);

\draw [red, line width = 2] (v11) edge (v30);
\draw [red, line width = 2] (v11) edge (v29);
\draw [red, line width = 2] (v11) edge (v28);
\draw [red, line width = 2] (v11) edge (v27);

\draw  [red, line width = 2](v10) edge (v30);
\draw [red, line width = 2] (v10) edge (v29);
\draw [red, line width = 2] (v10) edge (v28);
\draw [red, line width = 2] (v10) edge (v27);

\draw [red, line width = 2] (v9) edge (v30);
\draw [red, line width = 2] (v9) edge (v29);
\draw  [red, line width = 2](v9) edge (v28);
\draw [red, line width = 2] (v9) edge (v27);

\draw  [blue!20, line width = 2] (v6) edge (v30);
\draw  [blue!20, line width = 2](v6) edge (v29);
\draw  [blue!20, line width = 2](v6) edge (v28);
\draw  [blue!20, line width = 2](v6) edge (v27);

\draw  [blue!20, line width = 2](v5) edge (v30);
\draw  [blue!20, line width = 2](v29) edge (v5);
\draw [blue!20, line width = 2] (v5) edge (v28);
\draw[blue!20, line width = 2]  (v5) edge (v27);

\draw [blue!20, line width = 2] (v4) edge (v30);
\draw [blue!20, line width = 2] (v4) edge (v29);
\draw [blue!20, line width = 2] (v4) edge (v28);
\draw[blue!20, line width = 2]  (v4) edge (v27);

\draw [blue!20, line width = 2] (v3) edge (v30);
\draw [blue!20, line width = 2] (v3) edge (v29);
\draw [blue!20, line width = 2] (v3) edge (v28);
\draw[blue!20, line width = 2]  (v3) edge (v27);

\draw [blue!20, line width = 2] (v25) edge (v12);
\draw[blue!20, line width = 2]  (v12) edge (v24);
\draw  [blue!20, line width = 2](v12) edge (v23);
\draw[blue!20, line width = 2]  (v12) edge (v22);

\draw[blue!20, line width = 2]  (v11) edge (v25);
\draw  [blue!20, line width = 2](v11) edge (v24);
\draw[blue!20, line width = 2]  (v11) edge (v23);
\draw  [blue!20, line width = 2](v11) edge (v22);

\draw [blue!20, line width = 2] (v10) edge (v25);
\draw[blue!20, line width = 2]  (v10) edge (v24);
\draw[blue!20, line width = 2]  (v10) edge (v23);
\draw [blue!20, line width = 2] (v10) edge (v22);

\draw [blue!20, line width = 2] (v9) edge (v25);
\draw[blue!20, line width = 2]  (v9) edge (v24);
\draw [blue!20, line width = 2] (v9) edge (v23);
\draw[blue!20, line width = 2]  (v9) edge (v22);
\end{tikzpicture}

\caption{Example~\ref{example-5}.}\label{ex5}

\end{minipage}
\end{figure} 

\vspace{5mm}

To rule out this example, we add the condition
\begin{equation}\label{jm2}
 \mbox{\em For $C_{ 2n}$, if $N-n_1-n_2\leq 2$, then }\; N\geq 4n-1.
\end{equation}
 
\subsection{Results}

Our key result is that for large $n$, the necessary conditions~(\ref{jj1}),~(\ref{jj2}) and~(\ref{jm2})  for the presence in
a $2$-edge-colored $K_{n_1,\ldots,n_s}$ of a monochromatic $C_{ 2n}$, together are  also sufficient for this.

\begin{theorem}\label{tC2n} Let  $s\geq 2$ and $n$ be sufficiently large. Let $n_1\geq\ldots\geq n_s$ and $N=n_1+\ldots+n_s$
satisfy~(\ref{jj1}),~(\ref{jj2}) and~(\ref{jm2}). Then for each $2$-edge-coloring $f$ of the complete $s$-partite graph
$K_{n_1,\ldots,n_s}$, there exists a monochromatic cycle $C_{2n}$.
\end{theorem}

Based on Theorem~\ref{tC2n}, we derive our other results. The first of them is on cycles of length at least $2n$ 
(it extends a result
by DeBiasio
and  Krueger~\cite{DK1}). Recall that~(\ref{jm2}) is not necessary for the existence of a monochromatic $C_{\geq 2n}$, 
but~(\ref{jj1}),~(\ref{jj2}),~(\ref{jf1}) and~(\ref{jm1}) are.

\begin{theorem}\label{tCgeq2n} Let  $s\geq 2$ and $n$ be sufficiently large. Let $n_1\geq\ldots\geq n_s$ and $N=n_1+\ldots+n_s$
satisfy~(\ref{jj1}),~(\ref{jj2}),~(\ref{jf1}) and~(\ref{jm1}). Then for each $2$-edge-coloring $f$ of the complete $s$-partite graph
$K_{n_1,\ldots,n_s}$, there exists a monochromatic cycle $C_{\geq 2n}$.
\end{theorem}

The results for paths of even and odd length are somewhat different.
  The first of them shows that for large $n$, the necessary conditions~(\ref{jj1}) and~(\ref{jj2})  for the presence in
a $2$-edge-colored $K_{n_1,\ldots,n_s}$ of a monochromatic  connected matching $M_{n}$, together are  sufficient for the presence of
the monochromatic path $P_{ 2n}$.

\begin{theorem}\label{tP2n} Let  $s\geq 2$ and $n$ be sufficiently large. Let $n_1\geq\ldots\geq n_s$ and $N=n_1+\ldots+n_s$
satisfy~(\ref{jj1}) and~(\ref{jj2}). Then for each $2$-edge-coloring $f$ of the complete $s$-partite graph
$K_{n_1,\ldots,n_s}$, there exists a monochromatic path $P_{2n}$.
\end{theorem}

Our last  result  implies Conjecture~\ref{cc1}:

\begin{theorem}\label{tP2n+1} Let  $s\geq 2$ and $n$ be sufficiently large. Let $n_1\geq\ldots\geq n_s$ and $N=n_1+\ldots+n_s$
satisfy~(\ref{jj2}),~(\ref{jj31}) and~(\ref{jf2}). Then for each $2$-edge-coloring $f$ of the complete $s$-partite graph
$K_{n_1,\ldots,n_s}$, there exists a monochromatic path $P_{2n+1}$.
\end{theorem}

In the next section, we describe our main tools: the Szemer\'edi Regularity Lemma, connected matchings and theorems on the existence of Hamiltonian cycles in dense graphs. In Section~\ref{setupM} we set up and describe the structure of the proof of
Theorem~\ref{tC2n}, and in the next four sections we present this proof. In the last three sections we prove 
Theorems~\ref{tCgeq2n},~\ref{tP2n} and~\ref{tP2n+1}.

\section{Tools}
As in many recent papers on Ramsey numbers of paths (see~\cite{BLSSW,BS1, DK1,FL1,GRSS1,
GRSS2,KS1,LSS1,S1} and some references in them), our proof heavily uses the Szemer\'edi Regularity Lemma~\cite{Sz}
and  the idea of connected matchings in regular partitions of reduced graphs due to Figaj and  \L uczak~\cite{FL1}.

\subsection{Regularity}

We say that a pair $(V_1,V_2)$ of two disjoint vertex sets $V_1, V_2 \subseteq V(G)$ is $(\epsilon,G)${\em -regular} if 
\[
	\left|\frac{|E(X,Y)|}{|X||Y|} - \frac{|E(V_1,V_2)|}{|V_1||V_2|}\right| < \epsilon
\]
for all $X \subseteq V_1$ and $Y \subseteq V_2$ with $|X| > \epsilon |V_1|$ and $|Y| > \epsilon |V_2|$.

We use a $2$-color version of the Regularity Lemma, following Gy\'arf\'as, Ruszink\'o, S\'ark\"ozy, and Szemer\'edi~\cite{GRSS1}.

\begin{lemma}[$2$-color version of the Szemer\'edi Regularity Lemma]\label{regularity-lemma}
For every $\epsilon>0$ and integer $m>0$, there are positive integers $M$ and $n_0$ such that for $n \ge n_0$ the following holds. For all graphs $G_1$ and $G_2$ with $V(G_1) = V(G_2) = V$, $|V|=n$, there is a partition of $V$ into $L+1$ disjoint classes (clusters) $(V_0, V_1, V_2, \ldots, V_L)$ such that
\begin{itemize}
	\item $m \le L \le M$,
	\item $|V_1| = |V_2| = \ldots = |V_L|$,
	\item $|V_0| < \epsilon n$,
	\item Apart from at most $\epsilon \binom L2$ exceptional pairs, the pairs $\{V_i, V_j\}$ are $(\epsilon, G_s)$-regular for $s=1 \text{ and } 2$.
\end{itemize}
\end{lemma}

Additionally, if $G_1 \cup G_2$ is a multipartite graph with partition $V = V_1^* \cup V_2^* \cup \ldots \cup V_s^*$, with $s < 6$, we can guarantee that {\em each of the clusters $V_1, V_2, \ldots, V_L$ is contained entirely in a single part of this partition}.

To do so, for a given $\epsilon > 0$, we begin by arbitrarily partitioning each $V_i^*$ into parts $V_{i,1}^*, V_{i,2}^*, \ldots$, each of size $\lfloor \frac{\epsilon}{10}n\rfloor$, with a part $V_{i,0}^*$ of size at most $\frac{\epsilon}{10}n$ left over. This is an equitable partition of $V - \bigcup_{i=1}^k V_{i,0}^*$, a set of at least $(1 - \frac{9\epsilon}{10})n$ vertices. The  Regularity Lemma allows us to refine any equitable partition into one that satisfies the conclusions of Lemma~\ref{regularity-lemma}. Working with the subgraphs of $G_1$ and $G_2$ excluding the vertices in $\bigcup_{i=1}^k V_{i,0}^*$, take such a refinement with parameters $\frac{\epsilon}{9}$ and $m$, then add $\bigcup_{i=1}^k V_{i,0}^*$ to its exceptional cluster $V_0$. The resulting exceptional cluster still has size at most $\epsilon n$, so we have obtained a partition satisfying the conditions of Lemma~\ref{regularity-lemma} in which each of $V_1, V_2, \ldots, V_L$ is entirely contained in one of $V_1^*, V_2^*, \ldots, V_k^*$.

\subsection{Connected matchings}

Let $\alpha'(G)$ denote the size of a largest matching and $\alpha'_*(G)$ denote the size of a largest 
connected matching in $G$. Let $\alpha(G)$ denote the independence number and $\beta(G)$ denote the size of a smallest vertex cover in $G$.

Figaj and  \L uczak~\cite{FL1} were the first to use the fact that the existence of large connected matchings in the reduced graph of a regular partition of a large graph $G$ implies the existence of long paths and cycles in $G$. A flavor of it is illustrated by the following fact.

\begin{lemma}[Lemma 8 in~\cite{LSS1} and Lemma~1 in~\cite{KS1}]\label{L1}
Let a real number
$c > 0$ and a positive integer $k$
be given. If for every $\epsilon>0$ there exists a
$\delta>0$ and an $n_0$ such that for every even
$n > n_0$ and each graph $G$ with $v(G)>(1 +\epsilon) cn$ and
$e(G)\geq (1-\delta){v(G)\choose 2}$ and each $k$-edge-coloring of $G$
has a monochromatic  connected matching $M_{n/2}$, then for large $N$, $R_k(C_N)\leq(c+o(1))N$ (and hence
$R_k(P_N)\leq(c+o(1))N$).
\end{lemma}


We use the following property of $(\epsilon,G)$-regular pairs:  

\begin{lemma}[Lemma 3 in~\cite{GRSS1}]\label{srp}
\label{super-regular-pair}
For every $\delta>0$ there exist  $\epsilon>0$ and $t_0$ such that the following holds. Let $G$ be a bipartite graph with bipartition $(V_1, V_2)$ such that $|V_1|=|V_2| = t\ge t_0$, and let the pair $(V_1,V_2)$ be $(\epsilon,G)$-regular. Moreover, assume that $\deg_G(v) > \delta t$ for all $v \in V(G)$. 

Then for every pair of vertices $v_1 \in V_1, v_2 \in V_2$, $G$ contains a Hamiltonian path with endpoints $v_1$ and $v_2$.
\end{lemma}

Since we are aiming at an exact bound, we need a stability version of a result similar to Lemma~\ref{L1}. To state it, we need some
 definitions.

\begin{defn}\label{suitable}
For $\epsilon>0$, an $N$-vertex $s$-partite graph $G$ with parts $V_1,\ldots,V_s$ of sizes $n_1\geq n_2\geq\ldots\geq n_s$, and a 2-edge-coloring $E = E_1 \cup E_2$, is $(n,s,\epsilon)$-{\em suitable} if the following conditions hold:
\begin{equation}\label{suit-s1}\tag{S1}
	N=n_1+\ldots +n_s\geq 3n-1,
\end{equation}
\begin{equation}\label{suit-s2}\tag{S2}
	n_2+n_3+\ldots +n_s\geq 2n-1,
\end{equation}
and if $\widetilde{V}_i$ is the set of vertices in $V_i$ of degree at most $N-\epsilon n-n_i$ and $\widetilde{V}=\bigcup_{i=1}^s \widetilde{V}_i$, then 
\begin{equation}\label{suit-s3}\tag{S3}
	|\widetilde{V}|=|\widetilde{V}_1|+\ldots +|\widetilde{V}_s|<\epsilon n.
\end{equation}
We do not require $E_1 \cap E_2 = \emptyset$; an edge can have one or both colors. We write $G_i = G[E_i]$ for $i=1,2$.
\end{defn}

Our stability theorem gives a partition of the vertices of near-extremal graphs called a $(\lambda,i,j)$-{\em bad partition}. There are two types of bad partitions.

\begin{defn}
For $i\in \{1,2\}$ and $\lambda > 0$, a partition $V(G)=W_1\cup W_2$ of $V(G)$ is $(\lambda,i,1)$-{\em bad} if the following holds:
\begin{enumerate}
\item[(i)] $(1-\lambda)n \le |W_2| \le (1+\lambda)n_1$;
\item[(ii)] $|E(G_i[W_1,W_2])| \le \lambda n^2$;
\item[(iii)] $|E(G_{3-i}[W_1])| \le \lambda n^2$.
\end{enumerate}
\end{defn}

\begin{defn}
For $i\in \{1,2\}$ and $\lambda > 0$, a partition $V(G)=V_j\cup U_1\cup U_2$, $j \in [s]$, of $V(G)$ is $(\lambda,i,2)$-{bad} if the following holds:
\begin{enumerate}
\item[(i)] $|E(G_i[V_j,U_1])| \le \lambda n^2$;
\item[(ii)] $|E(G_{3-i}[V_j,U_2])| \le \lambda n^2$;
\item[(iii)] $n_j=|V_j| \ge (1-\lambda)n$;
\item[(iv)] $(1-\lambda)n \le |U_1| \le (1+\lambda)n$;
\item[(v)] $(1-\lambda)n \le |U_2| \le (1+\lambda)n$.
\end{enumerate}
\end{defn}

 Our stability theorem~\cite{BKLL1} is:

\begin{theorem}[Theorem~9~\cite{BKLL1}]\label{t1}
Let $0<\epsilon< 10^{-3}\gamma<10^{-6}$, $n\geq s\geq 2$, $n > \frac{100}{\gamma }$. Let $G$ be  an
 $(n,s,\epsilon)$-{suitable}  graph. If $\max\{\alpha'_*(G_1), \alpha'_*(G_2)\}\le n(1+\gamma)$, then for some $i\in [2]$ and $j \in [2]$, $V(G)$ has a $(64\gamma,i,j)$-bad partition.
\end{theorem}

\subsection{Theorems on Hamiltonian cycles in bipartite graphs}



\begin{theorem}[Chv\'atal~\cite{CHV},  see also Corollary 5 in Chapter 10 in~\cite{BGG}]\label{chvatal}
Let $H$ be a $2n$-vertex bipartite graph with vertices $u_1, u_2, \ldots, u_n$ on one side and $v_1, v_2, \ldots, v_n$ on the other, such that $d(u_1) \le \ldots \le d(u_n)$ and $d(v_1) \le \ldots \le d(v_n)$. 

If $d_H(u_i) \le i < n \implies d_H(v_{n-i}) \ge n-i+1$, then $H$ is Hamiltonian.
\end{theorem}


\begin{theorem}[Berge~\cite{BGG}]\label{berge}
Let $H$ be a $2m$-vertex bipartite graph with vertices $u_1, u_2, \ldots, u_m$ on one side and $v_1, v_2, \ldots, v_m$ on the other, such that $d(u_1) \le \ldots \le d(u_m)$ and $d(v_1) \le \ldots \le d(v_m)$. Suppose that for the smallest two indices $i$ and $j$ such that $d(u_i) \le i+1$ and $d(v_j) \le j+1$, we have $d(u_i) + d(v_j) \ge m+2$.

Then $H$ is Hamiltonian bi-connected: for every $i$ and $j$, there is a Hamiltonian path with endpoints $u_i$ and $v_j$.
\end{theorem}

\begin{theorem}[Las Vergnas~\cite{LV}, see also Theorem 11 on page 214 in~\cite{BGG}]\label{lasvergnas}
Let $H$ be a $2n$-vertex bipartite graph with vertices $u_1, u_2, \ldots, u_n$ on one side and $v_1, v_2, \ldots, v_n$ on the other, such that $d(u_1) \le \ldots \le d(u_n)$ and $d(v_1) \le \ldots \le d(v_n)$.  Let $q$ be an integer, $0 \le q \le n-1$.

If, whenever $u_i v_j \notin E(H)$, $d(u_i) \le i+q$, and $d(v_j) \le j+q$, we have
\[
	d(u_i) + d(v_j) \ge n + q + 1,
\]
then each set of $q$ edges that form vertex-disjoint paths is contained in a Hamiltonian cycle of $G$.
\end{theorem}

\subsection{Using the tools}
Our strategy to prove Theorem~\ref{tC2n} is: We first apply a $2$-colored version of the Regularity Lemma to $G$ to obtain a reduced graph $G^r$. If $G^r$ has a large monochromatic connected matching then we find a long monochromatic cycle using  Lemma~\ref{L1}. If $G^r$ does not have a large monochromatic connected matching, then we use Theorem~\ref{t1}  to obtain a bad partition of $G^r$. We then transfer the bad partition of $G^r$ to a bad partition of $G$ and work with this partition. In some important cases, theorems on Hamiltonian cycles  help to find a  monochromatic cycle $C_{2n}$ in $G$.

\section{Setup of the proof of Theorem~\ref{tC2n}}\label{setupM}
Formally, we need to prove the theorem for every
 $N$-vertex complete $s$-partite graph $G$ with parts $(V_1^*, V_2^*,\ldots, V_s^*)$ such that the numbers $n_i=|V_i^*|$ satisfy $n_1 \ge n_2 \ge \ldots \ge n_s$ and the following three conditions:

\begin{enumerate}
\item[$(S1')$] $N = n_1 + \ldots + n_s \ge 3n-1$;

\item[$(S2')$]  $N-n_1 = n_2 + \ldots + n_s \ge 2n-1$;

\item[$(S3')$] If $N-n_1-n_2\leq 2$, then  $ N\geq 4n-1$.
\end{enumerate}

For a given  large $n$, we consider a possible counterexample with the minimum $N+s$. In view of this,
it is enough to consider the lists $(n_1,\ldots,n_s)$ satisfying $(S1')$, $(S2')$ and $(S3')$ such that

\begin{enumerate}
\item[(a)] for each $1\leq i\leq s$, if $n_i>n_{i+1}$, then the list $(n_1,\ldots,n_{i-1},n_i-1,n_{i+1},\ldots,n_s)$ does not satisfy some of  $(S1')$, $(S2')$ and $(S3')$;
\item[(b)] if $s\geq 4$, then the list $(n_1,\ldots,n_{s-2},n_{s-1}+n_s)$ (possibly with the entries rearranged into a non-increasing order) does not satisfy some of $(S1')$, $(S2')$ and $(S3')$.
\end{enumerate}

{\bf Case 1:} $N-n_1-n_2\geq 3$ and  $N>3n-1$. Then $(S3')$ holds by default. If  $n_1 > n_2$, 
 then the list $(n_1-1,n_2,n_3,\ldots,n_s)$  still satisfies the conditions $(S1')$, $(S2')$ and $(S3')$, a contradiction to~(a).
Hence  $n_1=n_2$. Choose the maximum $i$ such that $n_1=n_i$. If $N-n_1>2n-1$, consider 
the list $(n_1,\ldots,n_{i-1},n_i-1,n_{i+1},\ldots,n_s)$. In this case $(S1')$ and $(S2')$ still are satisfied for this list; so by (a), $(S3')$
fails for it. As we assumed $N - n_1 -n_2 \ge 3$, we must have $i \ge 3$ and $N - n_1 -n_2 = 3$ for $(S3')$ to fail for this list; this further implies $n_1 = n_i \le 3$, so $N = n_1 + n_2 + 3 \le 9$, a contradiction. Thus in this case $N-n_1=2n-1$.
Therefore, $n_1= N-(N-n_1)\geq 3n-(2n-1)=n+1$ and hence $n_2\geq n+1$, so $N-n_1-n_2\leq (2n-1)-(n+1)=n-2$.
Then the list $(n_1,n_1,N-2n_1)$ satisfies $(S1')$--$(S3')$.
Summarizing, we get
\begin{equation}\label{n-n1=2n-1}
\mbox{\em if $N-n_1-n_2\geq 3$ and $N>3n-1$, then}\quad s=3,\;  n_2+n_3 = 2n-1\quad \text{\em  and }\quad n_1 = n_2.
\end{equation}

{\bf Case 2:} $N-n_1-n_2\geq 3$ and  $N=3n-1$. Again $(S3')$ holds by default. By $(S2')$, $n_1\leq n$, hence $N-n_1-n_2\geq n-1$.
If $s\geq 4$ and $n_{s-1}+n_s\leq n$, then let $L$ be the list obtained from $(n_1,\ldots,n_s)$ by replacing the two entries $n_{s-1}$
and $n_s$ with $n_{s-1}+n_s$ and then possibly rearrange the entries into non-increasing order. By construction, $L$ satisfies
 $(S1')$--$(S3')$, a contradiction to~(b). Hence $n_{s-1}+n_s\geq n+1$. We also have $n_{s-1} + n_s \ge n+1$ if $s = 3$, since in this case $n_{s-1} + n_s = N - n_1 \ge 2n-1$. If $s\geq 6$, then
 $N\geq 3(n_{s-1}+n_s)\geq 3n+3$, contradicting  $N=3n-1$. Thus
\begin{equation}\label{N3n-1}
\mbox{\em if $N-n_1-n_2\geq 3$ and $N=3n-1$, then}\; n_1\leq n,\quad s\leq 5,\quad n_{s-1}+n_s \geq n+1.
\end{equation}

{\bf Case 3:} $N-n_1-n_2\leq 2$. Then $N\leq 2n_1+2$, so
by $(S3')$, $2n_1+2\geq N\geq 4n-1$, implying $n_1\geq 2n-1$. If $n_1\geq 2n$, then $(S2')$ implies that 
$G\supseteq K_{2n,2n-1}$. 
If $n_1= 2n-1$, then by $(S3')$, $N-n_1\geq 2n$, so again $G\supseteq K_{2n,2n-1}$. 
Thus we can assume that
\begin{equation}\label{N4n-2}
\mbox{\em if $N-n_1-n_2\leq 2$, then}\quad G= K_{2n,2n-1}.
\end{equation}

As we have seen, 
\begin{equation}\label{Se4}
\mbox{in each of Cases $1$, $2$ and $3$ we have $s \le 5$.}
\end{equation}


Fix an arbitrary $2$-edge-coloring $E(G)=E_1\cup E_2$ of $G$. For $i\in [2]$ and $v\in V(G)$, let 
 $G_i := (V(G), E_i)$ and $d_i(v)$ denote the degree of $v$ in $G_i$.

\section{Regularity}\label{regularity}

\subsection{Applying the $2$-colored version of the Regularity Lemma}


We first choose parameter $\alpha$ so that $0 < \alpha \ll 1$ and then choose $\epsilon$ such that $\epsilon < 10^{-15}$ and $0 < \epsilon \ll \alpha$ so that the pair $(\frac{\alpha}{2}, \epsilon)$ satisfies the relation of $(\delta, \epsilon)$ in Lemma~\ref{srp} with $\frac{\alpha}{2}$ playing the role of $\delta$. Here, $\epsilon$ is the parameter for the  Regularity Lemma, and $\alpha$ is our cutoff for the edge density at which we give an edge of the reduced graph a color.

We apply Lemma~\ref{regularity-lemma} to obtain a partition $(V_0, V_1, \ldots, V_L)$ of $V(G)$, with each of $V_1, V_2, \ldots, V_L$ contained entirely in one of $V_1^*, V_2^*, \ldots, V_k^*$. Define the $k$-partite \emph{reduced graph} $G^r$ as follows:
\begin{itemize}
\item The vertices of $G^r$ are $v_i$ for $i=1, 2, \ldots, L$. A $k$-partition $(V_1', V_2', \ldots, V_k')$ of $V(G^r)$ is induced by the $k$-partition of $G$, and reordered if necessary so that $|V_1'| \ge |V_2'| \ge \ldots \ge |V_k'|$.
\item There is an edge between $v_i$ and $v_j$ iff $v_i$ and $v_j$ are in different parts of the $k$-partition, and the pair $\{V_i, V_j\}$ is $(\epsilon,G_s)$-regular for both $s=1 \text{ and } s=2$.

\item The reduced graph $G^r$ is missing at most $\epsilon \binom L2$ edges between distinct pairs $\{V_i', V_j'\}$.

\item We give $G^r$ a $2$-edge-multicoloring: two graphs $(G_1^r, G_2^r)$ whose union include every edge of $G^r$, but are not necessarily edge-disjoint. We add edge $v_i v_j \in E(G^r)$ to $G_s^r$ if $G_s$ contains at least $\alpha |V_i||V_j|$ of the edges between $V_i$ and $V_j$. Since $G = G_1 \cup G_2$ contains all $|V_i||V_j|$ edges between $V_i$ and $V_j$, each edge of $G^r$ is added to either $G_1^r$ or $G_2^r$, and possibly to both.
\end{itemize}

Let $t = |V_1| = |V_2| = \ldots = |V_L|$, $\ell_i = |V_i'|$ for $i=1, \ldots, k$, and $\ell := \frac{n - \epsilon N}{t}$; since $N \le 4n-2$, we have $\ell t \ge (1-5\epsilon)n$.

Because $|V_0| \le \epsilon N$, we have $(1-\epsilon)N \le Lt \le N$ and $n_i - \epsilon N \le \ell_i t \le n_i$. Therefore,
\begin{itemize}
\item $Lt \ge (1-\epsilon)N \ge 3n-1 - \epsilon N = 3(\ell t + \epsilon N) - 1 - \epsilon N \ge 3 \ell t - 1 + 2\epsilon n$, which means $L \ge 3 \ell-1$.
\item $Lt - \ell_1t \ge N - n_1 - \epsilon N \ge 2n-1 - \epsilon N \ge 2(\ell t + \epsilon N)-1 - \epsilon N \ge 2 \ell t - 1 +\epsilon N$, which means $L - \ell_1 \ge 2 \ell-1$.
\end{itemize}

Since the number of $V_i$'s missing at least $\sqrt{\epsilon} L$ edges is less than $\sqrt{\epsilon} L$, $G^r$ is $(\ell, s, \sqrt{\epsilon})$-suitable. We apply Theorem~\ref{t1} to the graph $G^r$ with $\gamma$ such that $10^{-3} > \gamma \gg \alpha$ and $\gamma > 2000\sqrt{\epsilon}$. Then we conclude that either $G^r$ has a monochromatic connected matching of size $(1+\gamma)\ell$, or else $V(G)$ has a $(64\gamma, i,j)$-bad partition for some $i \in [2]$ and $j \in [2]$.

\subsection{Handling a large connected matching in the reduced graph}

For every edge $v_i v_j \in G_1^r$, the corresponding pair $(V_i,V_j)$ is $(\epsilon,G_1)$-regular and contains at least $\alpha t^2$ edges of $G_1$. Let $X_{ij} \subseteq V_i$ be the set of all vertices of $V_i$ with fewer than $\frac{\alpha}{2} t$ edges of $G_1$ to $V_j$, and let $Y_{ij} \subseteq V_j$ the set of all vertices of $V_j$ with fewer than $\frac{\alpha}{2}t$ edges of $G_1$ to $V_i$. Note we have $Y_{ij} = X_{ji}$ but we keep using the notation $Y_{ij}$ for emphasising they are in different parts. Then $\frac{|E(X_{ij}, V_j)|}{|X_{ij}||V_j|} \le \frac{\alpha}{2}$, so $|X_{ij}| \le \epsilon t$ to avoid violating $(\epsilon,G_1)$-regularity; similarly, $|Y_{ij}| \le \epsilon t$. Call vertices of $V_i \cup V_j$ which are not in $X_{ij} \cup Y_{ij}$ \emph{typical} for the pair $(V_i, V_j)$ (or for the edge $v_i v_j$ of $G_1$).

Let $\mathcal{M}$ be a connected matching in $G_1^r$ of size $(1+\gamma)\ell$. Give the edges in $\mathcal{M}$ an arbitrary cyclic ordering.

If $v_{i_1}v_{j_1}$ and $v_{i_2}v_{j_2}$ are edges of $\mathcal{M}$ which are consecutive in the ordering, we shall find a path $P(j_1, i_2)$ in $G_1$ joining a vertex of $V_{j_1} \setminus Y_{i_1j_1}$ to a vertex of $V_{i_2} \setminus X_{i_2j_2}$. To do so, we begin by finding a path $P^r$ from $v_{j_1}$ to $v_{i_2}$ in $G_1^r$, then find a realization of that path in $G_1$. Pick a starting point of $P(j_1, i_2)$ typical both for the edge $v_{i_1}v_{j_1}$ and for the first edge of $P^r$. Next, choose the path greedily, making sure to satisfy the following conditions:
\begin{itemize}
\item Choose a neighbor of the previous vertex chosen which is typical for the next edge of $P^r$ (or for $v_{i_2j_2}$ when we reach the end of $P^r$).
\item Choose a vertex which has not been chosen for any previous paths.
\end{itemize}
As we construct $P(j_1, i_2)$, the last vertex we have chosen is always typical for the edge of $P^r$ we are about to realize; therefore we have at least $\frac{\alpha}{2}t$ options for its neighbors. At most $\epsilon t$ of them are eliminated because they are not typical for the next edge, and at most $L^2$ are eliminated because they have been chosen for previous paths. Since $L$ is upper bounded by $M$ which is independent of $n$, and $\epsilon \ll \alpha$, we can always choose such a vertex.

Moreover, we may choose the paths such that their total length has the same parity as $|\mathcal{M}|$. If the component of $G_1^r$ containing $\mathcal{M}$ is not bipartite, then each path can be chosen to have any parity we like. If the component of $G_1^r$ containing $\mathcal{M}$ is bipartite, then this condition is satisfied automatically: if we join the paths of $P^r$ we chose by the edges of $\mathcal{M}$, we get a closed walk, which must have even length.

Once all these paths are chosen, we combine them into a long even cycle in $G_1$. For each edge $v_i v_j$ in the matching $\mathcal{M}$, we have vertices $x \in V_i$ and $y \in V_j$, both typical for $(V_i, V_j)$, which are the endpoints of two paths we have constructed. We show that we can find a path from $x$ to $y$ using only edges of $G_1$ between $V_i$ and $V_j$ of any odd length between $t-1$ and $(1-3\epsilon)2t-1$. 

To do so, we choose any $X \subseteq V_i$ with $|X| \ge \frac t2$ that contains $x$ and at least $\frac{\alpha}{2}t$ neighbors of $y$; similarly, we choose $Y \subseteq V_j$ with $|Y| = |X|$ that contains $y$ and at least $\frac{\alpha}{2}t$ neighbors of $x$. If we want the path to have length $2Ct - 1$ where $C \in [\frac12, 1 - 3\epsilon]$, we begin by choosing $X$ and $Y$ of size $(C+3\epsilon)t$. The pair $(X,Y)$ is $(2\epsilon,G_1)$-regular with density at least $\alpha-\epsilon$, so there are at most $2\epsilon$ vertices in each of $X$ and $Y$ which have fewer than $\frac{\alpha}{2}t$ neighbors on the other side; by our construction of $X$ and $Y$, $x$ and $y$ are not among them. 

Let $X' \subseteq X$ and $Y' \subseteq Y$ be the subsets obtained by deleting these low-degree vertices, leaving at least $(C+\epsilon)t$ vertices on each side, and then deleting enough vertices from each part to make $|X'| = |Y'| = Ct$. The pair $(X', Y')$ is $(3\epsilon, G_1)$-regular, and all vertices have minimum degree at least $(\alpha -3\epsilon)t$, so by Lemma~\ref{super-regular-pair}, there is a path from $x$ to $y$ using all vertices of $X'$ and $Y'$, which has the desired length $2Ct-1$.

If we use $C = 1-3\epsilon$ for each edge $v_i v_j$ in the matching $\mathcal{M}$, then the cycle contains at least $2(1-3\epsilon)t$ vertices for each edge of $\mathcal{M}$, even ignoring the paths we constructed between them, while $|\mathcal{M}| \ge (1+10\epsilon)\ell$; therefore the total length is at least
\[
	2(1-3\epsilon)(1 + 10\epsilon)\ell t \ge 2 (1-3\epsilon)(1+10\epsilon)(1-5\epsilon)n \ge (1+\epsilon)2n.
\]
If we use $C = \frac12$ each edge $v_iv_j$, then the cycle contains only $t$ vertices for each edge of $\mathcal{M}$, giving approximately half as many edges. Up to parity, we are free to choose any length in this range, and therefore it is possible to construct a path in $G_1$ of length exactly $2n$.

\subsection{Handling a bad partition of the reduced graph}
We will show in Sections~\ref{1bad} and~\ref{2bad} how to find a long monochromatic cycle in a bad partition of $G$. In this subsection, we show that a bad partition of $G^r$ corresponds to a bad partition of $G$. 

\begin{enumerate}
\item If $X \subseteq V(G^r)$ has size $C\ell$, then the corresponding set of vertices in $G$ is
$\quad
	\bigcup_{v_i \in X} V_i.$\\
It has size $C\ell t$, which is in the range $[(1-5\epsilon)Cn,Cn]$.

\item If $|E_{G_i^r}(X)| \le \lambda \ell^2$, then each of those $\lambda \ell^2$ edges of $G_i^r$ corresponds to at most $t^2$ edges of $G_i$, for $\lambda \ell^2 t^2 \le \lambda n^2$ edges. 

Additionally, edges not in $G_i^r$ may appear in $G_i$; across all of $G_i$ there are at most $\alpha t^2 \binom{L}{2} \le \frac12 \alpha N^2 \le 10\alpha n^2$ edges that occur in this way.

Moreover, edges from at most $\epsilon \binom L2$ exceptional pairs may appear in $G_i$, contributing at most $10\epsilon n^2$ edges in total by the same calculation.

To summarize, there are at most $(\lambda + 10\alpha + 10\epsilon)n^2$ edges in $G_i$ corresponding to $E_{G_i^r}(X)$. A similar argument applies to a bound on $|E_{G_i^r}(X,Y)|$ for $X,Y \subseteq V(G^r)$.

\item There are fewer than $\epsilon N \le 5\epsilon n$ vertices from the exceptional part $V_0$, which can generally be assigned to any part of any bad partition without changing the approximate structure.
\end{enumerate}

Thus, for $1 \gg \lambda \gg \alpha \gg \epsilon \gg 0$, if $G^r$ has a $(\lambda, i, 1)$-bad partition ($i \in [2]$) $V(G^r) = W_1^r \cup W_2^r$, then $G$ has a corresponding $(2\lambda, i, 1)$-bad partition with 

\textbf{(0):} $$W_1 := \left( \bigcup_{v_i \in W_1^r} V_i \right) \cup V_0 \text{ and } W_2 := \bigcup_{v_i \in W_2^r} V_i .$$

\textbf{(i):} $$(1-2\lambda)n \le (1-\lambda)(1-5\epsilon)n \le (1-\lambda)\ell t \le |W_2| \le (1+\lambda)\ell_1 t \le (1+\lambda)n_1.$$

\textbf{(ii):} $$|E(G_i[W_1, W_2])| \le (\lambda + 10\alpha + 10\epsilon + 5 \epsilon)n^2 \le 2 \lambda n^2.$$

\textbf{(iii):} $$|E(G_{3-i}[W_1])| \le (\lambda + 10\alpha + 10\epsilon + 5 \epsilon)n^2 \le 2 \lambda n^2.$$

If $G^r$ has a $(\lambda, i, 2)$-bad partition ($i \in [2]$) $V(G^r) = V_j' \cup U_1^r \cup U_2^r$ then $G$ has a corresponding $(2\lambda, i, 2)$-bad partition with 

\textbf{(0):} $$U_1 := \bigcup_{v_i \in U_1^r} V_i \cup (V_0-V^*_j) \text{ and } U_2 := \bigcup_{v_i \in U_2^r} V_i.$$

\textbf{(i):} $$|E(G_i[V_j^*, U_1])| \le (\lambda + 10\alpha + 10\epsilon + 5 \epsilon)n^2 \le 2 \lambda n^2.$$

\textbf{(ii):} $$|E(G_{3-i}[V_j, U_2])| \le (\lambda + 10\alpha + 10\epsilon + 5 \epsilon)n^2 \le 2 \lambda n^2.$$

\textbf{(iii):} $$n_j = |V_j^*| \ge \ell_j t \ge (1-\lambda)\ell t  \ge (1-\lambda)(1-5\epsilon)n  \ge (1-2\lambda)n.$$

\textbf{(iv):} $$(1+2\lambda)n \ge (1+\lambda)n+5\epsilon n \ge (1+\lambda)\ell t+5\epsilon n \ge |U_1| \ge (1-\lambda)\ell t \ge (1-\lambda)(1-5\epsilon)n \ge (1-2\lambda)n.$$

\textbf{(v):} $$(1+\lambda)n \ge (1+\lambda)\ell t \ge |U_2| \ge (1-\lambda)\ell t \ge (1-\lambda)(1-5\epsilon)n \ge (1-2\lambda)n.$$

Therefore, a $(64 \gamma, i, j)$-bad partition of $G^r$ corresponds to a $(128 \gamma, i, j)$-bad partition of $G$ for some $i \in [2]$ and $j \in [2]$. In the next three sections we show how to find a monochromatic cycle of length exactly $2n$ when $G$ has a $(\lambda, i, j)$-bad partition for some $i \in [2]$ and $j \in [2]$, where $\lambda = 128 \gamma$.

\section{Dealing with $(\lambda,i,1)$-bad partitions when $N-n_1-n_2\geq 3$}\label{1bad}

\subsection{Setup}

Without loss of generality, let $i=1$. Recall that $d_k(v)$ is the degree of $v$ in $G_k$, where $k \in [2]$. We assume that for some $\lambda < 0.01$
, there is a partition $V(G) = W_1 \cup W_2$ such that:
\begin{align}
(1-\lambda)n \le |W_2| &\le (1+\lambda)n_1; \label{cond6.1}\\
|E(G_1[W_1, W_2])| &\le \lambda n^2; \label{cond6.2} \\
|E(G_2[W_1])| &\le \lambda n^2. \label{cond6.3}
\end{align}

If $G$ has at least $4$ parts then either $N=3n-1$ and $N-n_1 \ge 2n-1$ which implies $n_1 \le n$, or else $N-n_1 = 2n-1$, $n_1 = n_2$ and $n_3+n_4 > n$ which implies $n_1 < n$. If $G$ is tripartite, then we could have $n_1$ much larger than $n$, but in this section, we will assume $n_1 < \frac53n$. The alternative, that $G$ is tripartite and $n_1 \ge \frac53n$, is handled in Subsection~\ref{subsection:nearly-bipartite}.

We know that $|W_1| \ge N - (1+\lambda)n_1 \ge 2n-1 - \lambda n_1 \ge (2-5\lambda)n$ since $n_1 \le 2n$. For any vertex $x$, fewer than $\frac53 n$ vertices of $W_1$ can be in the same part $V_i$ of $G$ as $x$, so at least $(\frac13 - 5\lambda)n > \frac14 n$ are in other parts of $G$. In other words, we have $d(x, W_1) \ge \frac 14 n$ for all $x \in V(G)$. 

We call a vertex $x \in V(G)$ $W_1$\emph{-typical} if $d_1(x, W_1) \ge \frac34 d(x, W_1)$, and $W_2$\emph{-typical} if $d_1(x, W_1) < \frac34 d(x, W_1)$.

If $x$ is $W_1$-typical, then $d_1(x, W_1) \ge \frac34 \cdot \frac14n = \frac3{16} n$. Since
\[
	\sum_{x \in W_2} d_1(x, W_1) = |E(G_1[W_1, W_2])| \le \lambda n^2,
\]
the number of $W_1$-typical vertices in $W_2$ is at most $\frac{\lambda n^2}{3n/16} < 6 \lambda n$.

Similarly, if $x$ is $W_2$-typical, then $d_2(x, W_1) \ge \frac14 \cdot \frac14 n = \frac1{16} n$. Since
\[
	\sum_{x \in W_1} d_2(x, W_1) = 2|E(G_2[W_1])| \le 2\lambda n^2,
\]
the number of $W_2$-typical vertices in $W_1$ is at most $\frac{2\lambda n^2}{n/16} = 32\lambda n$.

Let $W_1'$ be the set of all $W_1$-typical vertices, and $W_2'$ be the set of all $W_2$-typical vertices. The partition $(W_1', W_2')$ is almost exactly the same as the partition $(W_1, W_2)$: at most $40\lambda n$ vertices have been moved from one part to the other part to obtain $(W_1', W_2')$ from $(W_1, W_2)$. Therefore, if $x \in W_1'$, we still have $d_1(x, W_1') \ge \frac34 d(x, W_1) - 40\lambda n$, and if $x \in W_2'$, we still have $d_1(x, W_1') <\frac34 d(x,W_1) + 40\lambda n$. In either case, we still have $d(x, W_1') \ge \frac14 n - 40\lambda n$ for all $x$.

Moreover, $W_1'$ and $W_2'$ still satisfy similar conditions to $W_1$ and $W_2$:
\begin{enumerate}
\item $(1-41\lambda)n \le |W_2'| \le (1+\lambda)n_1 + 40\lambda n \le (1 + 81\lambda)n_1$ (since $n_1 \ge \frac n2$ in all cases).
\item $|E(G_1[W_1', W_2'])| \le \lambda n^2 + N \cdot (40\lambda n) \le 161\lambda n^2$, since we move at most $40\lambda n$ vertices with degree less than $N$.
\item $|E(G_2[W_1'])| \le \lambda n^2 + N \cdot (6\lambda n) \le 25\lambda n^2$, since we move at most $6\lambda n$ vertices with degree less than $N$ into $W_1'$.
\end{enumerate}

For convenience, let $\delta = 200\lambda$, which is at least as large as all multiples of $\lambda$ used above.

Our goal is to find a cycle of length $2n$ in either $G_1$ or $G_2$. We decide which type of cycle we will attempt to find based on the relative sizes of $W_1'$ and $W_2'$.

Suppose that $|W_1'| \ge 2n$ and, moreover, $|W_1' \setminus V_i| \ge n$ for all $i$. In this case, we find a cycle of length $2n$ in $G_1$; this is done in Subsection~\ref{subsection:inside-w1}.

Otherwise, we must have $|W_2'| \ge n$: either $|W_1'| \le 2n-1$, and $|W_2'| = N - |W_1'| \ge n$, or else $|W_1' \setminus V_i| \le n-1$ for some $i$, and
\[
	|W_2'| \ge |W_2' \setminus V_i| = |V \setminus V_i| - |W_1' \setminus V_i| \ge (N - n_i) - (n-1) \ge (2n-1)-(n-1) = n.
\]
In this case, we find a cycle of length $2n$ in $G_2$; this is done in Subsection~\ref{subsection:w1-to-w2}.

We use the following lemma to pick out ``well-behaved'' vertices in $W_1'$ and $W_2'$. For example, we commonly apply it to $G_2[W_1']$ or to $G_1[W_1', W_2']$.

\begin{lemma}
\label{lemma:good-vertices}
Let $H$ be an $n$-vertex graph with at most $\epsilon n^2$ edges, for some $\epsilon>0$, and let $S \subseteq V(H)$.  If $S' \subseteq S$ is any subset that excludes the $k$ vertices of $S$ with the highest degree, then every $v \in S'$ satisfies $d_H(v) < \frac{2\epsilon n^2}{k}$.

Additionally, when $H$ is bipartite, and $S$ is entirely contained in one part of $H$, every $v \in S'$ satisfies $d_H(v) < \frac{\epsilon n^2}{k}$.
\end{lemma}
\begin{proof}
In the first case, if we have $d_H(v) \ge \frac{2\epsilon n^2}{k}$ for any $v \in S'$, then we also have $d_H(v) \ge d$ for the $k$ vertices of $S$ with the highest degree, which we excluded from $S'$. The sum of degrees of these $k+1$ vertices exceeds $2\epsilon n^2$, so it is greater than twice the number of edges in $H$, a contradiction.

In the second case, if we have $d_H(v) \ge \frac{\epsilon n^2}{k}$ for any $v \in S'$, the same sum of degrees exceeds $\epsilon n^2$. But since the vertices of $S$ are all on one side of the bipartition of $H$, this sum of degrees cannot be greater than the number of edges in $H$, which is again a contradiction.
\end{proof}

\subsection{The nearly-bipartite subcase}
\label{subsection:nearly-bipartite}

In this subsection, we assume that $G$ is tripartite with $n_1 \ge \frac53 n$. Recall that when $G$ is tripartite we have $n_1 = n_2$ and $n_1 + n_3 = n_2 + n_3 = 2n-1$, and that throughout Section~\ref{1bad} we assume $N - n_1 - n_2 \ge 3$, or in this case that $n_3 \ge 3$.

\textbf{Case 1:} $|W_1 \cap V_i| \ge (1+10\lambda)n$ for $i=1$ or $i=2$. We assume $i=1$; the proof for the case $i=2$ is the same. In this case, let $X$ be an $n$-vertex subset of $V_1 \cap W_1$ avoiding the $5\lambda n$ vertices of $V_1 \cap W_1$ with the most edges of $G_2$ to $W_1 \setminus V_1$ and the $5\lambda n$ vertices of $V_1 \cap W_1$ with the most edges of $G_1$ to $W_2 \setminus V_1$. 

For any vertex $v \in X$, we have $d_2(v, W_1\setminus V_1) \le \frac{\lambda n^2}{5\lambda n} = \frac15 n$ and $d_1(v, W_2 \setminus V_1) \le \frac15 n$ by Lemma~\ref{lemma:good-vertices}.

We partition $V_2 \cup V_3$ into sets $Y_1$ and $Y_2$ by the following procedure.
\begin{enumerate}
\item The $2 \lambda n$ vertices of $W_1 \setminus V_1$ with the most edges of $G_2$ to $X$ are set aside, and the remaining vertices of $W_1 \setminus V_1$ are assigned to $Y_1$.

By Lemma~\ref{lemma:good-vertices}, any vertex $v$ assigned to $Y_1$ in this step has $d_2(v, X) \le \frac12 n$.

\item The $2\lambda n$ vertices of $W_2 \setminus V_1$ with the most edges of $G_1$ to $X$ are set aside, and the remaining vertices of $W_2 \setminus V_1$ are assigned to $Y_2$.

By Lemma~\ref{lemma:good-vertices}, any vertex $v$ assigned to $Y_2$ in this step has $d_1(v,X) \le \frac12 n$.

\item Each remaining vertex $v$ is assigned to $Y_1$ if $d_1(v, X) \ge \frac n2$ and to $Y_2$ otherwise (in which case $d_2(v, X) \ge \frac n2$).
\end{enumerate}
Since $|V_2 \cup V_3| = 2n-1$, we must have $|Y_1| \ge n$ or $|Y_2| \ge n$. Let $Y_j'$ be an $n$-vertex subset of $Y_j$, where $j \in [2]$ and $|Y_j| \ge n$. We apply Theorem~\ref{chvatal} to find a Hamiltonian cycle in the bipartite graph $H = G_j[X, Y_j']$. 

The minimum $H$-degree in $X$ is $\frac45n - 2\lambda n$, since each $v \in X$ had at most $\frac15 n$ edges to $W_j \setminus V_1$ which were not in $G_j$, and at most $2\lambda n$ vertices of $Y_j'$ did not come from $W_j \setminus V_1$ originally. The minimum $H$-degree in $Y_j'$ is $\frac n2$, so the condition of Theorem~\ref{chvatal} is satisfied: whenever $d_H(u_i) \le i$, we have $i \ge (\frac45 - 2\lambda)n$, so $d_H(v_{n-i}) \ge \frac n2 \ge (\frac15 + 2\lambda)n+1$.

\textbf{Case 2:} $|V_i \cap W_1| < (1+10\lambda)n$ for $i=1$ and $i=2$. By~\eqref{cond6.1}, we must have $|W_1| \ge N - (1+\lambda)n_1 = 2n-1 - \lambda n_1 > 2n - 3\lambda n$. Since $n_1=n_2 \ge \frac{5n}{3}$ and $n_2+n_3 = 2n-1$, fewer than $\frac13 n$ vertices of $W_1$ are in $V_3$, so at least $(\frac53 - 3\lambda)n$ of them are in $V_1 \cup V_2$; therefore $|W_1 \cap V_1| > (\frac23 - 13\lambda)n$ and $|W_1 \cap V_2| > (\frac23 - 13\lambda)n$.

Because $2n > n_1 = n_2 \ge \frac53 n$, we have $(\frac23 - 10\lambda)n < |V_i \cap W_2| < (\frac43 + 13\lambda)n$ for $i=1,2$, as well.

Next, we choose subsets $X_{ij} \subseteq V_i \cap W_j$ with $|X_{11}| = |X_{21}| = |X_{12}| = |X_{22}| = \frac n2 + 10$. To choose $X_{11}$ and $X_{21}$, avoid the $\frac{1}{20}n$ vertices with the most edges in $G_1$ to $W_2$ and the $\frac{1}{20}n$ vertices with the most edges in $G_2$ to $W_1$, so that each chosen vertex has at most $20\lambda n$ edges of each kind by Lemma~\ref{lemma:good-vertices}. To choose $X_{12}$ and $X_{22}$, avoid the $\frac{1}{10}n$ vertices with the most edges in $G_1$ to $W_1$, so that each chosen vertex has at most $10\lambda n$ such edges by Lemma~\ref{lemma:good-vertices}.

First, we observe that if $H$ is any of the graphs $G_1[X_{11}, X_{21}]$, $G_2[X_{12}, X_{21}]$, or $G_2[X_{11}, X_{22}]$, then given any vertices $v, w$ in $H$, we can find a $(v,w)$-path in $H$ on $m$ vertices, provided that $n-10 \le m \le n+10$ (this is not optimal, but it is more than we need) and that the parity of $m$ is correct. 

To do so, we apply Theorem~\ref{lasvergnas}. If $v$ and $w$ are on the same side of $H$, add a vertex $x$ to the other side adjacent to all vertices in the side containing $v$ and $w$; if not, add an edge $vw$. Then take a subgraph containing $\lceil \frac m2 \rceil$ vertices from each side, making sure to include $v, w$ and if applicable $x$. In this subgraph, the minimum degree is at least $\lceil \frac m2 \rceil - 20\lambda n$, so we can use Theorem~\ref{lasvergnas} to find a Hamiltonian cycle in this graph containing either the edge $vw$ or the edges $vx$ and $xw$. Deleting the vertex $x$ or the edge $vw$, whichever applies, creates a $(v,w)$-path in $H$ of the correct length.

Suppose that $G_2[X_{12}, X_{22}]$ contains a matching $M = \{u_1u_2, v_1 v_2\}$ of size $2$, where $u_1, v_1 \in X_{12}$ and $u_2, v_2 \in X_{22}$. In that case, we can find a $(u_1,v_1)$-path $P$ in $G_2[X_{12}, X_{21}]$ on $2\lceil \frac n2\rceil + 1$ vertices and a $(u_2, v_2)$-path $Q$ in $G_2[X_{11}, X_{22}]$ on $2\lfloor \frac n2\rfloor - 1$ vertices by the previous observation. Joining the paths $P$ and $Q$ using the edges of the matching $M$, we find a cycle of length $2n$ in $G_2$.

Now we assume $G_2[X_{12}, X_{22}]$ does not contain a matching of size $2$. If the size of a maximum matching in this graph is one, then there is a vertex cover of size one  since $G_2[X_{12}, X_{22}]$ is bipartite. We delete this vertex cover from $X_{12}$ or $X_{22}$ (it depends on where this vertex cover is). Having changed $X_{12}$ and $X_{22}$ in this way, $G_1[X_{12}, X_{22}]$ is a complete bipartite graph, so it also has the property that any two vertices in it can be joined by a path on $m$ vertices, provided that $n-10 \le m \le n+10$ and that the parity of $m$ is correct.

Note that there are at least three vertices in $V_3$.

We say that a vertex $v \in V_3$:
\begin{itemize}
\item is \emph{$j$-adjacent} to a set $S$ if it has at least two edges in $G_j$ to $S$.
\item \emph{S-connects} $G_j$ if it is $j$-adjacent to both $X_{11}$ and $X_{12}$, or if it is $j$-adjacent to both $X_{21}$ and $X_{22}$. (``S-connects'' because it is $j$-adjacent to two sets in the \emph{same} part of $V_1$ or $V_2$.)
\item \emph{C-connects} $G_1$ if it is $1$-adjacent to both $X_{11}$ and $X_{22}$, or if it is $1$-adjacent to both $X_{12}$ and $X_{21}$. (``C-connects'' because the $j$-adjacency \emph{crosses} from $V_1$ to $V_2$.)
\item \emph{C-connects} $G_2$ if it is $2$-adjacent to both $X_{11}$ and $X_{21}$, or if it is $2$-adjacent to both $X_{12}$ and $X_{22}$.
\item \emph{Folds into $G_1$} if it is $1$-adjacent to both $X_{11}$ and $X_{12}$, or if it is $1$-adjacent to both $X_{21}$ and $X_{22}$.
\item \emph{Folds into $G_2$} if it is $2$-adjacent to both $X_{11}$ and $X_{22}$, or if it is $2$-adjacent to both $X_{12}$ and $X_{21}$.
\end{itemize}

Some comments on these definitions: first, a vertex that is $j$-adjacent to at least three of $X_{11}, X_{12}, X_{21}, X_{22}$ is guaranteed to both S-connect and C-connect $G_j$. Second, a vertex that is $j$-adjacent to only two of $X_{11}, X_{12}, X_{21}, X_{22}$ for each value of $j$ may S-connect both $G_1$ and $G_2$, or C-connect $G_1$ and fold into $G_2$, or C-connect $G_2$ and fold into $G_1$. In particular, each vertex either S-connects or C-connects some $G_j$.

If there are two vertices in $V_3$ that both S-connect $G_j$, or both C-connect $G_j$, then we can find a cycle of length $2n$ in $G_j$. The cases are all symmetric; without loss of generality, suppose $v, w\in V_3$ both S-connect $G_1$. We can find a path $P$ in $G_1[X_{11}, X_{21}]$ on $2\lceil \frac n2 \rceil - 1$ vertices that starts at a $G_1$-neighbor of $v$ and ends at a $G_1$-neighbor of $w$, and a path $Q$ in $G_1[X_{12}, X_{22}]$ on $2\lfloor \frac n2 \rfloor - 1$ vertices that starts at a $G_1$-neighbor of $v$ and ends at a $G_1$-neighbor of $w$. Joining $P$ and $Q$ via $v$ at one endpoint and via $w$ on the other creates a cycle of length $2n$ in $G_1$.

If we cannot find two vertices as in the previous paragraph, then the best we can do is to find, for some $j$, a vertex $v \in V_3$ that S-connects $G_j$ and another vertex $w \in V_3$ that C-connects $G_j$. Since $v$ does not C-connect $G_j$, it must also S-connect $G_{3-j}$.

There is at least one more vertex $x \in V_3$. By assumption, it does not S-connect $G_{3-j}$ and neither S-connects nor C-connects $G_j$, so it must fold into $G_j$ (and C-connect $G_{3-j}$).

Without loss of generality, suppose that $j=1$ and $x$ has a $G_1$-neighbor in both $X_{11}$ and $X_{21}$. We add an artificial edge $e_x$ between a pair of such neighbors of $x$.

As before, we can find a path $P$ in $G_1[X_{11}, X_{21}]$ joining a neighbor of $v$ to a different neighbor of $w$; we add the requirement that it uses the edge $e_x$, which is still possible by Theorem~\ref{lasvergnas}. We can also find a path $Q$ in $G_1[X_{12}, X_{22}]$ joining a neighbor of $v$ to a different neighbor of $w$. Since $v$ S-connects $G_1$ and $w$ C-connects $G_1$, one of these paths will have even length and the other will have odd length, but we can choose them to have $2n-3$ vertices total.

Now join the paths $P$ and $Q$ using the vertices $v$ and $w$, then replace the artificial edge $e_x$ by two edges to $x$ from its endpoints. The result is a cycle of length $2n$ in $G_1$.

\subsection{Finding a cycle in $G_1$}
\label{subsection:inside-w1}

In this subsection, we are considering a $2$-edge-colored graph $G$ and a partition $W_1' \cup W_2'$ of $V(G)$ satisfying the following properties:
\begin{enumerate}
\item \label{inw1-G}
$G$ is a complete $s$-partite graph with parts $V_1, V_2, \dots, V_s$ of size $n_1, n_2, \dots, n_s$, with $s\ge 3$ and $n_1 + \ldots + n_s \le 4n$. 

\item \label{inw1-W2size}
$(1-\delta)n \le |W_2'| \le (1+\delta)n_1$.

\item \label{inw1-W1edges}
$|E(G_1[W_1', W_2'])| \le \delta n^2$ and $|E(G_2[W_1'])| \le \delta n^2$.

\item \label{inw1-degree}
If $x \in W_1'$, then $d_1(x, W_1') \ge \frac34 d(x, W_1) - \delta n$.

\item \label{inw1-W1size}
$|W_1'| \ge 2n$ and $|W_1' \setminus V_i| \ge n$ for all $i$. (This is the assumption that leads to this subsection as opposed to Subsection~\ref{subsection:w1-to-w2}.)
\end{enumerate}
We can deduce a further degree condition that holds for all vertices $x \in W_1'$:
\begin{enumerate}
\setcounter{enumi}{5}
\item \label{inw1-degree3}
By Properties~\ref{inw1-G} and~\ref{inw1-W2size}, $|W_1'| = |V(G)| - |W_2'| \le 4n - (1-\delta)n = (3+\delta)n$, so $d(x, W_1') \le (3+\delta)n$. By Property~\ref{inw1-degree}, we have $d_2(x, W_1') \le \frac14(3+\delta)n + \delta n \le (\frac34 + 2\delta)n$.
\end{enumerate}
To find a cycle of length $2n$ in $G_1$, we will choose two disjoint sets $X, Y \subseteq W_1'$ of size $n$, then apply Theorem~\ref{chvatal} to find a Hamiltonian cycle in $H = G_1[X,Y]$.

Let $a, b \in \{1,2,\ldots,s\}$ be such that $V_a \cap W_1'$ is the largest part of $G_1[W_1']$ and $V_b \cap W_1'$ is the second largest part of $G_1[W_1']$. To define $X$ and $Y$, we begin by assigning $V_a \cap W_1'$ to $X$ and $V_b \cap W_1'$ to $Y$. If either of these exceeds $n$ vertices, we choose $n$ of the vertices arbitrarily.

Continue by assigning the parts $V_i \cap W_1'$ to either $X$ or $Y$ arbitrarily, for as long as this does not make $|X|$ or $|Y|$ exceed $n$. Once this is no longer possible, then:
\begin{itemize}
\item If there are still at least two parts $V_i \cap W_1'$ left unassigned, then each of them must have more than $\max\{n-|X|,n-|Y|\}$ vertices. Therefore we can add vertices from one of them to $X$ to make $|X|=n$ (if necessary), and add vertices from the other to $Y$ to make $|Y|=n$ (if necessary). 
\item If there is only one part of $G_1[W_1']$ left unassigned, call it $V_{\text{split}} \cap W_1'$. We assign $n-|X|$ vertices of $V_{\text{split}} \cap W_1'$ to $X$ and $n-|Y|$ other vertices of $V_{\text{split}} \cap W_1'$ to $Y$. 
\item If there are no parts left unassigned, then we must have $|X|=|Y|=n$.
\end{itemize}
We must show that we do not run out of vertices in either of the last two cases. If $|V_a \cap W_1'| \le n$, then we do not run out because $|W_1'| \ge 2n$ (by Property~\ref{inw1-W1size}) and all vertices in $W_1' \setminus V_{\text{split}}$ are assigned to either $X$ or $Y$, so either $V_{\text{split}} \cap W_1'$ must contain enough vertices to fill $X$ and $Y$ or $X$ and $Y$ are already full. If $|V_a \cap W_1'| > n$, then we do not run out because $|W_1' \setminus V_a| \ge n$  (again, by Property~\ref{inw1-W1size}), and after $V_a \cap W_1'$ is assigned, all vertices of $W_1'$ are added to $Y$ until it is full.

The most difficult case for us is the one in which some part $V_{\text{split}} \cap W_1'$ is divided between $X$ and $Y$. To handle all cases at once, we assume this happens; if necessary, we choose some part $V_i \cap W_1'$ ($i\ne a,b$) to be a degenerate instance of $V_{\text{split}}$ which is entirely in $X$ or $Y$.

Let $n_x = |V_{\text{split}} \cap X|$ and $n_y = |V_{\text{split}} \cap Y|$. We assigned the largest part of $G[W_1']$ to $X$ and the second-largest to $Y$; therefore $X$ and $Y$ both contain at least $n_x + n_y$ vertices not in $V_{\text{split}}$. Since $|X|=|Y|=n$, we must have $n_x + (n_x + n_y) \le n$ and $n_y + (n_x + n_y) \le n$; therefore $n_x + n_y \le \frac23 n$, while individually $n_x \le \frac n2$ and $n_y \le \frac n2$.

We first prove some bounds on $d_1(x,Y)$ for $x \in X$ (and, by symmetry, $d_1(y, X)$ for $y \in Y$). If $x \notin V_{\text{split}}$, then $d(x, Y) = n$ (since there are no vertices of $Y$ in the same part of $G$ as $x$) while $d_2(x, W_1') \le (\frac34 + 2\delta)n$ by Property~\ref{inw1-degree3}, so $d_1(x, Y) \ge (\frac14 - 2\delta)n$. If $x \in V_{\text{split}}$, then $d(x, W_1') = (n - n_x) + (n-n_y)$, since all vertices of $W_1'$ outside $V_{\text{split}}$ have been assigned to either $X$ or $Y$, so $d_2(x, W_1') \le \frac14(2n - n_x - n_y) + \delta n$ by Property~\ref{inw1-degree}. This leaves $d_1(x, Y) \ge \frac12n - \frac34 n_y - \delta n \ge (\frac18 - \delta)n$.

If we exclude the $\frac{1}{10} n$ vertices of $X$ with the most edges to $W_1'$ in $G_2$, then by Lemma~\ref{lemma:good-vertices}, the remaining vertices $x\in X$ have $d_2(x, W_1') \le 20\delta n$. If $x \notin V_{\text{split}}$, this means $d_1(x,Y) \ge (1-20\delta)n$, and if $x \in V_{\text{split}}$, this means that $d_1(x,Y) \ge n - n_y - 20\delta n$.

Let $H = G_1[X,Y]$, let $u_1, u_2, \ldots, u_n$ be the vertices of $X$ ordered so that $d_H(u_1) \le \ldots \le d_H(u_n)$, and let $v_1, v_2, \ldots, v_n$ be the vertices of $Y$ ordered so that $d_H(v_1) \le \ldots \le d_H(v_n)$. 

Suppose $u_i \in X$ satisfies $d_H(u_i) \le i < n$. We have shown $d_1(x, Y) \ge (\frac18 - \delta)n$, so among $u_1, u_2, \ldots, u_i$, there must be a vertex not among the $\frac{1}{10}n$ vertices of $X$ with the most edges to $W_1'$ in $G_2$. For such a vertex, $d_1(x,Y) \ge n - n_y - 20\delta n$, so in particular $d_H(u_i) \ge n - n_y - 20\delta n$, which means $i \ge n - n_y - 20\delta n$. 

If we had $d_H(v_{n-i}) \le n-i$, then by repeating this argument for vertices in $Y$, we would have $d_H(v_{n-i}) \ge n - n_x - 20\delta n$, which would mean $n-i \ge n - n_x - 20\delta n$. Adding this to the inequality on $i$, we would get $n \ge 2n - n_x - n_y - 40\delta n$, which is impossible since $n_x + n_y \le \frac23 n$. So we must have $d_H(v_{n-i}) \ge n-i+1$, and by Theorem~\ref{chvatal}, $H$ contains a Hamiltonian cycle. This gives a cycle of length $2n$ in $G_1$.

\subsection{Finding a cycle in $G_2$}
\label{subsection:w1-to-w2}

In this subsection, we are considering a $2$-edge-colored graph $G$ and a partition $W_1' \cup W_2'$ of $V(G)$ satisfying the following properties:
\begin{enumerate}
\item \label{w1w2-G}
$G$ is a complete $s$-partite graph with parts $V_1, V_2, \dots, V_s$ of size $n_1, n_2, \dots, n_s$, with $s\ge 3$ and $n_1 + \ldots + n_s \le 4n$. Morever, $\frac53n > n_1 \ge \dots \ge n_s$; we considered the case $n_1 \ge \frac53n$ in Subsection~\ref{subsection:nearly-bipartite}.

\item \label{w1w2-Gcases}

Either $N - n_1 > 2n-1$ and $|V_i| \le n$ for all $i$, or $n_1 = n_2 \ge n$, $s=3$, and $N-n_1 = N-n_2 = 2n-1$.

\item \label{w1w2-W1edges}
$|E(G_1[W_1', W_2'])| \le \delta n^2$ and $|E(G_2[W_1'])| \le \delta n^2$.

\item \label{w1w2-degree}
If $x \in W_2'$, then $d(x, W_1') \ge \frac14 n - \delta n$, and $d_2(x, W_1') \ge \frac14 d(x, W_1') - \delta n$.

\item \label{w1w2-W2size}
$n \le |W_2'| \le (1+\delta)n_1$. (The lower bound  is the assumption that leads to this subsection as opposed to Subsection~\ref{subsection:inside-w1}.)
\end{enumerate}

\newcommand{\Good}{\mathsf{Good}}
\newcommand{\Bad}{\mathsf{Bad}}

Let $\Bad$ consist of the $\sqrt{\delta} n$ vertices of $W_2'$ that maximize $d_1(x, W_1')$; let $\Good = W_2' \setminus \Bad$. By Lemma~\ref{lemma:good-vertices}, $d_1(x, W_1') \le \sqrt{\delta} n$ for all $x \in \Good$.

Our strategy is to handle the vertices in $\Bad$: first by finding short vertex-disjoint paths contaning the vertices in $\Bad$, then by combining them into a single path. Finally, we extend this path to a cycle of length $2n$ in $G_2[W_1', W_2']$.

\subsubsection{Constructing paths containing each vertex of $\Bad$}

For every vertex $x\in \Bad$, we find a four-edge path $P(x)$ in $G_2$, which contains $x$, but begins and ends at a vertex of $\Good$. We construct these paths one at a time; for each vertex $x$, we must keep in mind that in each of $W_1'$ and $W_2'$, up to $2\sqrt{\delta}n$ vertices may have been used for previously chosen paths. 

This is not always possible; when it is not, we find a cycle of length $2n$ in another way.

\begin{lemma}\label{bad-or-cycle}
One of the following holds:
\begin{enumerate}
\item $G_2$ contains a collection $\{P(x) : x \in \Bad\}$ of vertex-disjoint paths of length $4$, such that for all $x \in \Bad$, $P(x)$ begins and ends at a vertex of $\Good$, and also contains $x$ and two vertices in $W_1'$.
\item $G_2$ contains a cycle of length $2n$.
\end{enumerate}
\end{lemma}
\begin{proof}
We attempt to find the collection of vertex-disjoint paths, one vertex of $\Bad$ at a time.

By Property~\ref{w1w2-degree} at the beginning of this section, even if $x \in \Bad$, we have $d(x, W_1') \ge (\frac14 - \delta)n$ and $d_2(x, W_1') \ge \frac14 d(x, W_1') - \delta n$, so $d_2(x, W_1') \ge (\frac1{16} -\frac54 \delta) n$. There is a part $V_i$ with $d_2(x, W_1' \cap V_i) \ge (\frac{1}{64} - \frac{5}{16}\delta) n$.

First we consider the first case of Property~\ref{w1w2-Gcases}. That is, suppose $N - n_1 > 2n-1$; then we have $|V_i| = n_i \le n_1 \le n$, so $|W_2' \cap V_i| \le (\frac{63}{64} + \frac{5}{16}\delta)n$. But $|W_2'| \ge 2n$ in total, so there must be another part $V_j$ with $|W_2' \cap V_j| \ge (\frac{1}{64} - \frac{5}{16}\delta) n$. We can choose two vertices $v, w\in V_j$ to use as the endpoints of $P(x)$: ruling out the vertices of $V_j \cap \Bad$ (at most $\sqrt{\delta}n$) and previously used vertices of $W_2'$ in $V_j$ (at most $2\sqrt{\delta}n$) we still have a number of choices linear in $n$.

Now we know not just the center vertex $x$ of the path $P(x)$ but also its two endpoints $v$ and $w$. To complete $P(x)$, we must find a common neighbor of $v$ and $x$, and another common neighbor of $w$ and $x$. This is possible, since there are at least $(\frac{1}{64} - \frac{5}{16}\delta) n$ neighbors of $x$ in $V_1' \cap V_i$; $v$ and $w$ have edges in $G_2$ to all but at most $\sqrt{\delta}n$ of them, and we exclude at most $2\sqrt{\delta}n$ more that have been already used.

We call the method above of choosing the collection $\{P(x) : x \in \Bad\}$ the \emph{greedy strategy}. As we have seen, it always works in the first case of Property~\ref{w1w2-Gcases}; it remains to see when it works in the second case. Now, we assume that $G$ is tripartite, $n_1 = n_2 \ge n$, and $N - n_1 = N - n_2 = 2n-1$.

The greedy strategy continues to work if we can always choose the part $V_j$ from which to pick the endpoints of $P(x)$. For this choice to always be possible, it is enough that at least two parts of $G$ contain $3\sqrt{\delta}n$ vertices of $W_2'$: both of them will have vertices outside $\Bad$ not previously chosen for any path, and one of them will not be the same as $V_i$.

If this does not occur, then one part $V_a$ of $G$ contains all but $6\sqrt{\delta}n$ vertices of $W_2'$, and each of the other two parts contains fewer than $3\sqrt{\delta}n$ vertices of $W_2'$. If $V_a$ contains fewer than $\frac{1}{20}n$ vertices of $W_1'$, then the greedy strategy still works: for any $x \in \Bad$, we have $d_2(x, W_1') \ge (\frac1{16} -\frac54 \delta) n > |V_a \cap W_1'| + 2\sqrt{\delta}n$, so we can always choose a part of $G$ other than $V_a$ to play the part of $V_i$. In this case, it does not matter that only $V_a$ contains many vertices of $W_2'$, because we only need to choose the endpoints of $P(x)$ from vertices in $V_a$.

The greedy strategy fails in the remaining case: when $V_a$ contains all but $6\sqrt{\delta}n$ vertices of $W_2'$ and at least $\frac{1}{20}n$ vertices of $W_1'$. Then $|V_a| > n$, so without loss of generality, $V_a = V_1$. In this case, we do not try to find the paths $P(x)$ and instead find a cycle of length $2n$ in $G_1$ or $G_2$ directly.

We have a lower bound on $n_1 = n_2 = |V_2|$: it is $|V_2 \cap W_1'| + |V_2 \cap W_2'| \ge (1 + \frac1{20} - 6\sqrt{\delta})n$. Since $|V_1 \cap W_2'| \le 3\sqrt{\delta}n$, we have $|V_1 \cap W_1'| \ge (\frac{21}{20} - 9\sqrt{\delta})n > n$.

Let $Y_1$ be a subset of exactly $n$ vertices of $V_1 \cap W_1'$, chosen to avoid the $\sqrt{\delta}n$ vertices of $V_1 \cap W_1'$ with largest degree in $G_1[W_1', W_2']$ and the $\sqrt{\delta}n$ vertices of $V_1 \cap W_1'$ with largest degree in $G_2[V_1 \cap W_1', W_1' \setminus V_1]$. (This is possible since $(\frac{21}{20} - 11\sqrt{\delta})n > n$ as well.) In both cases, if a vertex $x \in Y_1$ has degree $d$ in the corresponding graph, we get at least $\sqrt{\delta} nd$ edges in either $G_1[W_1', W_2']$ or $G_2[W_1']$ by looking at the vertices we deleted; therefore $\sqrt{\delta} nd \le \delta n^2$ and $d \le \sqrt{\delta}n$.

Redistribute vertices of $V_2 \cup V_3$ into two parts $(X_1, X_2)$ as follows:
\begin{itemize}
\item All vertices of $W_1' \setminus V_1$, except the $\sqrt{\delta} n$ vertices $v$ maximizing $d_2(v, Y_1)$, are put in $X_1$. A vertex $v$ of this type is guaranteed to have $d_2(v, Y_1) \le \sqrt{\delta}n$.
\item All vertices of $W_2' \setminus V_1$, except the vertices in $\Bad$, are put in $X_2$. A vertex $v$ of this type is guaranteed to have $d_1(v, Y_1) \le \sqrt{\delta}n$.
\item The remaining vertices, of which there are at most $2\sqrt{\delta} n$, are assigned to $X_1$ or $X_2$ based on their edges to $Y_1$. If $d_1(v, Y_1) \ge \frac n2$, then $v$ is put into $X_1$; otherwise, $d_2(v, Y_1) \ge \frac n2$, and $v$ is put into $X_2$.
\end{itemize}
The sets $X_1, X_2, Y_1$ satisfy the following properties. For any $v \in X_1$, $d_1(v, Y_1) \ge \frac n2$. For any $v \in X_2$, $d_2(v, Y_1) \ge \frac n2$. For any $v \in Y_1$, $d_2(v, X_1) \le 4\sqrt{\delta}n$, since $d_2(v, W_1') \le \sqrt{\delta}n$ and $X_1$ contains at most $3\sqrt{\delta} n$ vertices of $W_2'$; similarly, for any $v \in Y_1$, $d_1(v, X_2) \le 4\sqrt{\delta} n$.

Since $|X_1| + |X_2| = |V_2 \cup V_3| = 2n-1$, either $|X_1| \ge n$ or $|X_2| \ge n$. 

If $|X_1| \ge n$, then we let $X_1'$ be a subset of exactly $n$ vertices of $X_1$, and find a cycle of length $2n$ in $H = G_1[X_1', Y_1]$ by applying Theorem~\ref{chvatal}. The hypotheses of the theorem are satisfied by the minimum degree conditions above: for $u \in X_1'$, $d_H(u) \ge \frac12 n$, and for $v \in Y_1$, $d_H(v) \ge (1 - 4\sqrt{\delta})n$. 

Similarly, if $|X_2| \ge n$, then we let $X_2'$ be a subset of exactly $n$ vertices of $X_2$, and find a cycle of length $2n$ in $H = G_2[X_2', Y_1]$ by applying Theorem~\ref{chvatal}. The argument is the same as in the previous paragraph.
\end{proof}

\subsubsection{Finding a cycle using Theorem~\ref{lasvergnas}}

Applying Lemma~\ref{bad-or-cycle}, each of the $\sqrt{\delta} n$ vertices $x\in \Bad$ is the center of a length-$4$ path $P(x)$. Let $A$ be the $2\sqrt{\delta} n$ vertices of $W_1'$ in these paths and $B$ be the $3\sqrt{\delta} n$ vertices of $W_2'$ in these paths (including the vertices in $\Bad$). Additionally, let $C$ be the set of $\sqrt{\delta}n$ vertices of $W_1' \setminus A$ with the most edges to $W_2'$ in $G_1$; by Lemma~\ref{lemma:good-vertices}, every $x \in W_1' \setminus (A\cup C)$ satisfies $d_1(x, W_2') \le \sqrt{\delta}n$.

Next, we will construct a bipartite graph $H$ by choosing subsets $W_1'' \subseteq W_1' \setminus (A\cup C)$ of size $n - 2\sqrt{\delta} n$, and $W_2'' \subseteq W_2' \setminus B$ of size $n-3\sqrt{\delta} n$; the edges of $H$ are the edges of $G_2[W_1'' \cup A, W_2'' \cup B]$, except that we artificially join every internal vertex of every path $P(x)$ to every vertex on the other side of $H$. We will apply Theorem~\ref{lasvergnas} to find a Hamiltonian cycle in $H$ containing all $q = 4\sqrt{\delta} n$ edges belonging to the paths $P(x)$, after choosing $W_1''$ and $W_2''$ to make sure that the hypotheses of this theorem hold.

In terms of our future choice of $(W_1'', W_2'')$, let $n_{i,j} = |V_i \cap W_j''|$. If $u \in V_i \cap W_1''$, then the degree of $u$ in $H$ is at least $n - n_{i,2} - \sqrt{\delta}n$: $u$ has at most $\sqrt{\delta}n$ edges to $W_2''$ that are in $G_1$, not $G_2$, and its degree is further reduced by the $n_{i,2}$ vertices of $W_2''$ that are also in $V_i$. Similarly, if $v \in V_i \cap W_2''$, then the degree of $v$ in $H$ is at least $n - n_{i,1} - \sqrt{\delta}n$.

Let $n_{*,1} \ge n_{**,1}$ be the two largest values of $n_{i,1}$ and let $n_{*,2} \ge n_{**,2}$ be the two largest values of $n_{i,2}$. As in the statement of Theorem~\ref{lasvergnas} let $u_1, u_2, \ldots, u_n$ be the vertices of $W_1'' \cup A$ and let $v_1, v_2, \ldots, v_n$ be the vertices of $W_2'' \cup B$, ordered by degree in $H$.

We begin with a lemma showing that some choices of $(W_1'', W_2'')$ are guaranteed to satisfy the conditions of Theorem~\ref{lasvergnas}:

\begin{lemma}\label{lemma:lv-conditions}
Theorem~\ref{lasvergnas} can be applied, letting us find a cycle of length $2n$ in $H$, if we can choose $W_1''$ and $W_2''$ to satisfy the following two conditions:
\begin{enumerate}
\item For each $i$, either $n_{i,1} + n_{i,2} \le n - 10\sqrt{\delta} n$, or $n_{i,1} = 0$.

\item For either $j=1$ or $j=2$, at most one value of $n_{i,j}$ exceeds $(\frac12 - 10\sqrt{\delta})n$.
\end{enumerate}
\end{lemma}
\begin{proof}
Suppose that $u_i \in W_1'' \cup A$ and $d(u_i) \le i + q \le i + 4\sqrt{\delta} n$. The minimum $H$-degree of vertices in $W_1'' \cup A$ is $n - n_{*,2} - \sqrt{\delta}n$,  so we must have $i \ge n - n_{*,2} - 5\sqrt{\delta} n$. By Condition~1, at most $n - n_{*,2} - 10\sqrt{\delta} n$ vertices in $W_1''$ are in the same part as the largest part of $W_2''$; at most $2\sqrt{\delta} n$ vertices are endpoints of paths $P(x)$, so together these make up at most $n - n_{*,2} - 8\sqrt{\delta} n < i$ vertices. Therefore some of the vertices $u_1, \ldots, u_i$ are vertices of $W_1''$ in a different part, and therefore $d(u_i) \ge n - n_{**,2} - \sqrt{\delta}n$.

Similarly, suppose that $v_j \in W_2'' \cup B$ and $d(v_j) \le j + q \le j + 4\sqrt{\delta} n$. The minimum $H$-degree of vertices in $W_2'' \cup B$ is $n - n_{*,1} - \sqrt{\delta}n$, so we must have $j \ge n - n_{*,1} - 5\sqrt{\delta} n$. By Condition~1, at most $n - n_{*,1} - 10\sqrt{\delta} n + |B|$ vertices in $W_2''$ are in the same part as the largest part of $W_1''$, which is fewer than $j$. Therefore some of the vertices $v_1, \ldots, v_j$ are vertices of $W_2''$ in a different part, and therefore $d(v_j) \ge n - n_{**,1} - \sqrt{\delta}n$.

In such a case, we have $d(u_i) + d(v_j) \ge 2n - n_{**,1} - n_{**,2} - 2\sqrt{\delta}n$. We have $n_{**,1}, n_{**,2} \le \frac12 n$, and additionally by Condition~$2$, $n_{**,j} \le \frac12n - 10\sqrt{\delta} n$ for some $j$. Therefore $d(u_i) + d(v_j) \ge n + 8\sqrt{\delta}n \ge n + 4\sqrt{\delta} n + 1$, and the hypothesis of Theorem~\ref{lasvergnas} holds.
\end{proof}

It remains choose $W_1''$ and $W_2''$ so that they satisfy the conditions of Lemma~\ref{lemma:lv-conditions}, and to deal separately with cases where this is impossible.

\textbf{First, we consider the case in which all parts of $G$ have size at most $\frac54n$.} (By Property~\ref{w1w2-Gcases}, this automatically holds when $G$ has more than $3$ parts: if so, all parts of $G$ have size at most $n$.) Choose $W_2''$ arbitrarily. $W_1'$ must contain at least $N - (1+\delta)n_1 \ge N - n_1 - \delta n_1 \ge 2n-1 - 2\delta n$ vertices, of which only $2\sqrt{\delta}n$ vertices have been used by paths and $\sqrt{\delta} n$ more have been thrown away as $C$; therefore we have at least $2n-1 - 3\sqrt{\delta}n - 2\delta n$ choices for vertices in $W_1''$.

We set aside vertices of $W_1'$ which we forbid from being in $W_1''$. From each part, $V_i$, forbid either at least $|V_i| - (1 - 10\sqrt{\delta})n$ vertices, or else all vertices of $V_i \cap W_1'$, whichever is smaller. This forbids at most $(\frac14 + 10\sqrt{\delta}) n$ vertices from each part, and at most $10\sqrt{\delta} n$ vertices in the case $n_i \le n$. There are at most two parts with $n_i > n$, so we forbid at most $(\frac12 + 50\sqrt{\delta})n$ vertices. Now Condition~1 of Lemma~\ref{lemma:lv-conditions} will be satisfied no matter what: for each part $i$, we will either have $n_{i,1} + n_{i,2} \le (1 - 10\sqrt{\delta})n$, or else $n_{i,1} = 0$.

Next, we attempt to ensure that Condition~2 of Lemma~\ref{lemma:lv-conditions} holds. Call a part $V_i$ of $G$ \emph{$W_1''$-rich} if, after excluding the forbidden vertices, and vertices of $A \cup C$, there are still at least $20\sqrt{\delta}n$ vertices of $W_1'$ left in $V_i$; call it \emph{$W_1''$-poor} otherwise.

If there are at least three $W_1''$-rich parts, then we can choose $20\sqrt{\delta}n$ vertices from each of them for $W_1''$, and complete the choice of $W_1''$ arbitrarily. Condition~2 of Lemma~\ref{lemma:lv-conditions} must now hold for $j=1$: if we had $n_{*,1} \ge (\frac12 - 10\sqrt{\delta})n$ and $n_{**,1} \ge (\frac12 - 10\sqrt{\delta})n$, then together these two parts would contain all but $20\sqrt{\delta}n$ vertices of $W_1''$. This is impossible, since there is a third $W_1''$-rich part containing at least that many vertices of $W_1''$.

If there are not at least three $W_1''$-rich parts, we give up on Lemma~\ref{lemma:lv-conditions}, and satisfy the conditions of Theorem~\ref{lasvergnas} by a different strategy.

If $V_i$ is $W_1''$-poor, it must have many vertices of $W_2''$. More precisely, $V_i$ has at least $\min\{n, n_i\} - 10\sqrt{\delta}n$ vertices that we have not forbidden. Among these, there are up to $3\sqrt{\delta}n$ vertices which are in $A \cup C$, up to $3\sqrt{\delta}n$ vertices which are in $B$, and fewer than  $20\sqrt{\delta}n$ vertices that can be added to $W_1''$, so the remaining $\min\{n, n_i\} - 16\sqrt{\delta}n$ vertices must be in $W_2' \setminus B$.

Moreover, when $G$ is tripartite, $n_i \ge \frac34n-1$ for any part, so if a part is $W_1''$-poor, it contains at least $\frac34n - 16\sqrt{\delta}n-1$ vertices of $W_2' \setminus B$. When $G$ has more than three parts, at least two parts must be $W_1''$-poor; any two parts $V_i, V_j$ have $n_i + n_j > n$, so together, two $W_1''$-poor parts have at least $n - 32\sqrt{\delta}n$ vertices of $W_2'\setminus B$. In either case, there are one or two $W_1''$-poor parts which together contain at least $\frac23n$ vertices of $W_2' \setminus B$.

We change our choice of $W_2''$, if necessary, to include at least $\frac23n$ vertices from this $W_1''$-poor part or parts; otherwise, the choice is still arbitrary. Meanwhile, we choose no vertices from these parts from $W_1''$; this rules out at most $40\sqrt{\delta}n$ vertices in addition to our previous restrictions. Completing the choice of $W_1''$ arbitrarily, we are left with a pair $(W_1'', W_2'')$ that satisfies Condition~1 of Lemma~\ref{lemma:lv-conditions}, but possibly not Condition~2.

From Condition~$1$, we know that if $v_j \in W_2''$ satisfies $d(v_j) \le j+q$, we have $d(v_j) \ge n - n_{**,2} -\sqrt{\delta} n \ge \frac12 n - \sqrt{\delta}n$. Additionally, we know that for any $u_i \in W_1''$, $d(u_i) \ge \frac23n - \sqrt{\delta}n$, since there are at least $\frac23n$ vertices of $W_2''$ in a different part of $G$. Then $d(u_i) + d(v_j) \ge \frac76n - 2\sqrt{\delta} n \ge n + q + 1$, satisfying the hypothesis of Theorem~\ref{lasvergnas}.

\textbf{Next, we consider the case where $G$ has at most $3$ parts, with $N=3n-1$  and  $n_1 = n_2 > \frac54 n$.} Since $n_2\leq \frac{N}{2}=\frac{3n-1}{2}$, we have $n_3\geq (2n-1)-n_2> \frac13n$.

We begin this case by assuming that one of $W_1' \setminus (A \cup C)$ or $W_2' \setminus B$ intersects each part of $G$ in at least $20\sqrt{\delta} n$ vertices, and the other has at least $30\sqrt{\delta} n$ vertices outside each part of $G$; we will consider departures from this assumption later. This implies that for $j=1$ or $j=2$, we can choose $20\sqrt{\delta} n$ vertices from each part to add to $W_j''$, and match these by choosing $60\sqrt{\delta} n$ vertices to add to $W_{3-j}''$ with no more than $30\sqrt{\delta} n$ of these from one part. (No $V_i$ has more than $50\sqrt{\delta} n$ vertices chosen from it at this point.)

Then proceed by an iterative strategy. At each step, choose one vertex from $W_1' \setminus (A\cup C)$ not previously added to $W_1''$, and a vertex from $W_2' \setminus B$ not previously added to $W_2''$, so that these vertices are in different parts of $G$. Then add them to $W_1''$ and $W_2''$ respectively. This step is always possible when $|W_1'' \cup A|, |W_2'' \cup B| < n$: in this case, at least two parts still have unchosen vertices, since $|V_1|, |V_2| \ge \frac54n$ but fewer than $n$ vertices have been chosen. Additionally, choosing a pair of vertices, one from $W_1'$ and one from $W_2'$, is only impossible if $W_2'\setminus B$ has no more vertices, in which case $W_2''$ has reached its desired size.

Stop when $|W_2'' \cup B|=n$. When this happens, $W_1''$ still needs $\sqrt{\delta} n$ more vertices, and these can be chosen arbitrarily.

This process guarantees that Conditions 1 and 2 of Lemma~\ref{lemma:lv-conditions} hold. Before we begin iterating, we have chosen $60\sqrt{\delta} n$ vertices, but at most $50\sqrt{\delta} n$ from each part. After we begin iterating, we add at most one vertex from each part at each step. Therefore in the end, $n_{i,1} + n_{i,2} \le n - 10\sqrt{\delta} n$ for each $i$, satisfying Condition~1. Moreover, for some $j$, we added at least $20\sqrt{\delta} n$ vertices from each part to $W_j''$, ensuring that at most one value of $n_{i,j}$ can exceed $(\frac12 - 10\sqrt{\delta})n$ and satisfying Condition~2. 

Now we consider alternatives to our initial assumptions in this case. We cannot have $W_1' \setminus (A\cup C)$ have fewer than $30\sqrt{\delta} n$ vertices outside $V_i$ for any $i$, since it contains at least $2n-1 - 4\sqrt{\delta} n - 2\delta n$ vertices, and no $V_i$ is larger than $\frac53 n$. But it is possible that one of $V_1$ or $V_2$ contains all but $30\sqrt{\delta} n$ vertices of $W_2' \setminus B$; without loss of generality, it is $V_1$.

If so, add all vertices of $W_2' \setminus B$ in $V_1$ to $W_2''$, and choose the rest of $W_2''$ arbitrarily. The set $V_2 \cup V_3$  has $2n-1$ vertices, at most $30\sqrt{\delta} n + |B| = 33\sqrt{\delta} n$ of which are in $W_2'$, so we can pick all $n$ vertices of $W_1''$ from $V_2 \cup V_3$. Choose at least $10\sqrt{\delta} n$ of them from $V_3$ to satisfy Condition~1 of Lemma~\ref{lemma:lv-conditions} for $i=2$. Condition~1 also holds for $i=1$ (since $n_{i,1}=0$) and $i=3$ (since $n_3 < \frac34n$); Condition~$2$ holds for $j=2$.

Finally, we also violate the assumptions at the beginning of this case when neither $W_1' \setminus (A\cup C)$ nor $W_2'\setminus B$ have at least $20\sqrt{\delta} n$ vertices from each part of $G$. It is impossible that both of them have at most $20\sqrt{\delta} n$ vertices from $V_3$, so one of them must have at most $20\sqrt{\delta} n$ vertices from one of $V_1$ or $V_2$. 

If one of them (without loss of generality, $V_1$) contains at most $20\sqrt{\delta} n$ vertices of $W_1' \setminus (A \cup C)$, it must have at least $n$ vertices of $W_2'\setminus B$, since $|V_1| \ge \frac54n$, so choose all remaining vertices out of $W_2''$ from there. Outside $V_1$, we have at least $(2n-1 - 4\sqrt{\delta} n - 2\delta n) - 20\sqrt{\delta} n$ vertices of $W_1'\setminus (A\cup C)$, which leaves at most $24\sqrt{\delta} n + 2\delta n$ vertices we \emph{cannot} choose for $W_1''$. Choose $n$ vertices outside $V_1$ for $W_1''$, including at least $10\sqrt{\delta} n$ vertices of $V_3$. This satisfies Condition~1 for $i=1$ (since $n_{i,1} = 0$), $i=2$ (since $n_{i,2}=0$ and $n_{i,3} < n - 10\beta n$), and $i=3$ (since $n_3 < \frac34n$); Condition~$2$ holds for $j=2$.

If one of $V_1$ or $V_2$ (without loss of generality, $V_1$) contains at most $20\sqrt{\delta} n$ vertices of $W_2' \setminus B$, choose $n - 30\sqrt{\delta} n$ vertices of $W_1''$ from $V_1$ (satisfying Condition~1 for $i=1$ and Condition~2 by taking $j=1$). If $V_3$ contains at least $30\sqrt{\delta} n$ vertices of $W_1' \setminus (A\cup Z)$, take the remaining vertices of $W_1''$ from $W_3$. Otherwise, $V_3$ contains at least $60\sqrt{\delta} n$ vertices of $W_2'\setminus B$; choosing as many vertices as possible from $V_1 \cup V_3$ to add to $W_2''$, and the remaining vertices of $W_1''$ arbitrarily, we end up choosing no more than $n - 10\sqrt{\delta} n$ vertices from $V_2$. So Condition~1 holds for $i=2$ either because $n_{i,1}=0$ or because $n_{i,1}+n_{i,2} \le n - 10\sqrt{\delta} n$; Condition~1 holds for $i=3$ because $n_3 < \frac34n$.

\section{Dealing with $(\lambda,i,2)$-bad partitions when $N-n_1-n_2\geq 3$}\label{2bad}

A {\em cherry} is a path on three vertices. The {\em center} of a cherry is the vertex with degree $2$.

Suppose $N-n_1-n_2\geq 3$. By~(\ref{n-n1=2n-1})--(\ref{N4n-2}), we have two cases:

\textbf{1)} $N>3n-1$,  $s=3$,  $n_2+n_3 = 2n-1$ and  $ n_1 = n_2$ (i.e.,~(\ref{n-n1=2n-1}) holds), or

\textbf{2)} $N=3n-1$, $n_1\leq n$, $s\leq 5$,   and if $s\geq 4$, then $n_{s-1}+n_s \geq n+1$ (i.e.,~(\ref{N3n-1}) holds).





\subsection{The case when (\ref{n-n1=2n-1}) holds}


By (\ref{n-n1=2n-1}), $n_1=n_2>n$, $s=3$, and  $0<n_3 = 2n-1 - n_2 < n$.

\begin{lemma}\label{2n-1holds}
Let $G = K_{n_1, n_2, n_3}$ with $n_1 = n_2$ and $n_2 + n_3 = 2n-1$ be $2$-edge-colored with a $(\lambda, i, 2)$-bad partition. Then $G$ has a monochromatic cycle of length $2n$.
\end{lemma}

In this section, we prove Lemma~\ref{2n-1holds}, but postpone technical details of how the monochromatic cycles are constructed in each of four cases; these details are given in Claims~\ref{situation1}--\ref{situation4}.

\begin{proof}[Proof of Lemma~\ref{2n-1holds}]
Without loss of generality, let $i = 2$; we call color $1$ red and color $2$ blue. 

We begin by assuming that in the $(\lambda,2,2)$-bad partition $(V_j,U_1,U_2)$, $j=3$. Later, in Subsection~\ref{j-not-3}, we discuss the modifications to the proof when $j\ne 3$.

Since $(V_j,U_1,U_2)$  is a $2$-bad partition, we know the following conditions hold:
\begin{enumerate}
\item[(i)] $|V_3| \ge (1-\lambda)n$.
\item[(ii)] $(1-\lambda)n \le |U_1| \le (1+\lambda)n$.
\item[(iii)] $(1-\lambda)n \le |U_2| \le (1+\lambda)n$.
\item[(iv)] $E(G_2[V_3,U_1]) \le \lambda n^2$.
\item[(v)] $E(G_1[V_3, U_2]) \le \lambda n^2$.
\end{enumerate}
If a vertex $u_1$ in $U_1$ has blue degree at least $\frac{n_{3}}{2}$ to $V_{3}$ then we move $u_1$ to $U_2$. If a vertex $u_2$ in $U_2$ has red degree at least $\frac{n_{3}}{2}$ to $V_3$ then we move $u_2$ to $U_1$. Since there are at most $3 \lambda n$ vertices in $U_1$ with blue degree at least $\frac{n_3}{2}$ to $V_3$ and there are at most $3 \lambda n$ vertices in $U_2$ with red degree at least $\frac{n_3}{2}$ to $V_3$, we moved at most $3 \lambda n$ vertices out of $U_1$ and $U_2$ respectively and moved at most $3 \lambda n$ vertices into $U_1$ and $U_2$ respectively. Thus, we may assume $|U_1| \ge |U_2|$, $|U_1| = n+a_1$, $|U_2| = n+a_2$, and $a_1 \ge 0$. 

Note that (iv) and (v) change to:
\begin{enumerate}
\item[(iv)] $|E(G_2[V_3,U_1])| \le 4 \lambda n^2$,
\item[(v)] $|E(G_1[V_3,U_2])| \le 4 \lambda n^2$.
\end{enumerate}

Let $|V_3| = n - a_3$, where $ a_3 \le 10 \lambda n$. Let $B$ be the set of vertices in $V_3$ with blue degree at least $0.9n$ to $U_1$ and $|B| = b$. Let $R$ be the set of vertices in $V_3$ with blue degree at most $0.05n$ to $U_1$. By Condition~(iv), we know 
\[
	|B| \le 5\lambda n \text{ and } |R| \ge n - a_3 - 80 \lambda n.
\]
Let $C$ be a maximum collection of vertex-disjoint red cherries with center in $U_2$ and leaves in $U_1$. If there at least $m := a_3 + b$ cherries in $C$, then we use them, together with the edges between $U_1$ and $V_3$, to find a red cycle of length $2n$; this is done in Claim~\ref{situation1}.

Otherwise, we assume that $|C| \le m-1$: there are at most $m-1$ red cherries from $U_2$ to $U_1$. Every vertex in $U_2 - V(C)$ has red degree at most $2m-1$ to $U_1$, since otherwise we have a larger collection of red cherries.  

When $|U_2| = n + a_2 \ge n-b$, we can find a blue cycle using edges between $U_2$ and $V_3$, as well as enough edges between $U_1$ and $B$ to make up for the size of $U_2$ when $|U_2| < n$. This is done in Claim~\ref{situation2}. 

Otherwise, we assume that $|U_2| \le n-b-1$; in other words,
\begin{equation}\label{a2b}
a_2 \le -(b + 1).
\end{equation}
Our goal is now to use edges within $U_1$ to find a monochromatic cycle. Without loss of generality, we may assume that $|V_1 \cap U_1| \ge |V_2 \cap U_1|$. We first argue that $U_1 \cap V_2$ cannot be too small.

Earlier, we defined $|U_1| = n+a_1$, $|U_2| = n+a_2$, $|V_3| = n-a_3$. Since $|V_1| + |V_3| = |V_2| + |V_3| = 2n-1$ and $U_1 \cup U_2 = V_1 \cup V_2$, we have
\[
	2n + a_1 + a_2 = |V_1| + |V_2| = 4n-2 - 2|V_3| = 2n + 2a_3 - 2
\]
or 
\begin{equation}\label{a1a2a3}
a_1+a_2 = 2a_3 - 2.
\end{equation}
Therefore
\begin{align*}
|U_1 \cap V_2| &\ge |U_1| - |V_1| = |U_1| - \frac{|U_1| + |U_2|}{2}
	= n + a_1 - n - \frac{a_1 + a_2}{2} \\
	&= \frac{a_1-a_2}{2} 
	= a_3 - a_2 - 1
	= (b + a_3) + (-b - a_2) - 1.
\end{align*}
There are two possibilities for the vertices of $U_1 \cap V_2$:
\begin{itemize}
\item There are at least $m = b + a_3$ vertices in $U_1 \cap V_2$ which have red degree at least $0.1n$ to $U_1 \cap V_1$. In this case, we use Claim~\ref{situation3} to find a red cycle of length exactly $2n$.

\item There are at least $m' := -b - a_2$ vertices in $U_1 \cap V_2$ which have blue degree at least $|U_1 \cap V_1| - 0.1n \ge 0.4n$ to $U_1 \cap V_1$. In this case, we use Claim~\ref{situation4} to find a blue cycle of length exactly~$2n$.
\end{itemize}
One of these must hold, since $|U_1 \cap V_2| \ge m + m' - 1$, while by \eqref{a2b}, $m' = -b - a_2 \ge 1$: therefore there are either $m$ vertices for Claim~\ref{situation3} or $m'$ vertices for Claim~\ref{situation4}. In either case, we obtain a monochromatic cycle of length exactly $2n$, completing the proof.
\end{proof}

\subsubsection{The case of many cherries: $|C| \ge m$}

Recall that $C$ is a maximum collection of vertex-disjoint red cherries with centers in $U_2$ and leaves in $U_1$; $m= b - a_3$, where $b = |B|$ and $a_3 = n - |V_3|$.

\begin{claim}\label{situation1}
If $|C| \ge m$, then we have a red cycle of length exactly $2n$.
\end{claim}
\begin{proof}
We do the following steps. Let $C' \subseteq C$ be a collection of $m$ red cherries with centers in $U_2$ and leaves in $U_1$. Let $ \{u_1, \ldots, u_m\} = V(C') \cap U_2$ and $\{v_1, \ldots, v_{2m}\} = V(C') \cap U_1$ such that each $v_{2i-1}u_{i}v_{2i}$ is a cherry with center $u_i$, where $1 \le i \le m$. 

To find a cycle of length $2n$ in $G_R$ that contains the edges of $C'$, we will apply Theorem~\ref{lasvergnas} to an appropriately chosen bipartite graph.

First, create an auxiliary graph $G'_R$ by starting with $G_R$ and adding every edge between $\{u_1, \dots, u_m\}$ and $U_1$. This will help us  to satisfy the degree conditions of Theorem~\ref{lasvergnas}; however, these artificial edges will never be used by a cycle containing all the edges of $C'$, since each of $\{u_1, \dots, u_m\}$ already has degree $2$ in $C'$.

Second, let $X = (V_3 - B) \cup \{u_1, u_2, \dots, u_m\}$ (a set of $n$ vertices total) and let $Y \subseteq U_1$ be any set of size $n$ such that $\{v_1, \dots, v_{2m}\} \subseteq Y$. We check that the hypotheses of Theorem~\ref{lasvergnas} apply to $G'_R[X,Y]$.

Order vertices in $X$ and $Y$ separately by their degree from smallest to largest. Since vertices in $Y$ have red degree at least $\frac{n_3}{2}-b \ge 0.4n$ to $X$ and at most $100\lambda n \ll 0.001n$ vertices in $Y$ have blue degree at least $0.04n$ to $X$, the smallest index $k$ such that $d_R(y_k) \le k+q$ satisfies $d_R(y_k) \ge 0.95n$. Since vertices in $X$ have blue degree at most $0.9n$ to $U_1$, they have red degree at least $n-0.9n = 0.1n > 0.09n$ to $Y$. The smallest index $j$ such that $d_R(x_j) \le j+q$ satisfies $d_R(x_j) \ge 0.09n$. By Theorem~\ref{lasvergnas} and $0.09n+0.95n > n+q+1$, we can find a Hamiltonian cycle in $G'_R[X,Y]$ of length $2n$ containing the edges of $C'$, which is a cycle of length $2n$ in $G_R$.
\end{proof}

\subsubsection{The case of large $U_2$: $|U_2| \ge n-b$}

Recall that $|U_2| = n + a_2$, $B$ is the set of vertices in $V_3$ with blue degree at least $0.9n$ to $U_1$, and $b = |B|$. 

\begin{claim}\label{situation2}
If $b \ge -a_2$ (in other words, if $|U_2| = n+a_2 \ge n-b$), then we have a blue cycle of size exactly $2n$.
\end{claim}
\begin{proof}
Let $c := |C|$; let $V(C) \cap U_2 = \{u_1,\ldots,u_{c}\}$ and $V(C) \cap U_1 = \{v_1,v_2,\ldots,v_{2c}\}$. Let $B_2$ be the collection of vertices in $V_3-B$ with red degree at most $0.1n$ to $U_2$. By Condition~(v), 
\[
	q:=|B_2| \ge n - a_3 - 40\lambda n - b.
\]
Since $2n_1 = |U_1| + |U_2| = 2n+a_1+a_2,$ we know 
\[
|U_2 \cap V_2| = n_1 - |U_1 \cap V_2| \ge n_1 - \frac{n+a_1}{2} = n+ \frac{a_1+a_2}{2}-\frac{n}{2}-\frac{a_1}{2}=\frac{n+a_2}{2}
\]
and thus 
\begin{equation}\label{u2v1}
|U_2 \cap V_1| \le n+a_2 - \frac{n+a_2}{2} =\frac{n+a_2}{2}.
\end{equation}

\textbf{Step 1:} We first include $0.8n$ vertices in $V_3$ and $0.8n$ vertices in $U_2$ (all of $V_1 \cap U_2$ and $V(C)$) by Theorem 1.

\textbf{Details:} Since $|B_2| \ge n-a_3 - 40 \lambda n - b$, we take a set $X \subseteq B_2$ such that $|X| = 0.8n$. By~\eqref{u2v1}, we can take a set $Y \subseteq U_2$ such that $V_1 \cap U_2 \subseteq Y$, $V(C) \in Y$, and $Y = 0.8n$.

Now we consider $G_B[X,Y]$ and we order vertices in $X$ and $Y$ separately by their degree from smallest to largest. Since vertices in $Y$ have blue degree at least $0.8n - \frac{n_3}{2} > 0.2n$ to $X$, the smallest index $k$ such that $d_B(y_k) \le k+1$ satisfies  $d_B(y_k) \ge 0.2n-1$. Since vertices in $X$ have red degree at most $0.1n$ to $U_2$, they have blue degree at least $0.8n-0.1n = 0.7n$ to $Y$. The smallest index $j$ such that $d_B(x_j) \le j+1$ satisfies $d_B(x_j) \ge 0.7n-1$. By Theorem~\ref{berge} and $0.7n-1+0.2n-1 > 0.8n+2$, we can find a Hamiltonian red path $P_1'$ from $x \in X$ to some vertex $y \in Y - V_1 - V(C)$ in $G_B[X,Y]$ of length $1.6n-1$.

Since $x \in X \subseteq B_2$, $$d_B(x,U_2-Y) \ge n+a_2-0.8n-0.1n > 0.05n.$$ We extend the path $P_1'$ to $P_1$ of length $1.6n$ by adding a blue edge $xy'$ such that $y' \in U_2 - Y$.

\textbf{Step 2:} Use $\min\{0,-a_2\}$ vertices in $B$ to obtain a blue path. (We can skip this step if $a_2 \ge 0$.)

\textbf{Details:} Assume $a_2<0$; since $b \ge -a_2$, let $Z:=\{z_1, \ldots, z_{|a_2|}\} \subseteq B$. 

Since $$|V_1 \cap U_1| \ge \frac{n+a_1}{2} \ge |V_2 \cap U_1|,$$ each vertex in $B$ has blue degree at least $0.9n - |V_2 \cap U_1|$ to $U_1 \cap V_1$. Therefore, $$0.9n - |V_2 \cap U_1| \ge 0.9n - (n+a_1 - |V_1 \cap U_1|) = |V_1 \cap U_1| - a_1 - 0.1n \ge \frac{3}{4} |V_1 \cap U_1|.$$

We can find for each pair $(z_i,z_{i+1})$ a common neighbor $r_i \in V_1 \cap U_1 - V(C)$ where $1 \le i \le |a_2|-1$, a blue neighbor $r_0$ of $z_1$, a blue neighbor $r_{|a_2|}$ of $z_{|a_2|}$ such that $r_0, \ldots, r_{|a_2|}$ are all distinct.

We obtain a blue path $$P_2=r_0z_1r_1\ldots z_ir_i \ldots z_{|a_2|}r_{|a_2|}$$ of length $2|a_2|$.

Since $y'$ has at most one red neighbor to $U_1-V(C)$, at least one of $\{r_0, r_{|a_2|}\}$ is a blue neighbor of $y'$. We may assume $r_{|a_2|}y'$ is blue.

\textbf{Step 3:} Include the rest of vertices in $U_2$ to $U_1$.

\textbf{Details:} We proceed differently depending on whether $a_2 \ge 0$.

\begin{itemize}
\item If $a_2<0$ then we do the following. Let $K:=(U_2 - Y - \{y'\}) \cup \{y\} = \{y, f_1, \ldots, f_{k-1}\}$. Note that $k = |K| = n+a_2 - 0.8n = 0.2n + a_2$ and $K \subseteq V_2 \cap U_2 - V(C)$. Since each vertex in $K$ has at most one red neighbor to $U_1 - V_2 - V(C) - \{r_0, r_1, \ldots, r_{|a_2|}\}$, we find for $(y,f_1)$ a blue common neighbor $h_0 \in U_1 - V_2 - V(C) - \{r_0, r_1, \ldots, r_{|a_2|}\}$ and each pair $(f_i,f_{i+1})$ a distinct blue common neighbor, $h_i$, in $U_1 - V_2 - V(C) - \{r_0, r_1, \ldots, r_{|a_2|}\}$ where $ 1 \le i \le k-2$. We obtain a blue path $$P_3 = yh_0f_1\ldots f_ih_if_{i+1} \ldots f_{k-1}$$ of size $2k - 2 = 0.4n + 2a_2 - 2$.

We may assume $f_{k-1}r_0$ is blue since $f_{k-1}$ has only one red neighbor to $V_1 \cap U_1 - V(C)$ and there are many choices when we choose $r_0$ to connect with $z_1$. 

Finally, we connect $P_2$ and $P_1$ by adding the edge $r_{|a_2|}y'$, glue the paths $P_1$ and $P_3$ at $y$, then add the edge $f_{k-1}r_0$ to complete a blue cycle of length exactly $$2|a_2|+1+1.6n+0.4n + 2a_2-2+1=2n.$$

\item If $a_2 \ge 0$ then  in the previous argument we take $K = \{y,y',f_1, \ldots, f_{k-2}\}$ of size $0.2n+1$ and find common neighbors $h_0$ for $(y,f_1)$, $h_i$ for $(f_i,f_{i+1})$ where $1 \le i \le k-3$, and $h_{k-2}$ for $(f_{k-2},y')$. 
\end{itemize}
In either case, we obtain a path 
\[
	P_3= yh_0f_1\ldots f_ih_if_{i+1} \ldots f_{k-2}h_{k-2}y'
\]
of size $2k - 2 = 0.4n$. We glue $P_1$ and $P_3$ at $y$ and $y'$ to obtain a blue cycle of length exactly $1.6n+0.4n = 2n.$
\end{proof}

\subsubsection{Handling many vertices in $U_1 \cap V_2$  incident to red edges}

We will find a red cycle. 
Note that the size of $V_2 \cap U_1$ is at least $n+a_1-n_1.$

\begin{claim}\label{situation3}
If there are at least $m=b+a_3$ vertices in $U_1 \cap V_2$ of red degree at least $0.1n$ to $U_1 \cap V_1$, then we have a red cycle of length exactly $2n$.
\end{claim}
\begin{proof}
Let $B'$ be the collection of vertices in $U_1$ with blue degree at least $0.05n$ to $V_3$. By (iv), we have $$|B'| \le 80\lambda n.$$

\textbf{Step 1:} We first find a collection of red cherries $C_3$ with center in $U_1 \cap V_2$ and leaves in $U_1 \cap V_1 - B'$ of size $b + a_3=:m$.

\textbf{Details:} Since there are at least $m$ vertices in $U_1 \cap V_2$ of red degree at least $0.1n$ to $U_1 \cap V_1$ and $0.1n - 80 \lambda n \gg 2m$, we can find a collection of red cherries $C_3$ with centers in $U_1 \cap V_2$ and leaves in $U_1 \cap V_1 - B'$ of size $m$. Let $V(C_3) \cap V_2 = \{u_1, \ldots, u_m\}$ and $V(C_3) \cap V_1 = \{v_1, \ldots, v_{2m}\}$.

Let $R \subseteq V_3$ be the collection of vertices in $V_3$ with blue degree at most $0.05n$ to $U_1$.

\textbf{Step 2:} Then by Hall's Theorem we find matching $M$ for $V(C_3) \cap U_1$ to $R$ and then find common neighbor back to connect those vertices.

\textbf{Details:} Since $\{v_2, \ldots, v_{2m}\} \cap B' = \emptyset$, each of them has red degree at least $n - a_3 -0.05n - 80 \lambda n > 0.9n$ to $R$. Thus, we can find a matching $M$ for $\{v_2, \ldots, v_{2m}\}$ such that $V(M) \cap V_3 = \{w_2, \ldots, w_{2m}\}$ and each $v_iw_i$ is a matching edge, where $2 \le i \le 2m$.

Since $V(M) \cap V_3 \subseteq R$, we can find for each pair $(w_{2i}, w_{2i+1})$ a common red neighbor $g_i \in U_1$, where $1 \le i \le m-1$.

Therefore, we obtained a path $$P_1=v_1u_1v_2w_2g_1w_3v_3u_2v_4w_4 \ldots v_{2m-1}u_mv_{2m}w_{2m}$$ of length $6m-3$.

\textbf{Step 3:} We use Theorem~\ref{berge} to get a path saturating all vertices left in $V_3 - B - V(M)$.

\textbf{Details:} 
Let $X = V_3 - B - \{w_2, \ldots, w_{2m-1} \}$ and we know $$|X| = n - a_3 - b - (2m-2)=n-3m+2.$$ Choose $Y \subseteq U_1 - \{u_1, \ldots, u_m\} - \{v_2, \ldots, v_{2m} \}-\{g_1, \ldots, g_{m-1} \}$ such that $v_1 \in Y$. By~\eqref{a1a2a3}, $$a_1 = -a_2 + 2 a_3 - 2 \ge b+1 + a_3 + a_3 - 2 = m+a_3 - 1 \ge m$$ and thus $$n+a_1 - m - (2m-1) - (m-1) \ge n-3m+2.$$ Hence we can require $|Y| = n - 3m +2$.

Now we consider $G_R[X,Y]$ and we order vertices in $X$ and $Y$ separately by their degree from smallest to largest. Since vertices in $U_1$ have red degree at least $\frac{n_3}{2}$ to $V_3$, they have red degree at least $\frac{n_3}{2} - b - (2m-1) > 0.4n$ to $X$.

By Condition~(iv), there are at most $80 \lambda n$ vertices in $U_1$ with blue degree at least $0.05n$ to $V_3$. Thus, at least $|Y| - 80 \lambda n$ vertices in $Y$ have red degree at least $|X| - 0.05n > 0.94n$ to $X$, the smallest index $k$ such that $d_R(y_k,X) \le k+1$ satisfies $d_R(y_k,X) \ge 0.94n-1$. Since vertices in $X$ have blue degree at most $0.9n$ to $U_1$, they have red degree at least $n+a_1-m-(2m-1)-(m-1)-0.9n > 0.09n$ to $Y$. The smallest index $j$ such that $d_R(x_j,Y) \le j+1$ satisfies $d_R(x_j,Y) \ge 0.09n-1$. By Theorem~\ref{berge} and $0.09n-1+0.94n-1 > n+2$, we can find a Hamiltonian red path $P_2$ from $v_1$ to $w_{2m}$ in $G_R[X,Y]$ of length $$2(n-3m+2)-1 = 2n-6m+3.$$

We glue $P_1$ and $P_2$ at $v_1$ and $w_{2m}$ to obtain a red cycle of size exactly 
\[
	6m-3+ 2n - 6m + 3 = 2n. \qedhere
\]
\end{proof}

\subsubsection{Handling many vertices in $U_1 \cap V_2$  incident to blue edges}  

In this case, there should be many disjoint blue cherries inside $U_1$, and we will find a blue cycle.

\begin{claim}\label{situation4}
If there are at least $-a_2 - b $ vertices in $V_2 \cap U_1$ of blue degree at least $|V_1 \cap U_1| - 0.1n \ge 0.4n$ to $U_1 \cap V_1$, then we find a blue cycle of length exactly $2n$.
\end{claim}
\begin{proof}
\textbf{Step 1:} We find $m' = -a_2 - b$ blue cherries from $U_1 \cap V_2$ to $U_1 \cap V_1$. Possibly avoiding bad vertices. Then find common neighbors in $V_2 \cap U_2$ to connect those cherries.

\textbf{Details:} Since vertices in $U_2 \cap V_2 - V(C)$ have red degree at most one to $V_1 \cap U_1 - V(C)$, there are at most $|V_2 \cap U_2| \le \lambda n^2$ red edges between $V_2 \cap U_2 -V(C)$ and $V_1 \cap U_1 - V(C)$. Therefore, there are at most $20 \lambda n$ vertices in $V_1 \cap U_1 - V(C)$ with red degree at least $0.05n$ to $V_2 \cap U_2 - V(C)$ and at least $|U_1 \cap V_1| - |V(C) \cap U_1| - 20 \lambda n$ vertices in $U_1 \cap V_1 - V(C)$ with blue degree at least $|V_2 \cap U_2| - |V(C)| - 0.05n > \frac{3}{4}|V_2 \cap U_2|$ to $V_2 \cap U_2 - V(C)$, we call those vertices $B_3$. 

Since there are $m'$ vertices in $V_2 \cap U_1$ of blue degree at least $|V_1 \cap U_1| - 0.1n - |V(C)| - 20 \lambda n > 0.3n$ to $B_3$, we find $m'$ blue cherries, $C_4$, with center in $V_2 \cap U_1$ and leaves in $B_3$. Let $V(C_4) \cap V_2 = \{u_1, \ldots, u_{m'}\}$ and $V(C_4) \cap V_1 = \{v_1, \ldots, v_{2m'}\}$.

We can find for each pair $(v_{2i},v_{2i+1})$ a common blue neighbor, $w_i$, in $V_2 \cap U_2 - V(C)$, where $1 \le i \le m'-1$. We also find for $v_1$ a blue neighbor $w_0$ and $v_{2m'}$ a blue neighbor $w_{m'}$ distinct from $\{w_1, \ldots, w_{m'-1}\}$ and $V(C)$.

We obtain a blue path $$P_1=w_0v_1u_1v_2w_1\ldots v_{2m'-1}u_{m'}v_{2m'}w_{m'}$$ of length $4m'$.

\textbf{Step 2:} We find for vertices in $B$ common neighbors in $V_1 \cap U_1$, avoiding vertices already used.

\textbf{Details:} Since 
\begin{equation}\label{v2}
|V_1 \cap U_1| \ge \frac{n+a_1}{2} \ge |V_2 \cap U_1|,
\end{equation}
 each vertex in $B$ has blue degree at least $0.9n-2m' - |V_2 \cap U_1|$ to $U_1 \cap V_1 - V(C) $. Therefore, $$0.9n - 2m' - |V_2 \cap U_1| \ge 0.9n - 2m' - (n+a_1 - |V_1 \cap U_1|) = |V_1 \cap U_1| - a_1 - 2m' - 0.1n $$ $$\ge \frac{3}{4} |V_1 \cap U_1|.$$
We can find for each pair $(z_i,z_{i+1})$ a common neighbor $r_i $ where $1 \le i \le b-1$, a blue neighbor $r_0$ of $z_1$, a blue neighbor $r_{b}$ of $z_{b}$ such that $r_0, \ldots, r_{b}$ are all distinct and in $V_1 \cap U_1 - V(C)$.

We obtain a blue path $$P_2=r_0z_1r_1\ldots z_ir_i \ldots z_{b}r_{b}$$ of length $2b$.

\textbf{Step 3:} Take $0.9n$ vertices in $V_3$ and $0.9n$ vertices in $U_2$ including $V_1 \cap U_2$ and $V(C)$. Use Theorem~\ref{berge} to find a path.

\textbf{Details:} Recall that $B_2$ is the collection of vertices in $V_3$ with red degree at most $0.1n$ to $U_2$ and $|B_2| \ge n-a_3 - 40\lambda n - b$. Since $|B_2| \ge n-a_3 - 40 \lambda n - b$, we take a set $X \subseteq B_2$ such that $|X| = 0.9n$. By~\eqref{v2}, $|V_1 \cap U_2| \le 0.6n$ and we can take a set $Y \subseteq U_2 - \{w_0, w_1, \ldots, w_{m'-1}\}$ such that $V_1 \cap U_2 \subseteq Y$, $V(C) \subseteq Y$, $w_{m'} \in Y$, and $Y = 0.9n$.

First we find a blue edge $v'u'$ with $v' \in X$ and $u' \notin Y$. Now we consider $G_B[X,Y]$ and we order vertices in $X$ and $Y$ separately by their degree from smallest to largest. Since vertices in $Y$ have blue degree at least $0.9n - \frac{n_3}{2} > 0.3n$ to $X$, the smallest index $k$ such that $d_B(y_k,X) \le k+1$ satisfies $d_B(y_k,X) \ge 0.3n-1$. Since vertices in $X$ have red degree at most $0.1n$ to $U_2$, they have blue degree at least $0.9n-0.1n = 0.8n$ to $Y$. The smallest index $j$ such that $d_B(x_j,Y) \le j+1$ satisfies $d_B(x_j,Y) \ge 0.8n-1$. By Theorem~\ref{berge} and $0.8n-1+0.3n-1 > 0.9n+2$, we can find a Hamiltonian blue path $P_3'$ from $w_{m'}$ to $v'$ in $G_B[X,Y]$ of length $1.8n-1$. We then extend the path $P_3'$ to $P_3$ by adding the edge $v'u'$. Thus, the path $P_3$ has length $1.8n$.

\textbf{Step 4:} Finally, the rest of vertices in $U_2 \cap V_{2}$ have large blue degree to $V_{1} \cap U_1$, and we find common neighbors to include them.

\textbf{Details:}
Let $K:=(U_2 - Y - \{w_0,w_1, \ldots, w_{m'-1}\})= \{u', f_1, \ldots, f_{k-1}\}$. Note that $k = |K| = n+a_2-0.9n-m' = 0.1n+a_2-m'$ and $K \subseteq V_{2} \cap U_2-V(C)$. Since each vertex in $K$ has at most one red neighbor to $V_1 \cap U_1 - V(C) - \{u_1, \ldots, u_{m'}\}- \{v_1, \ldots, v_{2m'}\}- \{r_0, \ldots, r_b\}$, we find for $(u',f_1)$ a distinct blue common neighbor $h_0$, each pair $(f_i,f_{i+1})$ a distinct blue common neighbor, $h_i$, in $V_1 \cap U_1 - V(C) - \{u_1, \ldots, u_{m'}\}-\{v_1, \ldots, v_{2m'}\} - \{r_0, \ldots, r_b\}$ where $ 1 \le i \le k-2$. We may assume that $r_0f_{k-1}$ is blue (since $f_{k-1}$ has at most one red neighbor to $V_1 \cap U_1$ and $z_1$ has very large blue degree to $V_1 \cap U_1$, if $r_0f_{k-1}$ was not blue then we choose $r_0$ such that $r_0f_{k-1}$ is blue).

We obtain a blue path $$P_4 = u'h_0f_1\ldots f_ih_if_{i+1} \ldots h_{k-2}f_{k-1}$$ of size $2k-2 = 0.2n + 2a_2 - 2m'-2$.

Finally, we add the edge $r_bw_0$ to connect $P_2$ and $P_1$, glue $P_1$ and $P_3$ at $w_{m'}$, glue $P_3$ and $P_4$ at $u'$, and add the edge $r_0f_{k-1}$ to complete the cycle of length 
$$
	1+4m'+2b+1.8n+0.2n+2a_2-2m'+1 = 2n.
$$
\end{proof}

\subsubsection{Changes of the proof when $j \ne 3$}\label{j-not-3}

When $j\ne 3$, essentially the same proof works, with minor modifications.

Without loss of generality, we assume $j = 1$. We use the same setup as in the case when $j=3$ but replace every place of $V_3$ by $V_1$ and $n_3$ by $n_1$.

\textbf{Case 1:} $ n_1 \ge n+b.$

Since $n_1 \ge n+b \text{ and } |U_1| \ge n,$ we take a set of vertices $X \subseteq V_1 - B$ of size $n$ and a set of vertices $Y \subseteq U_1$ of size $n$. 

Now we consider $G_R[X,Y]$ and we order vertices in $X$ and $Y$ separately by their degree from smallest to largest. Since vertices in $Y$ have red degree at least $0.5n_1$ to $X$ and there are at most $80 \lambda n$ vertices with blue degree at least $0.05n$ to $V_1$, the smallest index $k$ such that $d_R(y_k,X) \le k+1$ is at least $0.95n-1$. Since vertices in $X$ have blue degree at most $0.9n$ to $U_1$, they have red degree at least $0.1n$ to $Y$. The smallest index $j$ such that $d_R(x_j,Y) \le j+1$ is at least $0.1n-1$. By Theorem~\ref{lasvergnas} and $0.1n-1+0.95n-1 > n+1$, there is a Hamiltonian cycle in $G_R[X,Y]$ of length $2n$.

\textbf{Case 2:} $n+1 \le n_1 \le n+b - 1.$

We still assume $n_1 = n - a_3$ with $a_3 < 0$. It is included in Case 1 by replacing $n_3$ with $n_1$, $V_3$ with $V_1$, $V_1$ with $V_2$, and $V_2$ with $V_3$. Note that in this case we have $$n+a_1 + n + a_2 = 2n - 1 $$ and thus $$a_1+a_2 = -1.$$ Equation~\eqref{u2v1} changes to 
$$|U_2 \cap V_3|= n_3 - |U_1 \cap V_3| \ge 2n-1-n+a_3-\frac{n+a_1}{2}=\frac{n}{2}-1+a_3-\frac{a_1}{2}$$ and thus $$|U_2 \cap V_2| \le n+a_2 - (\frac{n}{2}-1+a_3-\frac{a_1}{2})= \frac{n}{2}+1+a_2-a_3+\frac{a_1}{2}=\frac{n}{2}-a_3-\frac{a_1}{2}.$$
When choosing between Claim~\ref{situation3} and Claim~\ref{situation4}, we still have
\[
	|U_1|-|V_2| \ge n+a_1-n+a_3 = a_1 + a_3 = -1-a_2+a_3 = (b + a_3) + (-b - a_2) - 1
\]
and therefore one of the two claims can still be applied. 

\subsection{The case when (\ref{N3n-1}) holds}

In this case, we have
\begin{equation}\label{3n-1}
n_1+n_2+\ldots+n_s = 3n-1
\end{equation}
and
\begin{equation}\label{>=2n-1}
n_2+\ldots+n_s \ge 2n-1.
\end{equation}

By (\ref{Se4}),  $s \le 5$.

\begin{lemma}\label{3n-1holds}
Let $G = K_{n_1, n_2, \dots, n_s}$ satisfying \eqref{3n-1} and \eqref{>=2n-1} be $2$-edge-colored with a $(\lambda, i, 2)$-bad partition. Then $G$ has a monochromatic cycle of length $2n$.
\end{lemma}

\begin{proof}
Without loss of generality, let $i=2$. 
 By the definition of a $2$-bad partition, there is a $j \in [s]$ such that

\begin{enumerate}
\item[(i)] $n \ge |V_j| \ge (1-\lambda)n$.
\item[(ii)] $(1-\lambda)n \le |U_1| \le (1+\lambda)n$.
\item[(iii)] $(1-\lambda)n \le |U_2| \le (1+\lambda)n$.
\item[(iv)] $E(G_2[V_j,U_1]) \le \lambda n^2$.
\item[(v)] $E(G_1[V_j, U_2]) \le \lambda n^2$.
\end{enumerate}

We begin by assuming that $s=4$. Later, in Subsection~\ref{s-not-4}, we describe the modifications to the proof when $s\ne 4$. We also assume that $j = 1$, but avoid using the fact that $n_1 \ge n_2, n_3, n_4$, so this is done without loss of generality. 

In the case when (\ref{N3n-1}) holds we have $n_i \le n$ for all $i$; we also know that $n_2 \ge n_3 \ge n_4$, $n_1 = |V_j| \ge (1-\lambda)n$, and 
\[
	|U_1|+|U_2|=n_2+n_3+n_4 = 3n-1-n_1 \le 2n+\lambda n - 1,
\]
so $n_2 \ge \frac{n_2+n_3+n_4}{3} \ge \frac{2n}{3}$.

We move vertices as we did in the previous section so that for any $u \in U_1$, $d_R(u,V_1) \ge \frac{n_1}{2}$ and for any $v \in U_2$, $d_B(v,V_1) \ge \frac{n_1}{2}$. Note that (iv) and (v) change to (iv) $|E(G_2[V_1,U_1])| \le 4 \lambda n^2$ and (v) $|E(G_1[V_1,U_2])| \le 4 \lambda n^2$.
 
Let $|U_1| = n + a_1$, $|U_2| = n + a_2$, and $|V_1| = n - a_3$. Let $B$ be the set of vertices in $V_1$ with blue degree at least $0.9n$ to $U_1$, and let $b := |B|$. By Condition~(iv), we know $b \le 5 \lambda n$.

Let $C$ be a maximum collection of vertex-disjoint red cherries with center in $U_2$ and leaves in $U_1$. If there are at least $m := a_3 + b$ cherries in $C$, then we use them, together with the edges between $U_1$ and $V_1$, to find a red cycle of length $2n$. This is done in exactly the same way as in Claim~\ref{situation1}, except with $V_1$ playing the role of $V_3$.

Otherwise, we assume that $c := |C| \le m-1$, which means every vertex in $U_2 - V(C)$ has red degree at most $2m-1$ to $U_1$. 

When $|U_2| = n + a_2 \ge n - b$, we can find a blue cycle in almost the same way as in Claim~\ref{situation2}; the updated proof is given in Claim~\ref{situation2-2}.
 
Otherwise, we may assume that $|U_2| \le n-b-1$, in which case \eqref{a2b} holds.

As before, to proceed, we want to use edges within $U_1$. Let $k$ be such that $|V_k \cap U_1|$ is maximized. This intersection is still at most $|V_k| \le n$, while $|U_1| = n+a_1$, so $|U_1 - V_k| \ge a_1$.

Since $(n+a_1) + (n+a_2) = |U_1| + |U_2| = 3n-1 - |V_1| = 2n+a_3 - 1$, we have $a_1 + a_2 = a_3 - 1$, and therefore
\[
	|U_1 - V_k| \ge a_3 - a_2 - 1 = (b + a_3) + (-a_2 - b) - 1.
\]
There are two possibilities.
\begin{itemize}
\item There are at least $m = b + a_3$ vertices in $U_1 - V_k$ of red degree at least $0.1n$ to $U_1 \cap V_k$. In this case, we will find a red cycle of length exactly $2n$ by Claim~\ref{situation3-2}.

\item There are at least $m' = -a_2 - b$ vertices in $U_1 - V_k$ of blue degree at least $|U_1 \cap V_k| - 0.1n \ge 0.2n$ to $U_1 \cap V_k$. In this case, we find a blue cycle of length exactly $2n$ by Claim~\ref{situation4-2}.
\end{itemize}
One of these must hold, since $|U_1 - V_k| \ge m + m' - 1$, while by \eqref{a2b}, $m' \ge 1$; therefore there are either $m$ vertices for Claim~\ref{situation3-2} or $m'$ vertices form Claim~\ref{situation4-2}. In either case, we obtain a monochromatic cycle of length exactly $2n$, completing the proof.
\end{proof}

\subsubsection{The case of large $U_2$: $|U_2| \ge n - b$}\label{s2} 

\begin{claim}\label{situation2-2}
If $|U_2| = n+ a_2 \ge n - b$, then we have a blue cycle of size exactly $2n$.
\end{claim}
\begin{proof}
Since $|U_2| = n+a_2 \ge n-4\lambda n$, we know that the largest among $V_2 \cap U_2$, $V_3 \cap U_2$, $V_4 \cap U_2$ has size at least $0.33n$. We assume $|V_p \cap U_2|$ is the largest and 
\begin{equation}\label{vi-2}
|V_p \cap U_2| \ge 0.33n.
\end{equation}
By~\eqref{vi-2} and $|V_p| \le n$, we have $$|V_p \cap U_1| \le 0.67n$$ and there is a $q \in \{2,3,4\}-\{p\}$ such that 
\begin{equation}\label{vj-2}
|V_q \cap U_1| \ge 0.16n.
\end{equation}

\textbf{Step 1:} We first include say $0.8n$ vertices in $V_1$ and $0.8n$ vertices in $U_2$ (all of $(V-V_p) \cap U_2$ and $V(C)$) by Theorem 1.

\textbf{Details:} The details are almost the same with the Step 2 in Claim~\ref{situation2} except every place of $n_3$ is replaced by $n_1$, every place of $V_3$ is replaced by $V_1$, $V_1$ is replaced by $(V-V_p)$.

$\bullet$ If $a_2 \ge 0$, then we do not need step 2 and go to step 3 directly.

\textbf{Step 2:} Use $|a_2|$ vertices in $B$ to obtain a blue path.

\textbf{Details:} Since $b \ge |a_2|$, let $Z:=\{z_1, \ldots, z_{|a_2|}\} \subseteq B$. 

By~\eqref{vj-2} and each vertex $v$ in $B$ has blue degree at least $0.9n \gg \frac{1}{2}|U_1|$ to $U_1$, we can find for each pair $(z_i,z_{i+1})$ a blue common neighbor $r_i \in U_1 - V(C)$ where $1 \le i \le |a_2|-1$, a blue neighbor $r_0$ of $z_1$ such that $r_0 \in V_q\cap U_1 - V(C)$, a blue neighbor $r_{|a_2|}$ of $z_{|a_2|}$ such that $r_{|a_2|} \in V_q\cap U_1 - V(C)$ and $r_0, \ldots, r_{|a_2|}$ are all distinct.

Since $y'$ has at most one red neighbor to $U_1-V(C)$, we choose $r_{|a_2|}$ to be in $V_q \cap U_1 - V(C)$ and such that $r_{|a_2|}y'$ is blue. 

We obtain a blue path $$P_2=r_0z_1r_1\ldots z_ir_i \ldots z_{|a_2|}r_{|a_2|}$$ of length $2|a_2|$.

\textbf{Step 3:}  Include the rest of vertices in $U_2$ to $U_1$ by Theorem~\ref{berge}.

\textbf{Details:} The details are almost the same with the Step 2 in Claim~\ref{situation2} except every place of $V_2$ is replaced by $V_p$.
\end{proof}

\subsubsection{Handling many vertices in $U_1 - V_k$  incident to red edges}

\begin{claim}\label{situation3-2}
If there are at least $m=b+a_3$ vertices in $(V-V_k) \cap U_1$ of red degree at least $0.1n$ to $V_k \cap U_1$, then we have a red cycle of length exactly $2n$.
\end{claim}

\begin{proof}
 Let $B'$ be the collection of vertices in $U_1$ with blue degree at least $0.05n$ to $V_1$. Since there are at most $\lambda n^2$ blue edges between $U_1$ and $V_{1}$, we have $$|B'| \le 20\lambda n.$$

\textbf{Step 1:} We first find a collection of red cherries $C_3$ with center in $U_1 \cap (V-V_k)$ and leaves in $U_1 \cap V_k - B'$ of size $m$.

\textbf{Details:} The details are almost the same except we replace every place of $V_2$ by $V-V_k$, $V_1$ by $V_k$ and $V_3$ by $V_1$.

\textbf{Step 2:} By Hall's Theorem we find matching $M$ for $V(C_3) \cap U_1$ to $R$ and then find common neighbor back to connect those vertices.

\textbf{Details:} The details are almost the same except we replace every place of $V_3$ by $V_1$ and $n_3$ by $n_1$.

\textbf{Step 3:} Use Theorem~\ref{berge} to get a path saturating all vertices left in $V_1 - B - V(M)$.

\textbf{Details:} Let $X = V_1 - B - \{w_2, \ldots, w_{2m-1} \}$ and we know $|X| = n - a_3 - b - (2m-2)=n-3m+2$. We have $a_1 = a_3 - a_2 - 1 = m - a_2 - b - 1 \ge m$, and therefore
\[
	n+a_1-m-(2m-1)-(m-1)=n+a_1-4m+2 \ge n-3m+2.
\]
We can take $Y \subseteq U_1 - \{u_1, \ldots, u_m\} - \{v_2, \ldots, v_{2m} \}-\{g_1, \ldots, g_{m-1} \}$ such that $v_1 \in Y$ and $|Y| = n - 3m + 2$.

The rest of details are almost the same except we replace every place of $V_3$ by $V_1$ and $n_3$ by $n_1$.

\end{proof}   

\subsubsection{Handling many vertices in $U_1 - V_k$  incident to blue edges}\label{s4}  

In the case when many vertices in $U_1 - V_k$ are incident to blue edges, there should be many disjoint blue cherries inside $U_1$, and we find a blue cycle.

\begin{claim}\label{situation4-2}
If there are at least $m'=-a_2-b$ vertices in $U_1 - V_k$ of red degree at least $0.1n$ to $U_1 \cap V_k$, then we have a blue cycle of length exactly $2n$.
\end{claim}
\begin{proof} 
Since $V_k \cap U_1$ is the largest among $V_2 \cap U_1$, $V_3 \cap U_1$, and $V_4 \cap U_1$, we know 
\begin{equation}\label{vku2}
|V_k \cap U_1| \ge 0.33n \text{, } |V_k \cap U_2| \le 0.67n \text{ and } |U_2 - V_k| \ge 0.32n.
\end{equation}
\textbf{Step 1:} We find $m'$ blue cherries from $U_1 \cap (V-V_k)$ to $U_1 \cap V_k$, possibly avoiding bad vertices. Then we find common neighbors in $U_2$ to connect those cherries.

\textbf{Details:} The details are almost the same until the following sentence except that we replace every place of $V_2$ by $V-V_k$ and $V_1$ by $V_k$.

For all  pairs $(v_{2i},v_{2i+1})$ we can find  distinct common blue neighbors, $w_i$, in $(V-V_k) \cap U_2 - V(C)$, where $1 \le i \le m'-1$. 

By~\eqref{vku2}, there is an $\ell \in \{2,3,4\}- \{k\}$ such that 
\begin{equation}\label{vell}
|V_{\ell} \cap U_2| \ge 0.16n.
\end{equation}
 
We also find for $v_1$ a blue neighbor $w_0 \in V_{\ell} \cap U_2$ and $v_{2m'}$ a blue neighbor $w_{m'} \in V_{\ell} \cap U_2$ distinct from $\{w_1, \ldots, w_{m'-1}\}$ and $V(C)$.

We obtain a blue path $$P_1=w_0v_1u_1v_2w_1\ldots v_{2m'-1}u_{m'}v_{2m'}w_{m'}$$ of length $4m'$.

\textbf{Step 2:} We find for vertices in $B$ common neighbors in $V_k \cap U_1$, avoiding vertices already used.

\textbf{Details:} By~\eqref{vku2} and each vertex $v$ in $B$ has red degree at most $0.1n+a_1$ to $U_1$, $v$ has at least 
\begin{equation}\label{vku1}
|V_k \cap U_1|-2m' - 0.1n - a_1 > 0.6|V_k \cap U_1 - V(C)|
\end{equation}
to $U_1 \cap V_k - V(C)$. 
We can find for each pair $(z_i,z_{i+1})$ a common neighbor $r_i $ where $1 \le i \le b-1$, a blue neighbor $r_0$ of $z_1$, a blue neighbor $r_{b}$ of $z_{b}$ such that $\{r_0, \ldots, r_{b}\} \subseteq V_k \cap U_1 - V(C)$ are all distinct and $w_0r_b$ is blue.

We obtain a blue path $$P_2=r_0z_1r_1\ldots z_ir_i \ldots z_{b}r_{b}$$ of length $2b$.

\textbf{Step 3:} Take $0.9n$ vertices in $V_1$ and $0.9n$ vertices in $U_2$ including $(V-V_{\ell}) \cap U_2$ and $V(C)$. Use Theorem~\ref{berge} to find a path.

\textbf{Details:}  The details are almost the same except we replace every place of $V_1$ by $V-V_{\ell}$, $V_3$ by $V_1$ and $n_3$ by $n_1$.

\textbf{Step 4:} Finally, the rest of vertices in $U_2 \cap V_{\ell}$ have large blue degree to $(V-V_{\ell}) \cap U_1$, and we find common neighbors to include them.

\textbf{Details:} The details are almost the same except we replace every place of $V_1$ by $V-V_{\ell}$, $V_2$ by $V_{\ell}$, $V_3$ by $V_1$ and $n_3$ by $n_1$.
\end{proof}

\subsubsection{Changes of the proof when $s\ne 4$}\label{s-not-4}

When $s\ne 4$, essentially the  proof for $s=4$ works, with minor modifications.

\textbf{Case 1:} $s=3$. Then $n_2+n_3 \ge 2n-1$ implies $n_1 \ge n_2 \ge n$ and therefore 
\[
	n_1=n_2=n \text{ and } n_3=n-1.
\]
This case is addressed in Lemma~\ref{2n-1holds}.

\textbf{Case 2:} $s=5$. 
If $j=2$, then since $n_4+n_5>n$, $n_1 \ge n_2 \ge (1-\lambda)n$ and $n_3 > \frac{n}{2}$, we have $$N = n_1 + n_2 + n_3 + n_4 + n_5 \ge 2(1-\lambda)n + \frac{3n}{2} > 3n,$$ which is a contradiction with the case. By a similar argument, $j \notin \{3,4,5\}$. Thus, we may assume $j=1$. 

The argument is almost the same as when $s=4$. We only mention their difference.

In this case, $n_4+n_5>n$ implies that 
\begin{equation}\label{n1n2n3n4}
n_1 \ge n_2 \ge n_3 \ge n_4 > \frac{n}{2}, 
\end{equation}
 thus 
\begin{equation}\label{n2n3}
n_2+n_3 = 3n-1-n_1-n_4-n_5 < n + \lambda n -1.
\end{equation}
By \eqref{n1n2n3n4} and \eqref{n2n3}, we have 
\begin{equation}\label{n2n3n4n5}
\frac{n}{2}-\lambda n \le n_5 \le n_4 \le n_3 \le n_2 \le \frac{n}{2}+\lambda n.
\end{equation}

In Section \ref{s2}, in~\eqref{vi-2} we can only guarantee  $|V_p \cap U_2| \ge 0.24n$ instead of $0.33n$. By~\eqref{n2n3n4n5}, we can find a $q \in \{2,3,4,5\}-\{p\}$ such that $|V_q \cap U_1| \ge 0.16n$.

In Section \ref{s4}, in~\eqref{vku2}  we can only guarantee the largest $|V_k \cap U_1| \ge 0.24n$. Equation~\eqref{vell} still holds with $\ell \in \{2,3,4,5\}-\{k\}$. Everything else is the same.

\section{Completion of the proof of Theorem~\ref{tC2n}}\label{k2n2n-1}

In the previous three sections, we proved Theorem~\ref{tC2n} in the cases when $N-n_1-n_2\geq 3$. By~(\ref{N4n-2}),
in the case $N-n_1-n_2\leq 2$, it is sufficient
to show that for every $2$-edge-coloring of $K_{2n,2n-1}$, there is a monochromatic cycle of length exactly $2n$.
Thus, the next lemma completes the proof of Theorem~\ref{tC2n}.

\begin{lemma}\label{t2n2n-1}
If $n$ is sufficiently large, then for every $2$-edge-coloring of $K_{2n,2n-1}$, there is a monochromatic cycle of length exactly $2n$.
\end{lemma}

\begin{proof}
Let $G = K_{2n,2n-1}$. From Section~\ref{regularity}, we know that if the reduced graph $G^r$ has a connected matching of size at least $(1+\gamma)n$, then we can find a monochromatic cycle of length exactly $2n$. Suppose $G^r$ has no connected matching of size $(1+\gamma)n$ and thus, by Section~\ref{regularity} again, $G$ has a $(\lambda, i,j)$-bad partition for some $i \in [2]$ and $j \in [2]$.

Without loss of generality, we assume $i = 1$ and discuss separately cases $j=1$ and $j=2$.

\textbf{Case 1:} $G$ has a $(\lambda, 1,1)$-bad partition.
 By the setup in Section~\ref{1bad}, we have a partition $W_1 \cup W_2$ of $V(G)$ such that

\textbf{(i):} $(1-\lambda)n \le |W_2| \le (1+\lambda)n_1 = (1+\lambda)\cdot 2n$.

\textbf{(ii):} $|E(G_1[W_1, W_2])| \le \lambda n^2$.

\textbf{(iii):} $|E(G_2[W_1])| \le \lambda n^2$.

\textbf{(iv):} For every $v \in W_1$, $d_1(v, W_1) > 0.74 d(v, W_1)$.

\textbf{(v):} For every $w \in W_2$, $d_1(w, W_1) < 0.8d(w, W_1)$, i.e., $d_2(w,W_1) \ge 0.2 d(w, W_1)$.

We know $|W_1| = N - |W_2| = 4n-1 - |W_2|$, so by Condition~\textbf{(i)}, 
\begin{equation}\label{w1}
(2 - 3\lambda)n \le |W_1| \le (3 + \lambda)n.
\end{equation}

For simplicity, let $A:=W_1 \cap V_1$, $B:=W_2 \cap V_1$, $C:= W_1 \cap V_2$ and $D:=W_2 \cap V_2$.

\textbf{Case 1.1:} $|A| \ge n$ and $|C| \ge n$. Let $X \subseteq A$ and $Y \subseteq C$ such that $|X|=|Y|=n$. By Condition~\textbf{(iv)}, we know that each $x \in X$ has $$d_1(x,Y) \ge d_1(x, W_1)-(|C|-|Y|) \ge 0.74 \cdot d(x, W_1)+n-|C|$$
$$=0.74|C|+n-|C|=n-0.26|C| \ge n - 0.26(2n-1) \ge 0.4n.$$
Similarly, we know that each $y \in Y$ has $$d_1(y, X) \ge 0.4n.$$
By Condition~\textbf{(iii)}, we know that the number of vertices in $X$ with at least $0.95n$ edges to $Y$ in $G_1$ is at least $n - 20 \lambda n$ and the number of vertices in $Y$ with at least $0.95n$ edges to $X$ in $G_1$ is at least $n - 20 \lambda n$. Therefore, if we order vertices in $X$ by their degrees in non-decreasing order, say the ordering follows from $d(x_1) \le \ldots \le d(x_n)$, then the smallest index $i$ such that $d(x_i) \le i+1$ has the property that $d(x_i) \ge 0.95n$. Similarly, if we order vertices in $Y$ by their degree in non-decreasing order, say the ordering follows from $d(y_1) \le \ldots \le d(y_n)$, then the smallest index $j$ such that $d(y_j) \le j+1$ has the property that $d(y_j) \ge 0.95n$. Since $d(x_i)+d(y_j) > n+2$, by Theorem~\ref{berge}, we know $G_1[X, Y]$ is Hamiltonian bi-connected and we can find a cycle in $G_1$ of length exactly $2n$.

\begin{remark}\label{case1.1}
The same proof shows  that there is a red cycle of length exactly $\min\{|A|, |C|\}$.
\end{remark}

\textbf{Case 1.2:} $|A| \le (1-30\lambda)n$. By equation~\eqref{w1} and $|V_1|=2n$, 
\begin{equation}\label{size1.2}
|C| \ge (1+27\lambda)n \text{ and } |B| \ge (1+30\lambda)n.
\end{equation}
By Condition~\textbf{(iii)}, there are at most $20 \lambda n$ vertices in $C$ with red degree at least $0.05n$ to $B$. Let $C'$ be the $20 \lambda n$ vertices in $C$ of largest red degree to $B$. Let $Y$ be a subset of $C-C'$ with size $n$. Similarly, let $B'$ be the $20 \lambda n$ vertices in $B$ of largest red degree to $C$ and we define $X \subseteq B - B'$ of size $n$. We show there is a blue cycle of length exactly $2n$ in $G_2[X,Y]$.

By the definitions of $X$ and $Y$, we know that $d_2(x, Y) \ge 0.95n$ for $x \in X$ and $d_2(y, X) \ge 0.95n$ for $y \in Y$. By a similar argument with the last paragraph of \textbf{Case 1.1}, we can find a blue cycle of length exactly $2n$ in $G_2[X,Y]$.

\textbf{Case 1.3:} $|C| \le (1-30\lambda)n$. We find a blue cycle by an argument similar to \textbf{Case 1.2}.

\textbf{Case 1.4:} $|A| \ge (1+30\lambda) n$ and $|D| \ge n$. By Condition~\textbf{(iii)}, there are at most $20 \lambda n$ vertices in $A$ of red degree at least $0.05n$ to $D$. Let $X'$ be the $20 \lambda n$ vertices in $A$ of largest red degree to $D$. By Condition~\textbf{(v)}, we know $d_2(y, W_1) \ge 0.2 d(y,W_1)$ for $y \in D$.


By Condition~\textbf{(ii)}, there are at most $20 \lambda n$ vertices in $D$ of red degree at least $0.05n$ to $A$. Let $R$ be the $20 \lambda n$ vertices in $D$ of largest red degree to $A$. Since $d_2(v, W_1) \ge 0.2 d(v, W_1) > 0.2n$ for each $v \in R$ and $|R| = 20 \lambda n =:m$, we can order vertices in $R$ so that $R=\{r_1, \ldots, r_{m}\}$ and find for $R$ a distinct collection of blue cherries to $A-X'$. We may assume the other ends of the cherries are $S = \{s_1, \ldots, s_{2m}\}$ so that each $s_{2i-1}r_is_{2i}$ is a cherry. Since $S \subseteq A-X'$, each $s_i$ has blue degree at least $|D|-0.05n$ to $D$ and we can find for each $(s_{2i}, s_{2i+1})$ a distinct common blue neighbor $f_i$ in $D-R$, where $1 \le i \le m-1$. and thus form a blue path $$P_1 = s_1r_1s_2f_1s_3\ldots s_{2m}$$ from $s_1$ to $s_{2m}$. We then extend the path $P_1$ by finding a blue neighbor $r_0$ of $s_1$ in $D-R$ distinct from each vertex chosen in $P_1$. Note now $P_1$ has length $4m-1$ from $r_0$ to $s_{2m}$.

Let $X \subseteq (A - X' - V(P_1)) \cup \{s_{2m}\}$ such that $s_{2m} \in X$ and $|X| = n-2m+1$. Let $Y \subseteq (D-R-V(P_1)) \cup \{r_0\}$ such that $|Y| = n-2m+1$. Since $d_2(y, X) \ge 0.9n$ for $y \in Y$ and $d_2(x, Y) \ge 0.9n$ for $x \in X$, we claim that $G_2[X,Y]$ is Hamiltonian bi-connected by a similar argument with the last paragraph of \textbf{Case 1.2}. Therefore, we can find a blue path $P_2$ of length $2n-4m+1$ from $r_0$ to $s_{2m}$.

Finally, we glue $P_1$ and $P_2$ at $r_0$ and $s_{2m}$ to complete a blue cycle of length exactly $2n$. 

\textbf{Case 1.5:} $|C| \ge (1+30\lambda) n$ and $|B| \ge n$. It is similar to \textbf{Case 1.4}.

\textbf{Case 1.6:} $|B| \ge n$ and $|D| \ge n$. 

$\bullet$ If there is no blue edge in $G[B,D]$, then $G_1[B,D]$ is a complete bipartite graph and thus we can find a red cycle of length exactly $2n$.

$\bullet$ If there is a blue matching of size $2$ in $G_2[B,D]$, say the two matching edges are $v_1v_2$ and $u_1u_2$, where $v_1, u_1 \in V_1$ and $v_2, u_2 \in V_2$, then by \textbf{Case 1.2} and \textbf{Case 1.3}, we know $|A| \ge (1-30\lambda)n$ and $|C| \ge (1-30 \lambda)n$. By Condition~\textbf{(ii)}, there are at most $20 \lambda n$ vertices in $A$ such that the red degree to $D$ is at least $0.05n$ and there are at most $20 \lambda n$ vertices in $D$ such that the red degree to $A$ is at least $0.05n$. Similarly, there are at most $20 \lambda n$ vertices in $C$ such that the red degree to $B$ is at least $0.05n$ and there are at most $20 \lambda n$ vertices in $B$ such that the red degree to $C$ is at least $0.05n$.

Let $A' \subseteq A$ be the $|A|-20 \lambda n$ vertices with the largest blue degree to $D$, $D' \subseteq D$ be the $|D|-20 \lambda n$ vertices with the largest blue degree to $A$, $C' \subseteq C$ be the $|C|-20 \lambda n$ vertices with the largest blue degree to $B$ and $B' \subseteq B$ be the $|B|-20 \lambda n$ vertices with largest blue degree to $C$. 

By Condition~\textbf{(v)}, $d_2(u_2,W_1) \ge 0.2 d(u_2, W_1) \ge 0.19n$. We find a blue neighbor $w_1 \in A'$ of $u_2$. Let $A'' \subseteq A$ such that $w_1 \in A''$ and $|A''| = \lfloor n/2 \rfloor$. Let $D'' \subseteq D'$ such that $v_2 \in D''$ and $|D''| = \lfloor n/2 \rfloor$. By $A'' \subset A'$ and $D'' \subseteq D'$, $d_2(v, A'') \ge 0.4n$ for every $v \in D''$ and $d_2(v, D'') \ge 0.4n$ for every $v \in A''$. Since $0.4n+0.4n > 0.5n+1$, we can use Theorem~\ref{berge} to find a blue path $P_1$ of length $2 \cdot (\lfloor n/2 \rfloor - 1)$ from $v_2$ to $w_1$ and then extend $P_1$ by adding $w_1u_2$. Similarly, we can find a blue path $P_2$ with vertices in $B \cup C$ from $v_1$ to $u_1$ of length exactly $2 \cdot (\lceil n/2 \rceil - 1)$.

Finally, we connect $P_1$ and $P_2$ by adding the edge $v_1v_2$ and $u_1u_2$ to form a blue cycle of length exactly $2n$.

\begin{remark}\label{twoedges}
The argument also works whenever all of $A,B,C,D$ are of size in $[n-100\lambda,n+100\lambda n]$.
\end{remark}

$\bullet$ If the size of a maximum matching in $G_2[B,D]$ is exactly one, then let $v_1v_2$ be a blue edge, and say $\{v_2\} \subseteq D$ be a smallest vertex cover in $G_2[B,D]$ (the case $\{v_1\}$ is a smallest vertex cover has a similar proof and is simpler). If we delete $v_2$, then the remaining graph is a complete bipartite graph in $G_1$. If $|D| \ge n+1$ then we can find a red cycle of length $2n$ in $G_1[B,D-\{v_2\}]$. Thus, we may assume $|D|=n$ and $|C|=n-1$.


Let $B'' \subseteq B$ such that $|B''|=n$. We find a blue cycle in $G_2[B'',C \cup \{v_2\}]$. By Condition~\textbf{(v)}, $d_2(v,C) \ge 0.2d(v,C) \ge 0.19n$ for each $v \in B''$. We also know that each vertex $v_c$ in $C \cup \{v_2\}$ can have red degree at most one to $B$ (so it has blue degree at least $n-1$ to $B''$) since otherwise with vertices in $D-\{v_2\}$ we can find a red cycle of length $2n$. Since $n-1+0.19n > n+1$, we can use Theorem~\ref{berge} to find a blue cycle of length exactly $2n$.

\textbf{Case 1.7:} $n+1 \le |A| \le (n+30\lambda n)$ and $n \le |D| \le n+30 \lambda n$. By Remark~\ref{twoedges}, the size of a maximum matching in $G_2[B,D]$ is at most one. Let $v_1v_2 \in G_2$ such that $v_1 \in B$ and $v_2 \in D$. We may also assume that $\{v_2\}$ is a minimum vertex cover of $G_2[B,D]$ (the case $\{v_1\}$ is a smallest vertex cover has a similar proof and is simpler). Let $R \subseteq A$ be the set of vertices with red degree at least $0.8n$ to $D$. By Condition~\textbf{(ii)}, we know $|R| \le 2\lambda n$. 

We first show that $|D| = n$. Assume not, i.e., $|D| \ge n+1$. Then $|D-\{v_2\}| \ge n$. 

If $|A-R| \ge n$, then we find a blue cycle of length $2n$ in $G_2[A-R,D]$. To do so, take a subset $A' \subseteq A-R$ of size $n$ and $D' \subseteq D-\{v_2\}$ of size $n$. By Condition~\textbf{(v)}, $d_2(v, A) \ge 0.2 d(v,A) \ge 0.2n$ for $v \in D$ and thus $d_2(v, A') \ge 0.19n$ for $v \in D'$. By the definition of $A'$, we know $d_2(v, D') \ge 0.2n$ for $v \in A'$. By Condition~\textbf{(ii)}, we also know there are at most $20 \lambda n$ vertices in $A'$ of red degree at least $0.05n$ to $D$ and thus if we order vertices in $A'$ and $D'$ in non-decreasing order respectively, say $A' = \{u_1, \ldots, u_n\}$ and $D' = \{w_1, \ldots, w_n\}$, then the smallest index such that $d_2(u_i) \le i+1$ has $d_2(u_i) \ge 0.95n$ and the smallest index such that $d_2(w_j) \le j+1$ has $d_2(u_j) \ge 0.19n$. Since $0.95n+0.19n > n+1$, we can use Theorem~\ref{berge} to find a blue cycle of length exactly $2n$ in $G_2[A',D']$.

If $|A-R| \le n-1$, then we find a red cycle of length exactly $2n$ in $G_1[B \cup R, D-\{v_2\}]$. To do so, note that 1) $|B \cup R| = 2n - |A-R| \ge n+1$, 2) $G_1[B,D-\{v_2\}]$ is a red complete bipartite graph and 3) each vertex in $R$ has degree at least $0.8n-1$ to $D-\{v_2\}$. We can use Theorem~\ref{berge} to find a red cycle of length exactly $2n$, since this red graph is very dense and has both parts large enough.

\begin{remark}\label{adlarge}
The proof also shows that we can find a monochromatic cycle whenever $|A| \in [n-100\lambda n,n+100 \lambda n]$ and $n+1 \le |D| \le (1+100\lambda)n$.
\end{remark}

We assume $|D|=n$ from now on. Since vertices in $R$ has red degree at least $0.8n$ to $D$, if there are at least two vertices in $R$, say $r_1$ and $r_2$, then we find a red common neighbor $w \in D$ for $r_1$ and $r_2$. Note that by Remark~\ref{case1.1}, $G_1[A,C]$ is Hamiltonian-bi connected. Therefore, we can find a red cycle of length exactly $2n$ from a path $P_1$ from $r_1$ to $r_2$ of length $2n-2$ glued with the path $P_2 = r_1wr_2$. The only case remained is $|R|\le 1$. Then we have $|A-R| \ge n$ and we find a blue cycle of length $2n$ by the same argument as in two paragraphs ahead of this paragraph.

\begin{remark}\label{diff}
Note that the last sentence of the previous paragraph shows why we need $|A| \ge n+1$.
\end{remark}

The only  uncovered  case is :\\
\textbf{Case 1.8:} $n \le  |C| \le (1+30\lambda)n$ and $(1-30\lambda)n \le |A| \le n-1$.
We define $R$ to be vertices in $C$ with red degree at least $0.8n$ to $B$. By Remark~\ref{twoedges}, we may assume that the size of a maximum matching in $G_2[B,D]$ is at most one. 

If $|C-R| \ge n$, then we find a blue cycle of length exactly $2n$ in $G_2[B,C-R]$. Thus, we may assume that 
\begin{equation}\label{c-r}
|C-R| \le n-1.
\end{equation}

$\bullet$ If there is no edge in $G_2[B,D]$, then $G_1[B,D]$ is a complete bipartite graph and we are done if $|D \cup R| \ge n$. Thus, we may assume that $|D \cup R| \le n-1$.  Since $|C-R|+|R|+|D| = 2n-1$, $|C-R| \ge n$  and we have a contradiction.

$\bullet$ If the size of a maximum matching in $G_2[B,D]$ is exactly one, say $v_1v_2$ is such a matching with $v_1 \in B$ and $v_2 \in D$, then one of $\{v_1\}$ or $\{v_2\}$ is a minimum vertex cover of $G_2[B,D]$. We may assume that $\{v_2\}$ is a minimum vertex cover of $G_2[B,D]$ and the case when $\{v_1\}$ is a minimum vertex cover share similar proof and is simpler.

Since $G_1[B,D-\{v_2\}]$ is a complete bipartite graph, we are done if $|D| \ge n+1$. Thus, we may assume $|D| \le n$. Moreover, if $|D \cup R - \{v_2\}| \ge n$ then we can find a red cycle of length $2n$ in $G_1[D \cup R - \{v_2\},B]$, hence we may assume $$|D|+|R|-1 \le n-1.$$ But we also know that $|D|+|R|+|C-R| = 2n-1$. Thus, $$|C-R| \ge n-1,$$ and by~\eqref{c-r} we know $$|C-R|=n-1\quad \text{ and }\quad |D \cup R| = n.$$ If $v_2$ has at least two red edges to $B$ then we can find a red cycle in $G_[B,D \cup R]$ by first consider the two edges incident with $v_2$. Thus, $v_2$ has at most one red edge to $B$ and thus has at least $|B|-1$ blue edges to $B$. We can find a blue cycle in $G_2[(C-R)\cup \{v_2\}, B]$.


\medskip
\textbf{Case 2:} $G$ has a $(\lambda, 1,2)$-bad partition.
 This case is covered in \textbf{Case 1} in Subsection~\ref{j-not-3} (with the same proof).
\end{proof}

\section{Proof of Theorem~\ref{tCgeq2n} on monochromatic $C_{\geq 2n}$}\label{cc2n}	

For large $n$, we need to prove the theorem for every
 $N$-vertex complete $s$-partite graph $G$ with parts $(V_1^*, V_2^*,\ldots, V_s^*)$ such that the numbers $n_i=|V_i^*|$ satisfy $n_1 \ge n_2 \ge \ldots \ge n_s$ and Conditions~(\ref{jj1}),~(\ref{jj2}),~(\ref{jf1}) and~(\ref{jm1}).

Consider a possible counterexample $G$ with a $2$-edge-coloring $f$ and the minimum $N+s$. If $N-n_1-n_2\geq 3$, then
restriction~(\ref{jm2}) does not apply, so by Theorem~\ref{tC2n}, $G$ has a monochromatic $C_{2n}$, a contradiction. If $N-n_1-n_2\leq 2$ and~(\ref{jm2}) holds, then again by Theorem~\ref{tC2n}, $G$ has a monochromatic $C_{2n}$. Hence we need to consider only the case that
$N-n_1-n_2\leq 2$,  all~(\ref{jj1}),~(\ref{jj2}),~(\ref{jf1}) and~(\ref{jm1}) hold, but~(\ref{jm2}) does not hold. In particular, $n_1\geq 2n-1$, 
 but $N\leq 4n-2$. This means $N-n_1\leq (4n-2)-(2n-1)=2n-1$, so by~(\ref{jj2}), $N=4n-2$ and $n_1=2n-1$. If $N-n_1-n_2\leq 1$, this does not satisfy~(\ref{jm1}). Thus $N-n_1-n_2= 2$, and hence $G\supseteq K_{2n-1,2n-3,2}$. Therefore,  the following lemma implies
 Theorem~\ref{tCgeq2n}.

\begin{lemma}\label{t2n-12n-32}
If $n$ is sufficiently large, then for every $2$-edge-coloring of $K_{2n-1,2n-3,2}$, there is a monochromatic cycle of length at least $2n$.
\end{lemma}	

\begin{proof}
The set-up of the proof is similar to the proof of Lemma~\ref{t2n2n-1}. We only show the differences.

Let $V_3 = \{u_1, u_2\}$.
Define $V_1' = V_1$ and $V_2' = V_2 \cup V_3$.  We first consider $G[V_1', V_2']$ and then use the fact that $V_2' = V_2 \cup V_3$. Note that we have $|V_1'| = |V_2'| = 2n-1$.

By the proof in Lemma~\ref{t2n2n-1}, we narrow the uncovered cases to 1) $|A|=n-1$ and $n \le |C| \le (1+30\lambda)n$ and 2) $n \le |A| \le (1+30 \lambda)n$ and $|C| = n-1$.

\textbf{Case 1:} $|A|=n-1$ and $n \le |C| \le (1+30\lambda)n$.

Then we know $|B| = n$ and $(1-30\lambda)n - 1 \le |D| \le n-1$. By Remark~\ref{twoedges}, we know the size of a maximum matching, $\alpha'$, in $G_2[B,D]$ is at most one. Let $R$ be the set of vertices in $C$ with at least $0.8n$ red neighbours in $B$. By Condition~\textbf{(ii)}, $|R| \le 2\lambda n$. 

\begin{claim}\label{c-r,b}
If $|C-R| \ge n$ then we find a blue cycle of length $2n$ in $G_2[B,C-R]$.
\end{claim}
\begin{proof}
We pick $C' \subseteq C-R$ of size $n$. We know 

1) By Condition~\textbf{(v)} and the definition of $R$, each vertex in $B$ has blue degree at least $0.19n$ to $C'$ and each vertex in $C$ has blue degree at least $0.2n$ to $B$,

2) By Condition~\textbf{(ii)}, all but at most $20\lambda n$ vertices in $B$ has red degree at least $0.05n$ to $C'$ and all but at most $20 \lambda n$ vertices in $C$ has red degree at most $0.05n$ to $B$,

and

3) If we order vertices in $C'$ and $B$ in non-decreasing order by their degree in $G_2[C',B]$ respectively, then the smallest index with $d(x_i)\le i+1$ and the smallest index with $d(y_j) \le j+1$ satisfies $d(x_i) \ge 0.95n$ and $d(y_j) \ge 0.95n$. 

Since $0.95n+0.95n > n+1$, we can use Theorem~\ref{berge} to show $G_2[C',B]$ is Hamiltonian bi-connected and thus we can find a cycle by fixing an edge $e$ first and then find a Hamiltonian path in $G_2[C',B]$ without $e$, which is still Hamiltonian bi-connected.
\end{proof}
 
\begin{remark}\label{usefulremark}
Similarly to Claim~\ref{c-r,b}, we can show 

1) for any two vertices $c_1\in C, a_1 \in A$, there is a red path of length $2n-3$ from $c_1$ to $a_1$ in $G_1[A,C]$.

2) for any two vertices $c_1, c_2 \in C$, there is a red path of length $2n-2$ from $c_1$ to $c_2$ in $G_1[A,C]$.

3) for any two vertices $b_1, b_2 \in B$, there is a blue path of length $2n-2$ from $b_1$ to $b_2$ in $G_2[B,C-R]$.

4) for any two vertices $c_1 \in C-R, b_1 \in B$, there is a blue path of length $2n-3$ from $c_1$ to $b_1$ in $G_2[B,C-R]$.
\end{remark}

Therefore, we may assume 
\begin{equation}\label{c-r2}
|C-R| \le n-1 \text{ and thus }|D \cup R| \ge n.
\end{equation}

If $|R| \ge 2$, say $r_1,r_2 \in R$, then we find a common neighbour $r_b \in B$ for them. By Remark~\ref{usefulremark}, we can find a red path $P_1$ of length $2n-2$ in $G_1[C,A]$ and then extend $P_1$ to a red cycle of length $2n$ by adding $r_1r_br_2$. Thus, we may assume
\begin{equation}\label{sizes}
|C-R|=n-1, \text{ }|R|=1 \text{ and }|D|=n-1.
\end{equation}

Let $R = \{r\}$. If $\alpha' = 0$, then $G_1[B,D]$ is a complete bipartite graph. We can find a red cycle of length $2n$ in $G_1[B,D \cup R]$ by first fixing two neighbours in $B$ for $r$. 


If $\alpha' = 1$, say $v_1v_2$ is a maximum matching in $G_2[B,D]$ where $v_1 \in B$ and $v_2 \in D$. If $\{v_2\}$ is a minimum vertex cover, then $v_2$ has at most one red edge to $B$ since otherwise we find a red cycle by~\eqref{sizes} in $G_1[D \cup R, B]$ by first fixing two neighbours in $B$ for $v_2$. Thus, we may assume $v_2$ has at least $|B|-1$ blue edges to $B$ and thus we can find a blue cycle in $G_2[(C-R)\cup \{v_2\}, B]$ by Remark~\ref{usefulremark}.

We may assume $\{v_1\}$ is a minimum vertex cover. Note that $v_1$ has at most one red edge to $D$ since otherwise we find a red cycle in $G_1[B,D \cup R]$ by first fixing two red neighbours for $v_1$. For the same reason, each vertex in $A$ has at most one red edge to $D$. We use vertices in $V_3$ to find a monochromatic cycle.

If there is a red edge from $D$ to $C-R$, say $u_1y_1$ with $u_1 \in D$ and $y_1 \in C$, then we find a red cycle of length at least $2n$. To do so, by Remark~\ref{usefulremark}, we first find a red path $P_1$ from $y_1$ to $r$ of length $2n-2$ in $G_1[A,C]$. Then we find for $r$ and $u_1$ a red common neighbour in $B-\{v_1\}$, say $r_b$. Finally, we extend $P_1$ to a red cycle of length $2n+1$ by adding the red path $rr_bu_1y_1$. Since at least one of $u_1$ and $u_2$ are not in $R$, say $u_1 \notin R$, we may assume there is a blue edge $u_1y_1$ from $C-R$ to $D$ with $u_1 \in C-R$ and $y_1 \in D$. 

We find a blue cycle of length at least $2n$ by using $u_1$. To do so, by Condition~\textbf{(v)}, each vertex in $D$ has blue degree at least $0.2n-1$ to $A \cup \{v_1\}$ and each vertex in $C$ has blue degree at least $0.2n-1$ to $B$. We first fix a blue neighbour $z_1$ of $y_1$ with $z_1 \in A$ and then find a common blue neighbour, say $y_2 \in D-\{y_1\}$, for $v_1$ and $z_1$. We can find a blue path $P_1$ of length $2n-3$ from $u_1$ to $v_1$ in $G_2[C-R, B]$ by Remark~\ref{usefulremark} and then extend $P_1$ by adding the path $v_1y_2z_1y_1u_1$ to obtain a blue cycle of length $2n+1$.

\textbf{Case 2:} $n \le |A| \le (1+30 \lambda)n$ and $|C| = n-1$.
It is symmetric to \textbf{Case 1} until we use vertices in $V_3$. Thus, we may assume the maximum size of a matching in $G_2[B,D]$ is one, $v_1v_2$ is one maximum matching and $\{v_2\}$ is a minimum vertex cover and every vertex in $C \cup \{v_2\}$ has blue degree at least $|B|-1$ to $B$. Moreover, we may define $R \subseteq A$ similarly to \textbf{Case 1} and assume
\begin{equation}
|A-R| = n-1,\text{ }|R|=1\text{ and }|B|=n-1.
\end{equation}

If there is a red edge from $C$ to $D - \{v_2\}$, say $u_1y_1$ with $u_1 \in C$ and $y_1 \in D$, then we can find a red cycle of length at least $2n$. To do so, we first find a red path $P_1$ of length $2n-3$ from $u_1$ to $r$ by Remark~\ref{usefulremark}. Then we find a red neighbour $r_d$ of $r$ in $D-\{v_2,y_1\}$ and a common red neighbour $r_b$ of $r_d$ and $y_1$ in $B$. We extend the path $P_1$ to a red cycle of length $2n+1$ by adding the red path $rr_dr_by_1u_1$ to $P_1$.

Then we may assume there is a blue edge from $C$ to $D-\{v_2\}$, say $u_1y_1$ with $u_1 \in C$ and $y_1 \in D - \{v_2\}$. We first find a blue path of length $2n-2$ from $y_1$ to $v_2$ in $G_2[A-R,D]$ by Remark~\ref{usefulremark} and then find a common blue neighbour $y \in B$ for $v_2$ and $u_1$. Finally, we add the path $y_1u_1yv_2$ to $P_1$ to obtain a blue cycle of length $2n+1$.
\end{proof}

\section{Proof of Theorem~\ref{tP2n} on monochromatic $P_{2n}$}\label{p2n}

\subsection{A useful lemma}
If $G$ contains a monochromatic $C_{2n}$, then it certainly contains a monochromatic $P_{2n}$. So suppose
$G=K_{n_1,\ldots,n_s}$ does not  have a monochromatic  $C_{2n}$.
The lemma below is very helpful here and in the next section.

\begin{lemma}\label{NoC2n} Let  $s\geq 3$ and $n$ be sufficiently large. Let $n_1\geq\ldots\geq n_s$ and $N=n_1+\ldots+n_s$
satisfy~(\ref{jj1}) and~(\ref{jj2}). Suppose that for some $2$-edge-coloring $f$ of the complete $s$-partite graph
$G=K_{n_1,\ldots,n_s}$, there are no monochromatic cycles $C_{2n}$.
Then $G$ contains a monochromatic $P_{2n+1}$.
\end{lemma}

{\bf Proof.}
By  Theorem~\ref{tC2n}, if~(\ref{jj1}) and~(\ref{jj2}) hold but $G$ does not  have a monochromatic  $C_{2n}$, then~(\ref{jm2})
 fails.
 In particular, $N - n_1 - n_2 \le 2$.  Since $s\geq 3$, $N - n_1 - n_2 \ge 1$.
 We may assume $s=3$: if $s > 3$, then $N - n_1 - n_2 \le 2$ yields $s=4$ and $n_3=n_4=1$. In this case, deleting the edges between $V_3$ and $V_4$ and combining them into one part (of size $2$) only makes the case harder.

We use Condition~(\ref{jm2}) to find a monochromatic $C_{2n}$ only in the nearly-bipartite subcase of Section~\ref{1bad}: in Subsection~\ref{subsection:nearly-bipartite}. Therefore, if there is no monochromatic $C_{2n}$, but~(\ref{jj1}) and~(\ref{jj2}) hold,
 we have a graph $G$ that falls under this subcase.

In this case, we have found disjoint subsets $X_{11}, X_{12} \subseteq V_1$ and $X_{21}, X_{22} \subseteq V_2$ with
 $|X_{11}| = |X_{21}| = |X_{12}| = |X_{22}| = \frac n2 + 10$ satisfying the following property: if $H$ is any of the graphs $G_1[X_{11}, X_{21}]$, $G_1[X_{12}, X_{22}]$, $G_2[X_{12}, X_{21}]$, or $G_2[X_{11}, X_{22}]$, then given any vertices $v, w$ in $H$, we can find a $(v,w)$-path in $H$ on $m$ vertices, provided that $n-10 \le m \le n+10$ and that the parity of $m$ is correct.  

Now let $x \in V_3$ be an arbitrary vertex (since we know that $1 \le n_3 \le 2$). Without loss of generality, we may assume that $x$ has an edge in $G_1$ to $X_{11}$. If $x$ also has an edge in $G_1$ to $X_{12} \cup X_{22}$, then we obtain a long path in $G_1$ as follows:
\begin{itemize}
\item Let $P_1$ be a path in $G_1[X_{11}, X_{21}]$ of length at least $n$ starting from a neighbor of $x$ in $X_{11}$.
\item Let $P_2$ be a path in $G_1[X_{12}, X_{22}]$ of length at least $n$ starting from a neighbor of $x$.
\item Use $x$ to join $P_1$ and $P_2$ into a path.
\end{itemize}
Otherwise, all edges of $x$ to $X_{12} \cup X_{22}$ are in $G_2$; in particular, $x$ has a neighbor in $G_2$ in both $X_{12}$ and $X_{22}$. We obtain a long path in $G_2$ in a similar way:
\begin{itemize}
\item Let $P_1$ be a path in $G_2[X_{12}, X_{21}]$ of length at least $n$ starting from a neighbor of $x$ in $X_{12}$.
\item Let $P_2$ be a path in $G_2[X_{11}, X_{22}]$ of length at least $n$ starting from a neighbor of $x$ in $X_{22}$.
\item Use $x$ to join $P_1$ and $P_2$ into a path.
\end{itemize}
In either case, $G$ contains a monochromatic $P_{2n+1}$.
\qed


\subsection{Completion of the proof of Theorem~\ref{tP2n} }

As observed above, if $G$ has a monochromatic $C_{2n}$, then we are done. Otherwise, by  Theorem~\ref{tC2n} and
Lemma~\ref{NoC2n}, $G$ is bipartite. In this case,~(\ref{jj2}) yields $n_2\geq 2n-1$. Hence $n_1\geq 2n-1$,
and $G\supseteq K_{2n-1,2n-1}$. In this case,
Theorem~\ref{t21} yields the result.\qed

\section{Proof of Theorem~\ref{tP2n+1} on monochromatic $P_{2n+1}$}\label{p2n+1}

\subsection{Setup of the proof}\label{setup3}
For large $n$, we need to prove the theorem for each
  complete $s$-partite graph $G=K_{n_1,\ldots,n_s}$  such that the numbers $n_i$ satisfy $n_1 \ge n_2 \ge \ldots \ge n_s$ and the following three conditions:

\textbf{($T1'$)}  $N = n_1 + \ldots + n_s \ge 3n$;

\textbf{($T2'$)}  $N-n_1 = n_2 + \ldots + n_s \ge 2n-1$;

\textbf{($T3'$)} If $s=2$, then  $n_1\geq 2n+1$.

\medskip
For a given  large $n$, we consider a possible counterexample with the minimum $N+s$. In view of this,
it is enough to consider the lists $(n_1,\ldots,n_s)$ satisfying $(T1')$, $(T2')$ and $(T3')$ such that

(a) for each $1\leq j\leq s$, if $n_i>n_{i+1}$, then the list $(n_1,\ldots,n_{i-1},n_i-1,n_{i+1},\ldots,n_s)$ does not satisfy some of  $(T1')$, $(T2')$ and $(T3')$;

(b) if $s\geq 4$, then the list $(n_1,\ldots,n_{s-2},n_{s-1}+n_s)$ (possibly with the entries rearranged into a non-increasing order) does not satisfy some of $(T1')$, $(T2')$ and $(T3')$.

\medskip
{\bf Case 1:} $s\geq 3$ and  $N>3n$. Then $(T3')$ holds by default. If  $n_1 > n_2$, 
 then the list $(n_1-1,n_2,n_3,\ldots,n_s)$  still satisfies the conditions $(T1')$, $(T2')$ and $(T3')$, a contradiction to (a).
Hence  $n_1=n_2$. Choose the maximum $i$ such that $n_1=n_i$. If $N-n_1>2n-1$, consider 
the list $(n_1,\ldots,n_{i-1},n_i-1,n_{i+1},\ldots,n_s)$. In this case $(T1')$ and $(T2')$ still are satisfied; so by (a), $(T3')$
fails. But this means $s=3$ and $n_1=n_i=1$, so $N\leq 3$, a contradiction. Thus in this case $N-n_1=2n-1$.
Therefore, $n_1= N-(N-n_1)\geq 3n+1-(2n-1)=n+2$ and hence $n_2\geq n+2$, so $N-n_1-n_2\leq (2n-1)-(n+2)=n-3$.
Then the list $(n_1,n_1,N-2n_1)$ satisfies $(T1')$--$(T3')$.
Summarizing, we get
\begin{equation}\label{n-n1=2n-12}
\mbox{\em if $s\geq 3$ and $N>3n$, then}\quad s=3,\;  n_2+n_3 = 2n-1\quad \text{\em  and }\quad n_1 = n_2\geq n+2.
\end{equation}

{\bf Case 2:} $s\geq 3$ and  $N=3n$. Again $(T3')$ holds by default. By $(T2')$, $n_1\leq n+1$, hence $N-n_1-n_2\geq n-2$.
If $s\geq 4$ and $n_{s-1}+n_s\leq n+1$, then let $L$ be the list obtained from $(n_1,\ldots,n_s)$ by replacing the two entries $n_{s-1}$
and $n_s$ with $n_{s-1}+n_s$ and then possibly rearrange the entries into non-increasing order. By construction, $L$ satisfies
 $(T1')$--$(T3')$, a contradiction to (b). Hence $n_{s-1}+n_s\geq n+2$. If $s\geq 6$, then
 $N\geq 3(n_{s-1}+n_s)\geq 3n+6$, contradicting  $N=3n$. Thus 
\begin{equation}\label{N3n-12}
\mbox{\em if $s\geq 3$ and $N=3n$, then}\quad s\leq 5  \quad \text{\em  and if $s\geq 4$, then}\quad n_{s-1}+n_s \geq n+2.
\end{equation}

{\bf Case 3:} $s= 2$. Then by $(T3')$,  $n_1\geq 2n+1$ and by $(T2')$, $n_2\geq 2n-1$. Thus $G\supseteq K_{2n+1,2n-1}$, and we can assume that
\begin{equation}\label{N4n-22}
\mbox{\em if $s= 2$, then}\quad  G= K_{2n+1,2n-1}.
\end{equation}

As we have seen, always $s \le 5$.


\subsection{Completion of the proof}

Suppose $G$ satisfies (\ref{n-n1=2n-12})--(\ref{N4n-22}), and $f$ is a $2$-edge-coloring $G$ such that there is no monochromatic $P_{2n+1}$.

If $G$ has no monochromatic $C_{2n}$, then by Lemma~\ref{NoC2n}, $G$ is bipartite. So by~(\ref{N4n-22}),
$G= K_{2n+1,2n-1}$. But by Lemma~\ref{t2n2n-1},  $K_{2n,2n-1}\mapsto (C_{2n},C_{2n})$. Therefore, below we
assume that
the $2$-edge-coloring $f$ of $G$ is such that $G$ contains a
red cycle $C$ with  ${2n}$ vertices (i.e. $G_1$ contains $C$).

Let $V'=V(C)$ and $V''=V(G)-V'$. Similarly,
for $j=1,\ldots,s$, let $V'_j=V_j\cap C$ and $V''_j=V_j-V'_j$. If some red edge $e$  connects $V'$ with $V''$, then $C+e$ contains a red
$P_{2n+1}$, so below we assume that 
\begin{equation}\label{blu}
\mbox{\em all the edges in $G[V',V'']$ are blue, i.e.,  $G_2[V',V'']=G[V',V'']$.}
\end{equation}

{\bf Case 1:} $s=2$. Then $|V'_1|=|V'_2|=n$. By (\ref{N4n-22}), $|V''_1|=n+1$. By~(\ref{blu}), $G_2[V''_1,V'_2]=K_{n+1,n}$, but $K_{n+1,n}$ contains $P_{2n+1}$.

\medskip
{\bf Case 2:} $s\geq 3$ and $n_1\geq n$. If $V_1\supseteq V''$, then (since $|V''|\geq n$ by~(\ref{N3n-12}))
$$G_2[V'',V(G)-V_1]=G[V'',V(G)-V_1]=K_{n,N-n_1}\supseteq K_{n,2n-1} \supseteq P_{2n+1}.$$
Because $C$ is a cycle of length $2n$ and $V'_1$ is an independent set, $|V'_1|\leq n$. In particular, since $s\geq 3$,
$$\mbox{\em there are distinct $2\leq j_1,j_2\leq s$ such that there are vertices $v_1\in V'_{j_1}$ and $v_2\in V''_{j_2}$.}
$$
If $|V''_1|\geq n$, then $G_2(V''_1,V'-V'_1)$ is a complete bipartite graph with parts of size at least $n$, so it contains a path
$P$ with $2n$ vertices, starting from $v_1$. Adding to it edge $v_1v_2$, we get a blue $P_{2n+1}$.

Suppose now  $|V''_1|\leq n-1$. Then the complete bipartite graph $G_2(V''_1,V'-V'_1)$ has a path $Q_1$  with $2|V''_1|+1$ vertices starting from
$v_1$ and ending in $V'-V_1$. Also since $n_1\geq n$ and $|V''|\geq n$, the complete bipartite graph $G[V'_1,V''-V_1]$ contains
$K_{n-|V''_1|,n-|V''_1|}$ and hence contains a path $Q_2$ with $2(n-|V''_1|)$ vertices starting from $v_2$. Then connecting $Q_1$ with $Q_2$ by the edge $v_1v_2$ we create a $P_{2n+1}$.

\medskip
{\bf Case 3:} $s\geq 3$ and $n_1\leq n-1$. In this case, $N/n_1>3$, so $s\geq4$. Then (\ref{n-n1=2n-12})--(\ref{N4n-22})
imply that $N=3k$ and
$4\leq s\leq 5 $.  
 In particular,
\begin{equation}\label{ni}
N-n_i\geq 3n-(n-1)=2n+1\quad\mbox{\em for every  $1\leq i\leq s$.}
\end{equation}
Relabel $V_i$s so that $|V''_1|\geq \ldots\geq |V''_s|$. Let $s'$ be the largest $i$ such that $V''_i\neq \emptyset$.
We construct a path $Q$ with ${2n+1}$ vertices greedily in two stages.

\medskip
{\bf Stage 1:} For $i=1,\ldots,s'-1$, find a vertex $w_i\in V'-V_i-V_{i+1}$ so that all
$s'-1$ of them are distinct. We can do it, because $V''_i$ and $V''_{i+1}$ are non-empty, so 
$$|V_i'\cup V'_{i+1}|\leq (n_i-1)+(n_{i+1}-1)\leq 2n-4=|V'|-4.$$
At least four choices for each of the $s'-1\leq 4$ vertices $w_i$ allow us to  choose them all distinct.
Then we choose $w_0\in V'-V_1$ and $w_{s'}\in V'-V_{s'}$ so that all $w_0,\ldots,w_{s'}$ are distinct.

\medskip
{\bf Stage 2:} For $i=0,\ldots,s'-1$ we find a $(w_i,w_{i+1})$-path $Q_i$  such that
\\
(i) $V(Q_i)\cap V''=V''_{i+1}$;  and 
(ii) all paths $Q_0,\ldots, Q_{s'-1}$ are internally disjoint.

If we succeed, then $\bigcup_{i=0}^{s'-1}Q_i$ is a path that we are seeking.

Suppose we are constructing $Q_i$ and $V''_{i+1}=\{u_1,\ldots, u_q\}$. We start $Q_i$ by the edge $w_iu_1$.  Then on Step $j$ for $j=1,\ldots,q$, do as follows.

If $j=q$, then add edge $u_qw_{i+1}$ and finish $Q_i$. Otherwise, find a vertex $z_j\in V'-V_{ i+1}$ not yet used in any
$Q_{i'}$, then add to $Q_i$ edges $u_jz_j$ and $z_ju_{j+1}$, and then go to Step $j+1$. We can find this $z_j$ because
by~(\ref{ni}), $|V-V_i|\geq 2n+1$, at most $n-2$ of these vertices are in $V''$, and at most $n$ vertices of all paths $Q_{i'}$ are already chosen
 in $V'$. Since we always can choose $z_j$, our greedy procedure constructs $Q_i$, and all $Q_i$ together form the promised path $Q$.\qed

\bigskip
{\bf Acknowledgment.} We thank Louis DeBiasio for helpful discussions.

\end{document}